\documentclass{amsart}

\usepackage{amssymb, amsmath,mathtools,mathabx}
\usepackage{mathrsfs}
\usepackage{amscd}
\usepackage{verbatim}
\usepackage{stmaryrd}
\usepackage[dvipsnames]{xcolor}
\usepackage{a4wide}
\usepackage{thmtools,thm-restate}

\usepackage{enumerate}

\newcounter{nameOfYourChoice}

\usepackage[colorlinks,linkcolor={blue},citecolor={blue},urlcolor={purple},]{hyperref}

\usepackage[colorinlistoftodos,prependcaption,textsize=tiny]{todonotes}

\usepackage{tikz-cd}
\usetikzlibrary{arrows,positioning}
\tikzset{>=latex,shift left/.style ={commutative diagrams/shift left={#1}},
  shift right/.style={commutative diagrams/shift right={#1}}}

\usepackage{tabularx}
\newcolumntype{L}{>{\arraybackslash}X}
\usepackage{multirow}

\theoremstyle{plain}
\newtheorem{theorem}{Theorem}[section]
\theoremstyle{remark}
\newtheorem{remark}[theorem]{Remark}
\newtheorem{example}[theorem]{Example}
\theoremstyle{plain}
\newtheorem{corollary}[theorem]{Corollary}
\newtheorem{lemma}[theorem]{Lemma}

\newtheorem{proposition}[theorem]{Proposition}
\newtheorem{definition}[theorem]{Definition}

\newtheorem{assumption}[theorem]{Assumption}

\numberwithin{equation}{section}

\def\N{{\mathbb N}}

\def\R{{\mathbb R}}

\def\T{{\mathbb T}}

\newcommand{\E}{{\mathbb E}}
\renewcommand{\P}{{\mathbb P}}
\newcommand{\F}{{\mathscr F}}

\newcommand{\g}{\gamma}

\newcommand{\om}{\omega}
\renewcommand{\O}{\Omega}

\renewcommand{\a}{\kappa}

\newcommand{\loc}{{\rm loc}}

\newcommand{\Tor}{\mathbb{T}}
\newcommand{\Dom}{\mathcal{O}}

\newcommand{\A}{\mathcal{A}}
\newcommand{\D}{\mathscr{D}}

\DeclareMathOperator*{\esssup}{\textup{ess\,sup}}

\newcommand{\wt}{\widetilde}

\usepackage{stmaryrd}

\newcommand{\one}{{{\bf 1}}}
\newcommand{\embed}{\hookrightarrow}
\newcommand{\s}{\delta}

\newcommand{\dps}{\displaystyle}

\renewcommand{\div}{\normalfont{\text{div}}}

\newcommand{\crit}{\mathrm{c}}

\newcommand{\reg}{\delta}

\newcommand{\norm}[1]{{\left\vert\kern-0.25ex\left\vert\kern-0.25ex\left\vert #1
    \right\vert\kern-0.25ex\right\vert\kern-0.25ex\right\vert}}

\renewcommand{\emptyset}{\varnothing}

\newcommand{\m}{\zeta}

\newcommand{\Progress}{\mathscr{P}}

\newcommand{\rnoise}{g}

\newcommand{\am}{a}
\newcommand{\bm}{b}

\newcommand{\Borel}{\mathscr{B}}

\def\XXint#1#2#3{{\setbox0=\hbox{$#1{#2#3}{\int}$ }
\vcenter{\hbox{$#2#3$ }}\kern-.6\wd0}}

\newcommand{\ellip}{\nu}

\newcommand{\Y}{\mathcal{Y}}

\newcommand{\ej}{\mathcal{E}_{j}}

\newcommand{\dd}{\mathrm{d}}
\newcommand{\HH}{\mathcal{K}}
\newcommand{\MM}{\mathcal{M}}

\allowdisplaybreaks

\begin{document}

\dedicatory{Dedicated to Philippe Cl\'ement on the occasion of his 80th birthday}

\author{Antonio Agresti}
\address{Institute of Science and Technology Austria (ISTA), Am Campus 1, 3400 Klosterneuburg, Austria}
\curraddr{Delft Institute of Applied Mathematics\\
Delft University of Technology \\ P.O. Box 5031\\ 2600 GA Delft\\The
Netherlands}
\email{agresti.agresti92@gmail.com}

\author{Mark Veraar}
\address{Delft Institute of Applied Mathematics\\
Delft University of Technology \\ P.O. Box 5031\\ 2600 GA Delft\\The
Netherlands} \email{M.C.Veraar@tudelft.nl}

\thanks{The first author has received funding from the European Research Council (ERC) under the Eu\-ropean Union’s Horizon 2020 research and innovation programme (grant agreement No 948819) \includegraphics[height=0.4cm]{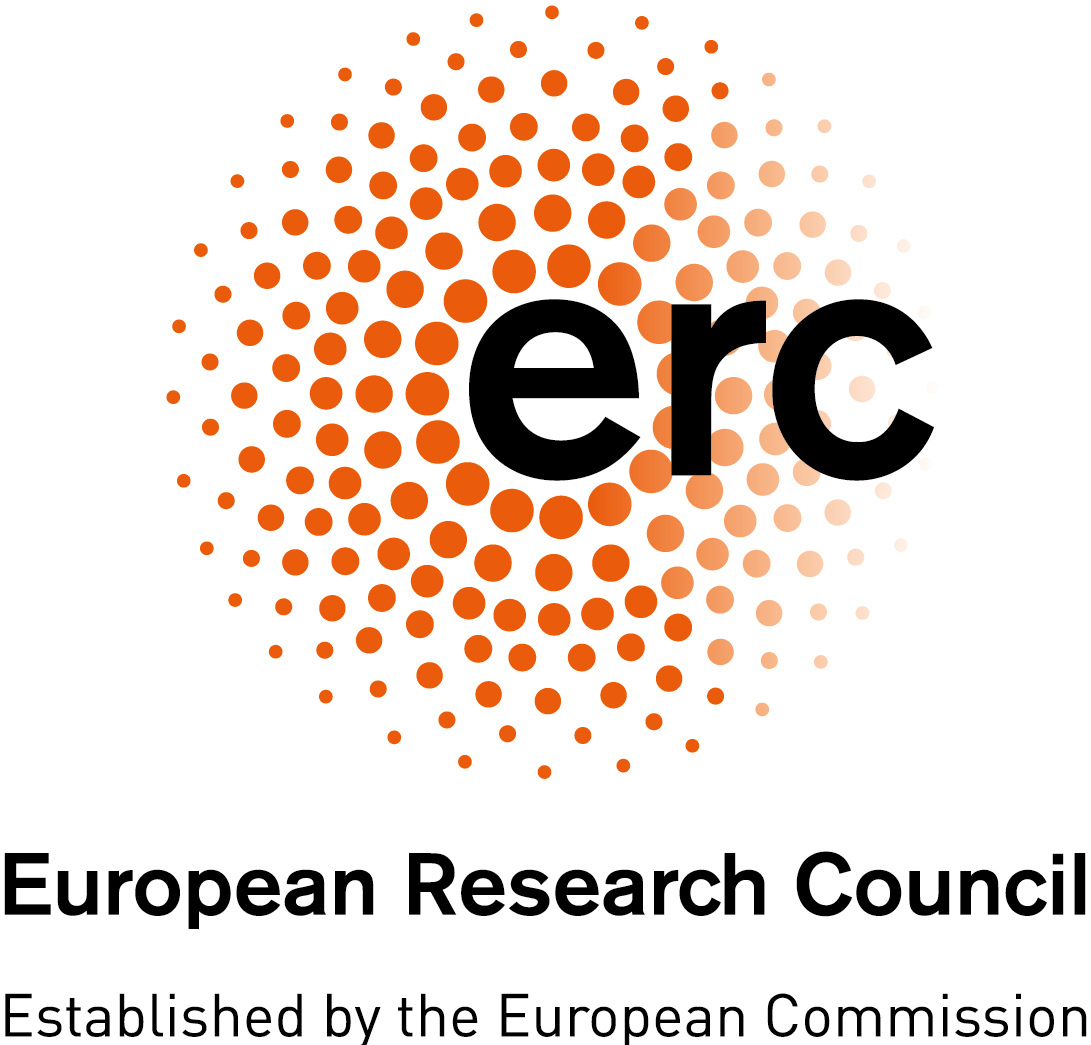}\,\includegraphics[height=0.4cm]{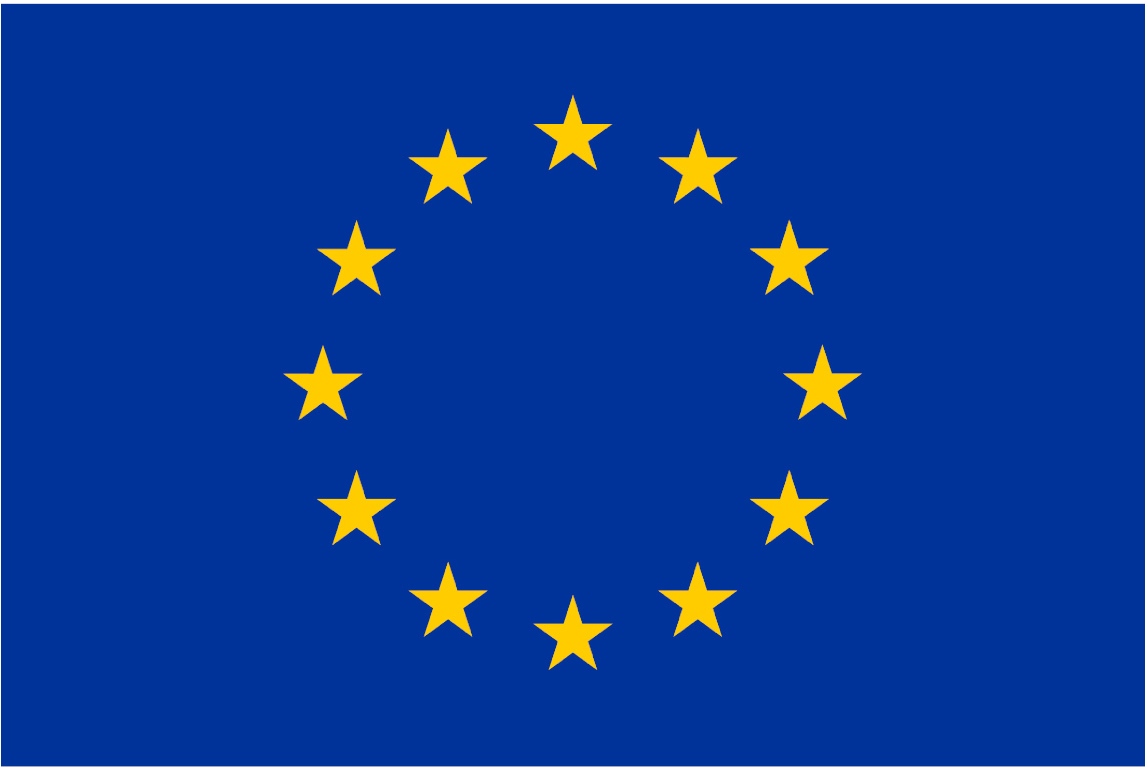}. The second author was supported by the VICI subsidy VI.C.212.027 of the Netherlands Organisation for Scientific Research (NWO)}

\date\today

\title[Global well-posedness of critical reaction-diffusion equations]{Reaction-diffusion equations with transport noise and\\ critical superlinear diffusion: Global well-posedness\\ of weakly dissipative systems}

\keywords{Reaction-diffusion equations, semilinear SPDEs, stochastic evolution equations, global well-posedness, regularity, blow-up criteria, stochastic maximal regularity, Allen-Cahn equation, coagulation dynamics, Lotka-Volterra equations, predator-prey systems, Brusselator, weak coercivity, weak dissipativity.}

\subjclass[2010]{Primary: 35K57, Secondary: 60H15, 35R60, 35A01, 35B44, 35B65, 35K90, 42B37, 58D25}

\begin{abstract}
In this paper, we investigate the global well-posedness of reaction-diffusion systems with transport noise on the $d$-dimensional torus. We show new global well-posedness results for a large class of scalar equations (e.g.\ the Allen-Cahn equation), and dissipative systems (e.g.\ equations in coagulation dynamics). Moreover, we prove global well-posedness for two weakly dissipative systems: Lotka-Volterra equations for $d\in\{1, 2, 3, 4\}$ and the Brusselator for $d\in \{1, 2, 3\}$. Many of the results are also new without transport noise.
The proofs are based on maximal regularity techniques, positivity results, and sharp blow-up criteria developed in our recent works, combined with energy estimates based on It\^o's formula and stochastic Gronwall inequalities.
Key novelties include the introduction of new $L^{\m}$-coercivity/dissipativity conditions and the development of an $L^p(L^q)$-framework for systems of reaction-diffusion equations, which are needed when treating dimensions $d\in \{2, 3\}$ in the case of cubic or higher order nonlinearities.
\end{abstract}

\maketitle



\section{Introduction}
\label{s:intro}

In this paper, we prove \emph{global} well-posedness for several systems of stochastic reaction-diffusion equations with transport noise on the torus $\T^d$. The systems are of the form:
\begin{equation}
\label{eq:reaction_diffusion_system_intro}
\left\{
\begin{aligned}
\dd u_i -\ellip_i\Delta u_i \, \dd  t&= \Big[\div(F_i(\cdot, u)) +f_{i}(\cdot, u)\Big]\,\dd t + \sum_{n\geq 1}  \Big[(b_{n,i}\cdot \nabla) u_i+ \rnoise_{n,i}(\cdot,u) \Big]\,\dd w_t^n,\\
u_i(0)&=u_{i,0},
\end{aligned}\right.
\end{equation}
on $\Tor^d$, where $i\in \{1,\dots,\ell\}$ for some integer $\ell\geq 1$ and $\inf_i\ellip_i>0$. Moreover $u=(u_i)_{i=1}^{\ell}:[0,\infty)\times \O\times \Tor^d\to \R^{\ell}$ is the unknown process, $(w^n)_{n\geq 1}$ is a sequence of standard independent Brownian motions on a filtered probability space
and
$$
(b_{n,i}\cdot\nabla) u_i:=\sum_{j=1}^d b^j_{n,i} \partial_j u_i \ \ \ \text{(transport -- noise coefficients).}
$$
Finally, $F_i,f_i,g_{n,i}$ are given nonlinearities of polynomial growth.

Reaction-diffusion equations have been widely studied in applied mathematics due to their use in biology, economy, chemistry, and population dynamics. Standard examples include the Allen-Cahn equation, reversible chemical reactions, the Brusselator system, and the Lotka-Volterra equations. The reader is referred to \cite{P15_bio,P10_survey,R84_global} for further examples.
Stochastic perturbations can be used to model uncertainty in the determination of parameters, non-predictable forces acting on the system, and the advection of a turbulent fluid.
Stochastic analogues of reaction-diffusion equations have received much attention in recent years, and much progress has been made (see e.g.\
\cite{Agr22,Cer05,Cer09,EGK22,GY21,KN19,MaRo10,S21} and references therein). However, many open problems remain, and in particular, the extensions to the stochastic setting of many known deterministic results are still not well understood. Further references are given in Subsection \ref{ss:literature} below.

The present manuscript can be seen as a continuation of our previous work \cite{AVreaction-local}, where we studied local well-posedness and positivity of solutions to reaction-diffusion equations, by applying the new abstract setting of critical spaces developed in \cite{AV19_QSEE_1,AV19_QSEE_2}.
In this present work, we focus on \emph{global} well-posedness.
In doing so, we exploit (sharp) blow-up criteria for solutions to \eqref{eq:reaction_diffusion_system_intro}, established in \cite{AVreaction-local}. One of the main contributions of our work is to allow for stochastic perturbations of transport type, which can be seen as a simplified model for the advection of a turbulent fluid (see e.g.\ \cite{GE63_turbulence,KD86_turbulence_chemical_reactions,LW76,MC98,SH91_turbulence} for physical motivations, and  \cite[Subsection 1.3]{AVreaction-local} for the modelling). The reader is referred to \cite{DP22,FP21} for (different) situations where this fact can be made rigorous. However, let us emphasize that our results are new even in the case $b\equiv 0$ due to the presence of the critical superlinear diffusion coefficient $g$.
Transport noise was introduced by R.H.\ Kraichnan in \cite{K68,K94} in the study of passive scalars advected by turbulent flows, and has subsequently been widely used in the context of fluid dynamics. In this theory, the regularity (or irregularity) of $(b_{n,i})_{n\geq 1}$ plays an important role. More precisely, one merely assumes
\begin{equation}
\label{eq:b_smoothness_intro}
(b_{n,i})_{n\geq 1} \in C^{\gamma}(\Tor^d;\ell^2(\N;\R^d)) \ \text{ for some }\ \gamma>0.
\end{equation}
The results of the current paper hold under the assumption \eqref{eq:b_smoothness_intro}.
Here, for simplicity, we only consider periodic boundary conditions. However, we expect that the techniques we use can be extended to the case of bounded domains $\Dom\subseteq\R^d$ with suitable boundary conditions. In the latter situation, further constraints on $b_{n,i}|_{\partial\Dom}$ might be needed.

In the main part of the current work, we actually consider reaction-diffusion equations where the leading operators are of the form $\div(a_{i}\cdot\nabla u_i)$ instead of $\ellip_i \Delta u_i$, see \eqref{eq:reaction_diffusion_system}. One motivation for this is that if one considers \eqref{eq:reaction_diffusion_system_intro} with Stratonovich transport noise instead of an It\^o one, then one can reformulate it as a system of SPDEs in It\^o form with a correction term which is in divergence form. Indeed, at least formally (in the case $g_{n,i}\equiv 0$),
\begin{equation}
\label{eq:Itostratintro}
\sum_{n\geq 1}(b_{n,i} \cdot\nabla) u_i\circ \dd w^{n}_t = \div(a_{b,i}\cdot\nabla u_i)\,\dd t +
\sum_{n\geq 1}(b_{n,i} \cdot\nabla) u_i\, \dd w^{n}_t,
\end{equation}
where $a_{b,i}^{j,k}=\frac{1}{2}\sum_{n\geq 1}b^{j}_{n,i}b^{j}_{n,i}$ for all $j,k\in \{1,\dots,d\}$ provided $\div\,b_{n,i}=0$ (incompressibility).

Global well-posedness of reaction-diffusion equations has been extensively studied in the deterministic literature, and this depends on the fine structure of the nonlinearities $(f_i,F_i,g_i)$. Under solely locally Lipschitz and polynomial growth assumptions, commonly used in reaction-diffusion equations, one \emph{cannot} expect global well-posedness even in the scalar case (see e.g.\ \cite{F66,QS19} and \cite[Chapter 6]{P15_bio}). For \emph{systems} of reaction-diffusion equations, blow-up phenomena are possible even in presence of mass dissipation \cite{PS97_blow_up}. The latter is rather surprising at first sight since the mass conservation is strong enough to prevent blow-up in the case of ODEs, i.e.\ in absence of spatial variability.
The previously mentioned blow-up results show that further assumptions on the nonlinearities $(f_i,F_i,g_{i})$ are needed to obtain the global well-posedness of reaction-diffusion equations. In practice, one needs that the nonlinearities are    \emph{dissipative} in some sense so that the formation of singularities is not possible.
To the best of our knowledge, there is no standard way in the deterministic and stochastic literature to capture the dissipative effect of the nonlinearities. We will formulate several \emph{coercivity/dissipativity} conditions on the nonlinearities $(f_i,F_i,g_i)$, and we will show how one can use these conditions to obtain global well-posedness in several concrete situations.

In our study of global well-posedness of \eqref{eq:reaction_diffusion_system_intro}, $L^p(L^q)$-theory ($L^p$ in time and $L^q$ in space)  plays a key role. This theory was already used in our previous work \cite{AVreaction-local} to derive high-order regularity and positivity of solutions, which are both essential ingredients in this manuscript. The positivity of solutions is of fundamental importance for physical reasons, but we will also use it to derive energy estimates.
The need for $L^p(L^q)$-theory in the study of global well-posedness of \eqref{eq:reaction_diffusion_system_intro} can be (roughly) explained as follows. In reaction-diffusion equations, the nonlinearities in \eqref{eq:b_smoothness_intro} are of high polynomial growth, say $h\gg 1$. In the case that $h$ or $d$ are large enough (e.g.\ $(d,h)=(3,2)$ or $(d,h) = (2,3)$), then one can no longer work in an $L^2(L^2)$-setting because the nonlinearities require better Sobolev embeddings. Alternatively, one could use an $L^2(H^{s,2})$-setting. However, this has two disadvantages. Firstly, typically requires transport noise of smoothness $\geq s$ which is not satisfied in general, as $\gamma>0$ in \eqref{eq:b_smoothness_intro} is typically small.
Secondly, energy estimates in $L^2(H^{s,2})$ are hard to prove. In particular, we could not find suitable coercivity/dissipativity conditions that can be verified in applications. In contrast, in an $L^{p}(L^q)$-setting this works very well, as we will see. Finally, we mention that
from the scaling argument in \cite[Subsection 1.4]{AVreaction-local}, the right space to study well-posedness of \eqref{eq:reaction_diffusion_system_intro}, is the (critical) Lebesgue space $L^{\frac{d}{2}(h-1)}(\Tor^d;\R^{\ell})$.

In addition to the more abstract statements on global well-posedness, we have included several concrete applications as well. These include the following well-known equations:
\begin{itemize}
\item Allen-Cahn equation (phase separation) - Theorem \ref{t:allen_Cahn_intro} and \ref{t:global_well_posedness_general}.
\item Coagulation dynamics \cite{FH21_coagulation} (used in chemistry and rain formation) - Theorem \ref{thm:coagulation}.
\item Several types of Lotka-Volterra models (predator-prey systems) - Theorems \ref{thm:SymbLotka-Volterra} and \ref{t:Lotka_Volterra}.
\item Brusselator (chemical morphogenetic processes) - Theorems \ref{t:brusselator_intro}, \ref{t:brusselator_three} and \ref{t:Brusselator}.
\end{itemize}
Several of the global well-posedness results for the above equations appear to be completely new (even without transport noise), and our regularity results also seem new for each of these equations. In the next subsections, we discuss further details in the scalar setting $\ell=1$ and the system case ($\ell\geq 2$)  separately.

\subsection{The scalar case and Allen-Cahn equation}
\label{ss:scalar_AC_intro}
In the scalar case $\ell=1$ of \eqref{eq:reaction_diffusion_system_intro}, we establish global well-posedness of \eqref{eq:reaction_diffusion_system_intro} under the following $L^{\m}$-\emph{coercivity/dissipativity condition}:

There exist $\m\geq \frac{d(h-1)}{2}\vee2$ and
$M,\theta>0$ such that, for all
$u\in C^1(\mathbb{T}^d)$,
\begin{align}
\label{eq:coercivity_Lm_intro}
\int_{\T^d} |u|^{\m-2} \Big(\ellip | \nabla u|^2   + F(\cdot, u) \cdot \nabla u - \frac{u f(\cdot, u)}{\m-1} -\frac12 \sum_{n\geq 1} \big[(b_n \cdot \nabla) u + g_n(\cdot, u) \big]^2 \Big) \, \dd x
&
\\
\nonumber
\geq  \theta \int_{\T^d} |u|^{\m-2}\big(|\nabla u|^2 - M|u|^2\big)  \, \dd x -M.
&
\end{align}
Here we have omitted the subscript $i$, as we are dealing with scalar equations, and $h>1$ denotes the growth of the nonlinearity $(f_i,F_i,g_i)$, see Assumption \ref{ass:reaction_diffusion_global}.
The condition \eqref{eq:coercivity_Lm_intro} can be seen as an $L^{\m}$-variant of the usual coercivity condition for stochastic evolution equations in the variational setting, see e.g.\ \cite{AV22_variational,rockner2022well} and \cite[Chapter 4]{LR15}. The restriction $\m\geq \frac{d}{2}(h-1)$ is related to the aforementioned criticality of $L^{\frac{d}{2}(h-1)}$  for \eqref{eq:reaction_diffusion_system_intro} (see \cite[Subsection 1.4]{AVreaction-local}). The restriction $\zeta\geq 2$ is necessary, because of some nontrivial obstacles in stochastic calculus in $L^{\zeta}$ for $\zeta<2$, cf.\ \cite{NVW13}. In the applications we consider it turns out that the minimal choice $\m=\frac{d(h-1)}{2}\vee 2$ leads to the largest class of nonlinearities $f$, $F$, $g$ (with $h$ fixed) which satisfy \eqref{eq:coercivity_Lm_intro}.

In Subsection \ref{ss:example_nonlinearities} we give examples of nonlinearities satisfying \eqref{eq:coercivity_Lm_intro} and two sufficient ``pointwise'' conditions for \eqref{eq:coercivity_Lm_intro}, namely Lemmas \ref{lem:dissipationI} and \ref{lem:dissipationII}. Moreover, in many situations of interest, one knows a-priori that solutions to \eqref{eq:reaction_diffusion_system_intro} are positive. In that case, it is enough to have \eqref{eq:coercivity_Lm_intro} for $u\geq 0$. This is particularly interesting for polynomial nonlinearities
of \emph{even} order (see Example \ref{ex:strong_dissipation}) which, to the best of our knowledge, have not been considered before in the literature.

To make the introduction and results of the paper more explicit, we discuss what our results give for the Allen-Cahn equation with transport noise and superlinear diffusion, i.e.\ \eqref{eq:reaction_diffusion_system_intro} with
\begin{equation}
\label{eq:AC_intro}
\ell = 1, \ \ f(\cdot,u)= u-u^3 \ \  \text{ and }\ \  g_n(\cdot,u)= \theta_n u^2 \ \text{ where }\ \theta:=(\theta_n)_{n\geq 1}\in \ell^2.
\end{equation}

\begin{theorem}[Global well-posedness -- 3d Allen-Cahn equation]
\label{t:allen_Cahn_intro}
Assume that $\|\theta\|_{\ell^2} \leq 1$ and that \eqref{eq:b_smoothness_intro} holds. Let $f$ and $g$ be as in \eqref{eq:AC_intro} and suppose that $F= 0$.
Fix $\delta\in (1,\frac{4}{3}\wedge (1+\alpha))$, $p\in [3,\infty)$ and $\a_{\crit}:=p(1-\frac{\s}{2})-1$. Let $u_0\in
L^0_{\F_0}(\O;L^3(\Tor^3))
$.
Suppose that, for some $\ellip_0\in (0,\ellip)$,
$$
\div \,b_n=0 \text{ in }\D'(\Tor^3) \  \text{ for } n\geq 1, \ \text{ and } \ \
\sum_{n\geq 1}|  b_n(x) \cdot \xi|^2\leq 2\ellip_0 |\xi|^2\ \text{ for } x\in \Tor^3, \ \xi\in \R^3.
$$
Then there exists a unique \emph{global} $(p,\a_{\crit},\s,q)$-solution $u$ to \eqref{eq:reaction_diffusion_system_intro}
satisfying, a.s.,
\begin{align}
\nonumber
u&\in H_{\loc}^{\theta,p}([0,\infty),t^{\kappa_{\crit}} \, \dd t;H^{2-\s,q}(\Tor^3))\cap  C([0,\infty);L^{3}(\Tor^3))\  \ \forall \theta\in [0,\tfrac{1}{2}), \\
\label{eq:AC_regularization_intro}
u&\in C^{\theta_1,\theta_2}_{\rm loc}((0,\infty)\times \Tor^3) \
\ \forall \theta_1\in  [0,\tfrac{1}{2}), \  \forall  \theta_2\in (0,1+\alpha).
\end{align}
\end{theorem}

$(p,\a,\s,q)$-solutions are defined in
Definition \ref{def:solution}. The above results follow from Theorem \ref{t:global_well_posedness_general} and Lemma \ref{lem:dissipationII}. Continuous dependence on the initial data also holds, see Proposition \ref{prop:L_m_continuity}.

The explicit form of $g_n$ in \eqref{eq:AC_intro} is been chosen to avoid technicalities here, and it is not needed for Theorem \ref{t:allen_Cahn_intro} to hold.
The complete picture is given in Example \ref{ex:strong_dissipation}.

The quadratic nonlinearity in \eqref{eq:AC_intro}, and the choice of the initial data from $L^3$ are optimal from a scaling point of view, see \cite[Subsection 1.4]{AVreaction-local} for $d=h=3$.
In addition to the growth condition of $g_n$, we also assumed in Theorem \ref{t:allen_Cahn_intro} the explicit bound $\|\theta\|_{\ell^2}\leq 1$ that limits the possible growth of $g_n$ at infinity. This restriction appears naturally when checking the $L^{\m}$-coercivity/dissipativity condition \eqref{eq:coercivity_Lm_intro} with $\m=\frac{d(h-1)}{2}=3$. Physically, the condition $\|\theta\|_{\ell^2}\leq 1$ ensures that the dissipation of the deterministic nonlinearity $u-u^3$ is balanced by the energy production of $g_n$.
Hence, in the situation of Theorem \ref{t:allen_Cahn_intro}, the restriction on $\theta$ seems optimal with our methods.

Although there have been several works on Allen-Cahn equations and their variations, see e.g.\ \cite{ABBK16,BBP17_2,FS19,HR18,HMS20,RW13,S00_AC}, part of the content of Theorem \ref{t:allen_Cahn_intro} seems to be new.
The novelty is that the diffusion $g_n$ is of optimal growth (even with explicit constrain on $\theta$) and the transport noise can be rough \eqref{eq:b_smoothness_intro}. In particular, we can cover the case of transport by turbulent fluids (see the comments below \eqref{eq:b_smoothness_intro}). Surprisingly, the solution to \eqref{eq:AC_intro} instantaneously regularizes in time and space \eqref{eq:AC_regularization_intro}, regardless of the one of the initial data $u_0$. With respect to the noise, the gain of regularity in \eqref{eq:AC_regularization_intro} is sub-optimal:
\begin{equation}
\label{eq:optimal_gain_regularity}
(b_n)_{n\geq 1}\in C^{\gamma}(\Tor^3;\ell^2(\N;\R^3))\quad \Longrightarrow \quad u(t)\in C^{1+\gamma-\varepsilon}( \Tor^3) \ \text{ a.s.\ for all }t, \varepsilon>0.
\end{equation}
Moreover, in the Sobolev scale, an optimal gain of regularity holds, cf.\ \cite[Theorem 4.2]{AVreaction-local}.
Such instantaneous gain of regularity can be used, for example, to study the separation property of solutions to Allen-Cahn type equations, see e.g.\ \cite[Theorem 1.3]{BOS22_separation} where a more complicated nonlinearity is considered.

\subsection{The system case and the 3d Brusselator}
In the system case, the situation is much more involved compared to the scalar one. First, there are several ways to extend the condition \eqref{eq:coercivity_Lm_intro} to the system case. One possibility is to require that \emph{all} components of \eqref{eq:reaction_diffusion_system_intro} satisfy \eqref{eq:coercivity_Lm_intro}, or more generally that the sum over $i$ of all components of the system \eqref{eq:reaction_diffusion_system_intro} satisfy an estimate like \eqref{eq:coercivity_Lm_intro}. Such a condition covers the case of ``fully" dissipative systems and we examine this in Section \ref{s:scalar_global}.
However, in most of the cases of practical interests, the dissipative effect of the nonlinearity is \emph{not} present in all components of the system \eqref{eq:reaction_diffusion_system_intro}. A prototype example that will be treated here is the Brusselator system (see Subsection \ref{ss:brusselator} for physical motivations):
\begin{equation}
\label{eq:brusselator_intro}
f_1(\cdot,u)= -u_1 u_2^2 \qquad \text{ and }\qquad f_2(\cdot,u)= u_1u_2^2.
\end{equation}
In the case of physically relevant positive solutions, one sees that $f_1$ is dissipative, while $f_2$ is not, i.e.\ $f_1(\cdot,u)u_1\leq -u_1^2u_2^2\leq 0$ and  $f_2(\cdot,u)u_2= u_1 u_2^3\geq 0$. Thus, it is not possible to prove an estimate for $u=(u_1,u_2)$ directly, since there is no way to adsorb the positive term $f_2(\cdot,u)u_2$ when performing energy estimates.
Our strategy for studying \eqref{eq:reaction_diffusion_system_intro} is to first estimate $u_1$, and based on that, to determine a bound for $u_2$. This is exactly the content of Subsection \ref{ss:suboptimal_coercivity} where we provide a generalization of the condition \eqref{eq:coercivity_Lm_intro} for the case where $M$ is allowed to be random.
The results we can prove for weakly dissipative systems are new, but a complete theory as in the deterministic case, is still out of reach. The reader is referred to Subsection \ref{ss:open_problems} for more details.
At the present stage, the situation for the system case is not as complete as the scalar one.

Below we state our main result for the 3d Brusselator system, i.e.\ \eqref{eq:reaction_diffusion_system_intro} with \eqref{eq:brusselator_intro}. The case of dimensions $\leq 2$ is discussed in Subsection \ref{sss:brusselator_two_dimensional}.

\begin{theorem}[Global well-posedness -- 3d Brusselator]
\label{t:brusselator_intro}
Suppose that \eqref{eq:b_smoothness_intro} holds. Suppose that $F = 0$, $f$ is as in \eqref{eq:brusselator_intro} and there exists an $N>0$ such that $g_{i}=(g_{n,i})_{n\geq 1}:\R^2 \to \ell^2$ satisfies
\begin{align}
\label{eq:ass_g_1_intro}
\|g_{i}(y)-g_i(y')\|_{\ell^2}
&\leq N (1+|y|+|y'|)|y-y'|, &\text{ for all } & \ y,y'\in \R^2, \ i\in \{1,2\},\\
\label{eq:ass_g_2_intro}
\|g_1(y)\|_{\ell^2}^2 &\leq N(1+ y_1)+ \frac{y_1y_2}{5}, &\text{ for all } & \ y=(y_1,y_2)\in [0,\infty)^2,\\
\label{eq:ass_g_3_intro}
\|g_2(y)\|_{\ell^2}^2 &\leq N (1+y_1+y_2+y_1y_2), &\text{ for all } &\ y=(y_1,y_2)\in [0,\infty)^2,
\end{align}
and $g_1(0,y_2) = g_2(y_1,0)=0$ for all $y_1, y_2\geq 0$.
Fix $\delta\in (1,\frac{4}{3}\wedge (1+\alpha))$, $p\in (2,\infty)$ and $\a_{\crit}:=p(1-\frac{\s}{2})-1$. Let $
u_0\in
L^0_{\F_0}(\O;L^3(\Tor^3;\R^2))
$ be such that a.s.\ $u_0\geq 0$ on $\Tor^3$, component-wise.
Suppose that, for some $\ellip_{0,i}\in (0,\ellip_i)$ and $i\in \{1,2\}$,
$$
\sum_{n\geq 1}|  b_{n,i}(x) \cdot \xi|^2\leq 2 \ellip_{0,i}|\xi|^2\ \text{ for } x\in \Tor^3, \ \xi\in \R^3.
$$
Then there exists a unique global $(p,\a_{\crit},\s,q)$-solution $u$ to \eqref{eq:reaction_diffusion_system_intro}
satisfying
\begin{align*}
u&\in H_{\loc}^{\theta,p}([0,\infty),t^{\kappa_{\crit}} \, \dd t;H^{2-\s,q}(\Tor^3;\R^2))\cap C([0,\infty);L^{3}(\Tor^3;\R^2))      \text{ a.s.\ }\forall \theta\in [0,\tfrac{1}{2}), \\
u&\in C^{\theta_1,\theta_2}_{\rm loc}((0,\infty)\times \Tor^3;\R^2)
 \text{ a.s.\  }\forall  \theta_1\in  [0,\tfrac{1}{2}), \ \forall  \theta_2\in (0,1),\\
 u&\geq 0 \text{ component-wise a.e.\ on }(0,\infty)\times \O\times \Tor^3.
\end{align*}
\end{theorem}

The above result follows from Theorem \ref{t:brusselator_three} with $h=3$, and its proof requires the full power of our theory and is the final application in this paper. As above, Proposition \ref{prop:L_m_continuity} ensures continuous dependence on the initial data.

Note that the Brusselator system \eqref{eq:reaction_diffusion_system_intro} has the same local scaling as the Allen-Cahn equation discussed in Subsection \ref{ss:scalar_AC_intro}, i.e.\ the deterministic nonlinearity has cubic growth. Hence, as for the Allen-Cahn equation, the growth of $g_1,g_2$ in \eqref{eq:ass_g_1_intro} and the choice of $L^3$-data is optimal as well.
Moreover, under additional smoothness assumptions on $g_1,g_2$ (cf.\ \cite[Assumption 4.1]{AVreaction-local}), also in the case of the Brusselator system, we have a sub-optimal gain of regularity w.r.t.\ the regularity of the noise, i.e.\ \eqref{eq:optimal_gain_regularity} holds for the solution $u=(u_1,u_2)$ to \eqref{eq:reaction_diffusion_system_intro}. An optimal gain of regularity holds again in the Sobolev scale.

Conditions \eqref{eq:ass_g_2_intro}-\eqref{eq:ass_g_3_intro}
can considered as the analogue of $\|\theta\|_{\ell^2}\leq 1$ in Theorem \ref{t:allen_Cahn_intro}.
Regarding the latter condition, requirements  \eqref{eq:ass_g_2_intro}-\eqref{eq:ass_g_3_intro} ensure that the dissipation of $f_1$ is stronger than the energy production of $(f_2,g_1,g_2)$. In particular, the numerical factor $\frac{1}{5}$ in the last term on the RHS\eqref{eq:ass_g_2_intro} seems optimal in the general situation of Theorem \ref{t:allen_Cahn_intro}, at least with our methods. To prove the global well-posedness of \eqref{eq:reaction_diffusion_system_intro} we use the previously mentioned variant of \eqref{eq:coercivity_Lm_intro} where $M$ is allowed to be random.

\subsection{Further comparison to the literature}
\label{ss:literature}
To the best of our knowledge, the present work is the first to consider superlinear diffusion and transport noise simultaneously (this is the reason why we only consider trace-class noise).

In the presence of transport noise, the only work known to us is by F.\ Flandoli
\cite{F91}. There, the author showed the global existence of scalar reaction-diffusion equations under a strong-dissipation condition on $f$ and $g\equiv 0$. Moreover, it is assumed that $f$ is of \emph{odd} order. The results of Section \ref{s:scalar_global} extend and refine those of \cite{F91}, cf.\ Example \ref{ex:strong_dissipation}.

In absence of transport noise, reaction-diffusion equations have been studied by many authors. Let us emphasise that in the case $b_{n,i}\equiv 0$, one can also treat rougher noise than the one we use in \eqref{eq:reaction_diffusion_system_intro} (i.e.\ white or coloured noise instead of trace-class noise). The one-dimensional case with rough and multiplicative noise has been investigated in \cite{DKZ19,Fra99,M98,M00}. For dimensions $d\geq 2$, the work of S.\ Cerrai \cite{C03} shows global existence under strong dissipative conditions on $f_i$ and sublinear growth of $g_i$.
Later, M.\ Salins in the works \cite{S21_superlinear_no_sign,S21_dissipative} extended the results of \cite{C03} in the case of scalar equations. More precisely, in
\cite{S21_superlinear_no_sign} global existence for scalar reaction-diffusion equations has been proven for superlinear $(f,g)$ without any dissipative conditions but employing a ``Osgood'' type condition. In particular, in \cite{S21_superlinear_no_sign}, $f$ cannot grow more than $|u|^{\g}$ for some $\g>1$.
The paper \cite{S21_dissipative} studies a reaction-diffusion equations in the scalar case where $f$ is strongly dissipative and $g$ is of ``lower-order" compared to the dissipative effect of $f$, see \cite[eq.\ (1.7)]{S21_dissipative}.
The work \cite{S21_dissipative} is closer to the present work. In the special case that both approaches are applicable (i.e.\ in the presence of trace-class noise), our results improve those of \cite{S21_dissipative}. The reader is referred to Remark \ref{r:comparison}\eqref{it:salins} for details.

Finally, we mention further work in the case of non-locally Lipschitz nonlinearity, which is beyond the scope of this paper.
In \cite{K15_holder_diffusion} the case of strong dissipative and locally Lipschitz $f$, and H\"older continuous $g$ was considered, where $g$ grows sublinearly. In the case of (possible) non-continuous $f_i$ and global Lipschitz diffusion $g$, well-posedness was considered in \cite{M18} by using the theory of monotone operators.

\subsection{Open problems}
\label{ss:open_problems}
There are still many open problems concerning the global well-posedness of stochastic reaction-diffusion equations. Two classes for which the {\em deterministic setting} is well understood are:
\begin{itemize}
\item Systems with a \emph{triangular} structure.
\item Systems with mass conservation and \emph{quadratic} nonlinearities.
\end{itemize}

Deterministic systems with a triangular structure are well-understood, see e.g.\ \cite[Section 3]{P10_survey}. Here triangular structure means that for a lower triangular invertible matrix $R$, the mapping $\R^{\ell}\ni y\mapsto R f(y)$ grows linearly in $|y|$. If $F\equiv 0$, then global strong solutions to \eqref{eq:reaction_diffusion_system} (without noise) exist, see e.g.\ \cite[Theorem 3.5]{P10_survey}. The core of the proof is a duality argument (see \cite[Theorem 3.4]{P10_survey}) which we do not know how to extend to SPDEs with transport noise. The Lotka-Volterra equation and the Brusselator system, analyzed in Subsection \ref{ss:Lotka_Volterra}-\ref{ss:brusselator}, are triangular systems.

Deterministic systems with mass conservation and quadratic nonlinearities with $F\equiv 0$ and mass control, global existence of strong solutions to \eqref{eq:reaction_diffusion_system} (without noise) has been shown in \cite{FMT20} provided $|f(y)|\lesssim 1+|y|^{2+\varepsilon}$ where $\varepsilon>0$ is sufficiently small). The proofs are based on an ingenious combination of interpolation inequalities and H\"{o}lder regularity theory for suitable auxiliary functions of $u=(u_i)_{i=1}^{\ell}$. At the present stage, the extension of these arguments to their stochastic variant with transport noise seems out of reach.

\subsection{Notation}
Here we collect some notation that will be used throughout the paper. Further notation will be introduced where needed. We write $A \lesssim_P B$ (resp.\ $A \gtrsim_P B$) whenever there is a constant $C>0$ depending only on $P$  such that $A\leq C B$ (resp.\ $A\geq C B$). We write $C(P)$ if the constant $C$ depends only on $P$. As usual
$a\vee b=\max\{a,b\} $ and $a\wedge b=\min\{a,b\}$ for $a,b\in \R$.

Letting $p\in (1,\infty)$ and $\a\in (-1,p-1)$, we denote by $w_{\a}$ the \emph{weight} $w_{\a}(t)=|t|^{\a}$ for $t\in\R$.
For a Banach space $X$ and an interval $I=(a,b)\subseteq \R$, $L^p(a,b,w_{\a};X)$ denotes the set of all strongly measurable maps $f:I\to X$ such that
$$
\|f\|_{L^p(a,b,w_{\a};X)}:= \Big(\int_a^b \|f(t)\|_{X}^p w_{\a}(t)\,\dd t\Big)^{1/p}<\infty.
$$
Furthermore, $W^{1,p}(a,b,w_{\a};X)\subseteq L^p(a,b,w_{\a};X)$ denotes the set of all $f$ such that $f'\in L^p(a,b,w_{\a};X)$ (here the derivative is taken in the distributional sense) and we set
$$
\|f\|_{W^{1,p}(a,b,w_{\a};X)}:=
\|f\|_{L^p(a,b,w_{\a};X)}+
\|f'\|_{L^p(a,b,w_{\a};X)}.
$$

Let $(\cdot,\cdot)_{\theta,p}$ and $[\cdot,\cdot]_{\theta}$ be the real and complex interpolation functor, respectively. The reader is referred to \cite{BeLo,Analysis1} for details. For each  $\theta\in (0,1)$ we set
$$
H^{\theta,p}(a,b,w_{\a};X) = [L^p(a,b,w_{\a};X),W^{1,p}(a,b,w_{\a};X)]_{\theta}.
$$
In the unweighted case, i.e.\ $\a=0$, we set $ H^{\theta,p}(a,b;X) :=H^{\theta,p}(a,b,w_0;X) $ and similar for Lebesgue spaces.
For $\A\in \{L^p,H^{\theta,p},W^{1,p}\}$, we denote by $\A_{\loc}(a,b;X)$ (resp.\ $\A_{\loc}([a,b),w_{\a};X)$) the set of all strongly measurable maps $f:(c,d)\to X$ such that $f\in\A(c,d;X)$ for all $a<c<d<b$ (resp.\ $f\in\A(a,c,w_{\a};X)$ for all $a<c<b$).

The $d$-dimensional torus is denoted by $\Tor^d$ where $d\geq 1$.
For $\theta_1, \theta_2\in (0,1)$, $C^{\theta_1, \theta_2}_{\loc}((a,b)\times \Tor^d;\R^\ell)$ denotes the space of all maps $v:(a,b)\times \Tor^d\to \R^\ell$ such that for all $a<c<d<b$ we have
\[
|v(t, x) - v(t', x')|\lesssim_{c,d} |t-t'|^{\theta_1} + |x-x'|^{\theta_2}, \ \text{ for all }t,t'\in [c,d], \ x,x'\in \Tor^d.
\]
This definition is extended to $\theta_1,\theta_2\geq 1$ by requiring that the partial derivatives $\partial^{\alpha,\beta}v$ (with $\alpha\in \N$ and $\beta\in \N^d$) exist and are in $C^{\theta_1-|\alpha|, \theta_2-|\beta|}_{\loc}((a,b)\times\Tor^d;\R^{\ell})$ for all
$\alpha \leq \lfloor \theta_1 \rfloor$ and $\sum_{i=1}^d \beta_i\leq \lfloor \theta_2 \rfloor$.

We will use the periodic Bessel potential spaces $H^{s,q}(\T^d)$ and Besov spaces $B^{s}_{q,p}(\T^d)$. Here $q\in (1,\infty)$ denotes the integrability, $p\in [1,\infty]$ is the microscopic parameter, and $s\in \R$ denotes the smoothness. For details on Besov spaces, see e.g.\ \cite{Tri83,Tr1}. Note that we write $u\in B^{s}_{q,p}(\T^d;\R^\ell)$  if each of the components of $u$ is in $B^{s}_{q,p}(\T^d)$, and similarly for Bessel potential spaces. We also employ the standard abbreviation $H^{s}:=H^{s,2}$. We often write $B^s_{q,p}$ instead of $B^{s}_{q,p}(\T^d;\R^\ell)$ if no confusion seems likely, and similarly for $H^{s,q}$ and $H^s$. 

Finally, we collect the main probabilistic notation. In the paper, we fix a filtered probability space $(\O,\mathscr{A},(\F_t)_{t\geq 0}, \P)$ and we denote by $\E[\cdot]=\int_{\O}\cdot\,\dd \P$ the expected value. A map $\sigma:\O\to [0,\infty]$ (note that the value $\infty$ is included) is called a stopping time if $\{\sigma\leq t\}\in \F_t$ for all $t\geq 0$. For stopping times $\sigma,\tau$ we introduce the notation (which might be nonstandard)
$$
[\tau,\sigma]\times \O:=\{(t,\om)\in [0,\infty)\times \O\,:\, \tau(\om)\leq t\leq \sigma(\om)\}.
$$
Similar definitions hold for
$[\tau,\sigma)\times \O$,
$(\tau,\sigma)\times \O$ etc.
$\Progress$ denotes the progressive $\sigma$-algebra on the above mentioned probability space.

\subsection{Overview}
Below we give an overview of the results proven here.

\begin{itemize}
\item Section \ref{s:main_results} - Review of the local existence theory of \cite{AVreaction-local}.
\item Section \ref{s:scalar_global} -
Global existence via $L^{\m}$-coercivity/dissipativity condition, including:
\begin{itemize}
\item[$\diamond$] Subsection \ref{ss:example_nonlinearities} - Fully dissipative systems/equations, e.g.\ Allen-Cahn equation.
\item[$\diamond$] Subsection \ref{ss:coagulation} - Coagulation dynamics.
\item[$\diamond$] Subsection \ref{subsec:Symbiotic Lotka-Volterra} - Symbiotic Lotka-Volterra equations.
\end{itemize}
\item Section \ref{s:global_system} - Global existence for weakly dissipative systems:
\begin{itemize}
\item[$\diamond$] Subsection \ref{ss:suboptimal_coercivity} - Energy estimates via iteration of $L^{\m}$-coercivity/dissipativity conditions.
\item[$\diamond$] Subsection \ref{ss:Lotka_Volterra} - Global existence for Lotka-Volterra equations.
\item[$\diamond$] Subsection \ref{ss:brusselator} - Global existence for Brusselator system.
\end{itemize}
\item Section \ref{s:continuous_dependence} - Continuous dependence on the initial data in the presence of energy estimates.
\end{itemize}

\section{A review of the local existence theory}
\label{s:main_results}
In this section, we recall the local well-posedness results of \cite{AVreaction-local}, as well as the blow-up criteria, regularity, and positivity results.

Consider the following system of SPDEs on $\T^d$:
\begin{equation}
\label{eq:reaction_diffusion_system}
\left\{
\begin{aligned}
\dd u_i -\div(a_i\cdot\nabla u_i) \,\dd t& = \Big[\div(F_i(\cdot, u)) +f_i(\cdot, u)\Big]\,\dd t + \sum_{n\geq 1}  \Big[(b_{n,i}\cdot \nabla) u_i+ \rnoise_{n,i}(\cdot,u) \Big]\,\dd w_t^n,\\
u_i(0)&=u_{0,i},
\end{aligned}\right.
\end{equation}
where $i\in \{1,\dots,\ell\}$ and $\ell\geq 1$ is an integer.
Here $u=(u_i)_{i=1}^{\ell}:[0,\infty)\times \O\times \Tor^d\to \R^\ell$ is the unknown process, $(w^n)_{n\geq 1}$ is a sequence of standard independent Brownian motion on the above mentioned filtered probability space and
$$
\div (a_i\cdot\nabla u_i):=\sum_{j,k=1}^d \partial_j(a^{j,k}_i \partial_k u_i)
\quad \text{ and }\quad
(b_{n,i}\cdot\nabla) u_i:=\sum_{j=1}^d b^j_{n,i} \partial_j u_i.
$$
As explained in Section \ref{s:intro}, $b_{n,i}^j$ models small-scale transport effects. The $a^{j,k}_i$
take into account inhomogeneous conductivity and may also take into account the It\^o correction in the case of Stratonovich noise. Note that in the $i$-th equation there is no mixing in the diffusion terms $\div (a_i\cdot\nabla u_i)$ and $(b_{n,i}\cdot \nabla) u_i$, which is the usual assumption in reaction-diffusion systems.

The following is the main assumption on the coefficients and nonlinearities.
\begin{assumption}
\label{ass:reaction_diffusion_global}
Let $d,\ell\geq 1$ be integers. We say that Assumption \ref{ass:reaction_diffusion_global}$(p,q,h,\s)$ holds if $q\in [2,\infty)$, $p\in [2,\infty)$, $h>1$, $\s\in [1, 2)$ and for all $i\in \{1,\dots,\ell \}$ the following hold:
\begin{enumerate}[{\rm(1)}]
\item\label{it:reaction_diffusion_global1} For each $j,k\in \{1,\dots,d\}$, $\am^{j,k}_i:\R_+\times \O\times \Tor^d\to \R$, $b_{i}^j:=(\bm^{j}_{n,i})_{n\geq 1}:\R_+\times \O\times \Tor^d\to \ell^2$ are $\Progress\otimes \Borel(\Tor^d)$-measurable;
\item\label{it:regularity_coefficients_reaction_diffusion}
There exist $N>0$ and $\alpha>\max\{\frac{d}{\rho},\s-1\}$ where $\rho \in [2,\infty)$
such that a.s.\ for all $t\in \R_+$ and $j,k\in \{1,\dots,d\}$,
\begin{align*}
\|\am^{j,k}_i(t,\cdot)\|_{H^{\alpha,\rho}(\Tor^d)}+\|(\bm^{j}_{n,i}(t,\cdot))_{n\geq 1}\|_{H^{\alpha,\rho}(\Tor^d;\ell^2)}
&\leq N;
\end{align*}
\item\label{it:ellipticity_reaction_diffusion} There exists $\ellip_i>0$ such that a.s.\ for all $t\in \R_+$, $x\in \Tor^d$ and $\xi\in \R^d$,
$$
\sum_{j,k=1}^d \Big(a_i^{j,k}(t,x)-\frac{1}{2}\sum_{n\geq 1} b^j_{n,i}(t,x)b^k_{n,i}(t,x)\Big)
 \xi_j \xi_k
\geq  \ellip_i |\xi|^2;
$$
\item\label{it:growth_nonlinearities} For all $j\in \{1,\dots,d\}$,  the maps
\begin{align*}
F_i^j, \ f_i:\R_+\times \O\times \Tor^d\times \R\to \R,\qquad
g_i:=(g_{n,i})_{n\geq 1}:\R_+\times \O\times \Tor^d\times \R\to \ell^2,
\end{align*}
are $\Progress\otimes \Borel(\Tor^d)\otimes \Borel(\R)$-measurable. Set $F_i:=(F_i^j)_{j=1}^d$. Assume that
\begin{equation*}
F_i^j(\cdot,0), \ f_i(\cdot,0)\in L^{\infty}(\R_+\times \O\times \Tor^d),\qquad
g_i(\cdot,0)\in L^{\infty}(\R_+\times \O\times \Tor^d;\ell^2),
\end{equation*}
and a.s.\ for all $t\in \R_+$, $x\in \Tor^d$ and $y\in\R$,
\begin{align*}
|f_i(t,x,y)-f_i(t,x,y')|
&\lesssim (1+|y|^{h-1}+|y'|^{h-1})|y-y'|,\\
|F_i(t,x,y)-F_i(t,x,y')|
&\lesssim (1+|y|^{\frac{h-1}{2}}+|y'|^{\frac{h-1}{2}})|y-y'|.
\\ \|g_i(t,x,y)-g_i(t,x,y')\|_{\ell^2}
&\lesssim (1+|y|^{\frac{h-1}{2}}+|y'|^{\frac{h-1}{2}})|y-y'|.
\end{align*}
\end{enumerate}
\end{assumption}

Note that Assumption \ref{ass:reaction_diffusion_global}\eqref{it:regularity_coefficients_reaction_diffusion} and Sobolev embeddings imply that $a^{j,k}_i\in C^{\gamma}(\T^d)$ and $(b_{n,i}^j)_{n\geq1} \in C^{\gamma}(\T^d;\ell^2)$ uniformly w.r.t.\ $(t,\om)$ for some $\gamma>0$ depending only on $\alpha,d$ and $\rho$.

Often we will additionally assume the following on the parameters (see \cite[Assumption 2.4 and Section 7]{AVreaction-local}).
\begin{assumption}\label{ass:admissibleexp}
Let $d\geq 2$. We say that Assumption \ref{ass:admissibleexp}$(p,q,h,\delta)$ holds if
$h>1$,
and one of the following cases holds:
\begin{enumerate}[(i)]
\item\label{it:admissibleexp1} $p\in (2,\infty)$, $q\in [2,\infty)$, and $\s\in [1,\frac{h+1}{h})$,  satisfy
\begin{align*}
\frac{1}{p}+\frac{1}{2}\Big(\reg+\frac{d}{q}\Big) \leq \frac{h}{h-1}, \ \
\frac{d}{d-\reg} <q<\frac{d(h-1)}{h+1-\reg(h-1)}, \  \text{and} \ \a_{\crit}:=p\Big(\frac{h}{h-1}-\frac{1}{2}(\reg+\frac{d}{q})\Big)-1.
\end{align*}
\item\label{it:admissibleexp2} $p=q=2$, $\kappa=0$, $\delta=1$ and $h\leq \frac{4+d}{d}$ with the additional restriction $h<3$ if $d=2$.
\end{enumerate}
\end{assumption}
The parameters $q$ and $p$ will be used for spatial and time integrability, respectively, and the parameter $\delta$ is used to decrease spatial smoothness. The parameter $\kappa_{\crit}$ is related to the critical weight $t^{\kappa}$ used for the time-integrability.  It is interesting to note that large values of $h$ require lower values of $\delta$, which requires higher values of $q$.

Below, we stress the dependence on $(p,\a,q,h,\s)$ in the definition of solutions. However, in \cite[Proposition 3.5 and Remark 7.4]{AVreaction-local} it is proved that the solutions to \eqref{eq:reaction_diffusion_system} for different choices of the parameters $(p,\a,q,h,\s)$ satisfying Assumption \ref{ass:admissibleexp}, actually coincide. The sequence of independent standard Brownian motions $(w^n)_{n\geq 1}$ uniquely induce an $\ell^2$-cylindrical Brownian motion given by $W_{\ell^2}(v):=\sum_{n\geq 1} \int_{\R_+} v_n dw^n_t$ where $v=(v_n)_{n\geq 1}\in L^2(\R_+;\ell^2)$.

\begin{definition}
\label{def:solution}
Assume that Assumption \ref{ass:reaction_diffusion_global}$(p,q,h,\s)$ be satisfied for some $h>1$ and let $\a\in [0,\frac{p}{2}-1)$. Suppose that $u_0\in B^{2-\s-2\frac{1+\a}{p}}_{q,p}(\Tor^d;\R^{\ell})$ a.s.
\begin{itemize}
\item Let $\sigma$ be a stopping time and $u=(u_i)_{i=1}^{\ell}:[0,\sigma)\times \O\to H^{2-\s,q}(\Tor^d;\R^\ell)$ be a stochastic process.
We say that $(u,\sigma)$ is a \emph{local $(p,\a,\s,q)$-solution} to \eqref{eq:reaction_diffusion_system} if there exists a sequence of stopping times $(\sigma_j)_{j\geq 1}$ such that the following hold for all $i\in \{1,\dots,\ell\}$.
\begin{itemize}
\item $\sigma_j\leq \sigma$ a.s.\ for all $j\geq 1$ and $\lim_{j\to \infty} \sigma_j =\sigma$ a.s.;
\item for all $j\geq 1$, the process $\one_{[0,\sigma_j]\times \O} u_i$ is progressively measurable;
\item a.s.\ for all $j\geq 1$, we have $u_i\in L^p(0,\sigma_j ,w_{\a};H^{2-\s,q}(\Tor^d))$ and
\begin{equation*}
\begin{aligned}
\div(F_i(\cdot, u)) +f_i(\cdot, u)\in L^p(0,\sigma_j,w_{\a};H^{-\s,q}(\Tor^d)),&\\
(g_{n,i}(\cdot,u))_{n\geq 1}\in L^p(0,\sigma_j,w_{\a};H^{1-\s,q}(\Tor^d;\ell^2));&
\end{aligned}
\end{equation*}
\item a.s.\ for all $j\geq 1$ the following holds for all $t\in [0,\sigma_j]$:
\begin{equation*}
\begin{aligned}
u_i(t)-u_{0,i}
&=\int_{0}^{t} \Big(\div(a_i\cdot\nabla u_i)+ \div(F_i(\cdot, u)) +f_i(\cdot, u)\Big)\,\dd s\\
&=\int_{0}^t\Big( \one_{[0,\sigma_j]}\big[ (b_{n,i}\cdot\nabla)u + g_{n,i}(\cdot,u) \big]\Big)_{n\geq 1} \dd W_{\ell^2}(s).
\end{aligned}
\end{equation*}
\end{itemize}
\item $(u,\sigma)$ is a \emph{$(p,\a,\s,q)$-solution} to \eqref{eq:reaction_diffusion_system} if for any other local $(p,\a,\s,q)$-solution $(u',\sigma')$ to \eqref{eq:reaction_diffusion_system} we have $\sigma'\leq \sigma$ a.s.\ and $u=u'$ on $[0,\sigma')\times \O$.
\end{itemize}
\end{definition}

Note that a $(p,\a,q,\s)$-solution is \emph{unique} by definition. The main result on local existence, uniqueness, and regularity, is as follows (see \cite[Theorem 2.6]{AVreaction-local}).
\begin{theorem}[Local existence in critical spaces]
\label{t:reaction_diffusion_global_critical_spaces}
Let Assumptions \ref{ass:reaction_diffusion_global}$(p,q,h,\s)$ and \ref{ass:admissibleexp}$(p,q,h,\s)$ be satisfied. Then for any
\begin{equation*}
u_0\in L^0_{\F_0}(\O;B^{\frac{d}{q}-\frac{2}{h-1}}_{q,p}(\Tor^d;\R^{\ell})),
\end{equation*}
the problem \eqref{eq:reaction_diffusion_system} has a (unique) $(p,\a_{\crit},\s,q)$-solution $(u,\sigma)$ such that a.s.\ 
 $\sigma>0$ and
\begin{equation}
\label{eq:regularity_u_reaction_diffusion_critical_spaces}
\begin{aligned}
u&\in C([0,\sigma);B^{\frac{d}{q}-\frac{2}{h-1}}_{q,p}(\Tor^d;\R^{\ell})) ,
\\
u&\in H^{\theta,p}_{{\rm loc}}([0,\sigma),w_{\a_{\crit}};H^{2-\s-2\theta,q}(\Tor^d;\R^{\ell})) \ \text{ $\forall \theta\in [0,\tfrac{1}{2})$ if $p>2$, and  $\theta=0$ if $p=2$.}
\end{aligned}
\end{equation}
Moreover, $u$ instantaneously regularizes in space and time: a.s.,
\begin{align}
\label{eq:reaction_diffusion_H_theta}
u&\in H^{\theta,r}_{\rm loc}(0,\sigma;H^{1-2\theta,\zeta}(\Tor^d;\R^{\ell}))  \  \forall\theta\in  [0,1/2), \  \forall r,\zeta\in (2,\infty),\\
\label{eq:reaction_diffusion_C_alpha_beta}
u&\in C^{\theta_1,\theta_2}_{\rm loc}((0,\sigma)\times \Tor^d;\R^{\ell}) \  \forall \theta_1\in  [0,1/2), \ \forall \theta_2\in (0,1).
\end{align}
\end{theorem}

Note that \eqref{eq:reaction_diffusion_C_alpha_beta} is a consequence of \eqref{eq:reaction_diffusion_H_theta} and Sobolev embeddings.
Next, we state one of our blow-up criteria which we use to obtain global existence for the solution to \eqref{eq:reaction_diffusion_system} provided by Theorem \ref{t:reaction_diffusion_global_critical_spaces} (see \cite[Theorems 2.11 and Proposition 7.3]{AVreaction-local}).

\begin{theorem}[Blow-up criteria]\label{thm:blow_up_criteria}
Let the assumptions of Theorem \ref{t:reaction_diffusion_global_critical_spaces} be satisfied and let $(u,\sigma)$ be the $(p,\a_{\crit},q,\s)$-solution to \eqref{eq:reaction_diffusion_system}.
Let $h_0\geq 1+\frac{4}{d}$. Suppose that $p_0\in (2,\infty)$, $h_0\geq h$, $\s_{0}\in (1,2)$  are such that Assumptions \ref{ass:reaction_diffusion_global}$(p_0,q_0,h_0,\s_0)$ and \ref{ass:admissibleexp}$(p_0,q_0,h_0,\s_0)$ hold.
Let $\zeta_0 = \frac{d}{2}(h_0-1)$. The following hold for all $0<s<T<\infty$:
\begin{enumerate}[{\rm(1)}]
\setcounter{enumi}{\value{nameOfYourChoice}}
\item\label{it:blow_up_not_sharp_L}
If $q_0 = \zeta_0$, then for all $\zeta_1>q_0$
\[\P\Big(s<\sigma<T,\, \sup_{t\in [s,\sigma)}\|u(t)\|_{L^{\zeta_1}(\Tor^d;\R^{\ell})} <\infty\Big)=0.
\]
\item\label{it:blow_up_sharp_L}
If $q_0>\zeta_0$, $p_0\in \big(\frac{2}{\s_0-1},\infty\big)$, $p_0\geq q_0$, and $\frac{d}{q_0}+\frac{2}{p_0}=\frac{2}{h_0-1}$, then
\[
\P\Big(s<\sigma<T,\, \sup_{t\in [s,\sigma)}\|u(t)\|_{L^{\zeta_0}(\Tor^d;\R^{\ell})}+
\|u\|_{L^{p_0}(s,\sigma;L^{q_0}(\Tor^d;\R^{\ell}))} <\infty\Big)=0.
\]
\item\label{it:blow_up_sharp_2} If $p_0=q_0=2$, $\delta_0=1$, then
\[\P\Big(s<\sigma<T,\, \sup_{t\in [s,\sigma)}\|u(t)\|_{L^{2}(\Tor^d;\R^{\ell})}+
\|u\|_{L^{2}(s,\sigma;H^{1,2}(\Tor^d;\R^{\ell}))} <\infty\Big)=0.
\]
\end{enumerate}
\end{theorem}

The main result on positivity (in a distributional sense) is as follows (see \cite[Theorem 2.13]{AVreaction-local}).

\begin{proposition}[Positivity]
\label{prop:positivity}
Let the assumptions of Theorem \ref{t:reaction_diffusion_global_critical_spaces} be satisfied.
Let $(u,\sigma)$ be the $(p,\a_{\crit},\s,q)$-solution to \eqref{eq:reaction_diffusion_system} where $\a_{\crit}$ is as in Theorem \ref{t:reaction_diffusion_global_critical_spaces}. Suppose that
$$
u_0\geq 0 \ \text{a.s.\ (component-wise)},
$$
and that there exist progressive measurable processes $c_1,\dots,c_{\ell} :\R_+\times \O\to \R$ such that for all $i\in \{1,\dots,\ell\}$, $n\geq 1$, $y=(y_{i})_{i=1}^{\ell}\in [0,\infty)^{\ell}$ and
a.e.\ on $\R_+\times \O$,
\begin{align*}
f_i(\cdot,y_1,\dots,y_{i-1},0,y_{i+1},\dots,y_{\ell})&\geq 0,\\
F_{i}(\cdot,y_1,\dots,y_{i-1},0,y_{i+1},\dots,y_{\ell})&= c_i(\cdot),\\
g_{n,i}(\cdot,y_1,\dots,y_{i-1},0,y_{i+1},\dots,y_{\ell})&=0.
\end{align*}
Then
$
u(t,x)\geq 0
$
 a.s.\ for all  $x\in \Tor^d$ and $t\in [0,\sigma)$.
\end{proposition}

\section{Global existence and uniqueness for fully dissipative systems}
\label{s:scalar_global}
In this section, we turn to the problem of determining when the solutions of \eqref{eq:reaction_diffusion_system} provided by Theorem \ref{t:reaction_diffusion_global_critical_spaces} are global.
It is well-known that  Assumption \ref{ass:reaction_diffusion_global} is not enough to ensure global existence of \eqref{eq:reaction_diffusion_system_intro} in general even in the scalar case $\ell=1$, and further conditions on the nonlinearities are required (cf. \cite{F66}, \cite{pruss2016moving}, \cite[Section 19]{QS19} for the deterministic case).
In this section, we propose a new type of $L^\m$-coercivity/dissipation condition on $(a_i,F_i,f_i,b_i,g_i)$, where $\m = \frac{d(h-1)}{2}\vee 2$, and $h$ is as in Assumption \ref{ass:reaction_diffusion_global}.
Explicit examples will be discussed in Subsections \ref{ss:example_nonlinearities} and \ref{ss:applications_fully_dissipative} below.

\begin{assumption}[$L^\m$-coercivity/dissipativity]
\label{ass:dissipation_general}
Let Assumption \ref{ass:reaction_diffusion_global}$(p,q,h,\s)$ be satisfied. Let $\m\in [ 2,\infty)$. We say that Assumption \ref{ass:dissipation_general}$(\m)$ holds if there exist constants $\theta,M,C,\alpha_1,\dots,\alpha_{\ell}>0$ such that a.e.\ in $[0,\infty)\times \Omega$ for all $u=(u_i)_{i=1}^{\ell}\in C^1(\mathbb{T}^d;\R^{\ell})$,
\begin{align*}
\sum_{i=1}^{\ell} \alpha_i\int_{\T^d} |u_i|^{\m-2} \Big(\am_i \nabla u_i\cdot \nabla u_i   + F_i(\cdot, u) \cdot \nabla u_i - \frac{u_i f_i(\cdot, u)}{\m-1} -\frac12 \sum_{n\geq 1} \big[(b_{n,i} \cdot \nabla) u + g_{n,i}(\cdot, u) \big]^2 \Big) \, \dd x &
\\
 \geq  \theta
\sum_{i=1}^{\ell}  \int_{\T^d} |u_i |^{\m-2}|\nabla u_i |^2 \,\dd x -  M
\sum_{i=1}^{\ell} \int_{\T^d}|u_i |^{\zeta} \, \dd x -C&.
\end{align*}
Moreover, given $S\subseteq \R^{\ell}$ we say that Assumption \ref{ass:dissipation_general}$(\m)$ holds for $S$-valued functions if the above is satisfies for all $u\in C^1(\T^d;\R^{\ell})$ with $u(x)\in S$ for all $x\in \T^d$.
\end{assumption}

The constants $\alpha_1,\dots,\alpha_{\ell}$ can be chosen to exploit the dissipation effect of a certain nonlinearity appearing in the $i$-th equation in the system \eqref{eq:reaction_diffusion_system} with $i\in \{1,\dots,\ell\}$ fixed. In case of high dissipation in the $i$-th equation, one chooses $\alpha_{i}$ relatively large compared to $\alpha_j$ for $j\neq i$ (see Assumption \ref{ass:condsymb} in the case of symbiotic Lotka-Volterra equations).

By approximation and using Assumption \ref{ass:reaction_diffusion_global}, one can check that if the above assumption holds, then it extends to all $u\in H^{1,q_0}(\T^d;\R^{\ell})$ with $q_0>d$. A typical situation where the above assumption is satisfied is in the case $\frac{y_i f_i(\cdot, y)}{\m-1}\leq 0$ for $|y|\to \infty$ and all $i\in \{1,\dots,\ell\}$, i.e.\ the system of SPDEs \eqref{eq:reaction_diffusion_system} has \emph{dissipation} in all components. In the applications given in Subsections \ref{ss:example_nonlinearities} and \ref{ss:applications_fully_dissipative}, we will see the conditions needed to check Assumption \ref{ass:condsymb}$(\m)$ become more restrictive for larger values of $\m$. The $(a,b)$-terms can usually be estimated using Assumption  \ref{ass:reaction_diffusion_global}\eqref{it:ellipticity_reaction_diffusion}.
The reason for also considering Assumption \ref{ass:dissipation_general} for the restricted class of functions taken values in $S$ is that often the solution to \eqref{eq:reaction_diffusion_system} is known to take values $[0,\infty)^{\ell}$, $[-1,1]^{\ell}$ or $[0,1]^{\ell}$, etc.\ (see Proposition \ref{prop:positivity} for a sufficient condition for positivity).

Next, we state the main result of this section.
\begin{theorem}[Global well-posedness -- Fully dissipative systems]
\label{t:global_well_posedness_general}
Suppose that the assumptions of Theorem \ref{t:reaction_diffusion_global_critical_spaces} hold, and let $(u,\sigma)$ be the $(p,\a_{\crit},\s,q)$-solution to \eqref{eq:reaction_diffusion_system} where $\a_{\crit}$ is as in Theorem \ref{t:reaction_diffusion_global_critical_spaces} and suppose that $u$ takes values in $S$ a.e.\ on $[0,\sigma)\times\Omega$ for some $S\subseteq \R^{\ell}$.
Suppose that Assumption \ref{ass:dissipation_general}$(\m)$ holds for $S$-valued functions with $\m \geq \frac{d(h-1)}{2}\vee 2$. Then $(u,\sigma)$ is \emph{global} in time, i.e.\ $\sigma=\infty$ a.s.
In particular, \eqref{eq:regularity_u_reaction_diffusion_critical_spaces}--\eqref{eq:reaction_diffusion_C_alpha_beta} hold with $\sigma=\infty$. Moreover, for all $\lambda\in (0,1)$, there exist $N_0, N_{0,\lambda}>0$ such that for any $0<s<T$, $L\geq 0$ and $i\in \{1,\dots,\ell\}$,
\begin{align}
\label{eq:aprioribounds1} \sup_{t\in [s,T]}\E \one_{\Gamma} \|u_i(t)\|_{L^{\m}}^{\m } + \E \int_{s}^{T}\int_{\Tor^d} \one_{\Gamma} |u_i|^{\m-2} |\nabla u_i|^{2}\,\dd x\, \dd t&\leq N_0\Big(1+\E\one_{\Gamma}\|u(s)\|_{L^{\m}}^{\m}\Big),
\\ \label{eq:aprioribounds2}
\E \sup_{t\in [s,T)} \one_{\Gamma} \|u_i(t)\|_{L^{\m}}^{\m \lambda} + \E \Big|\int_{s}^{T}\int_{\Tor^d}\one_{\Gamma}  |u_i|^{\m-2} |\nabla u_i|^{2}\,\dd x\,\dd t\Big|^{\lambda}
& \leq N_{0,\lambda}\Big(1+\E\one_{\Gamma}\|u(s)\|_{L^{\m}}^{\m \lambda}\Big),
\end{align}
where $\Gamma = \{\|u(s)\|_{L^{\m}}\leq L\}$.
Moreover, one can take $s=0$ if for some $\varepsilon>0$, $u_0\in B^{\varepsilon}_{\zeta,\infty}(\T^d;\R^{\ell})$ a.s., or $u_0\in L^2(\T^d;\R^{\ell})$ a.s.\ and $\zeta = 2$.
\end{theorem}

The final assertion extends to $u_0\in L^{\zeta}(\T^d;\R^{\ell})$ a.s.\ and $\Gamma=\O$ as will be shown in Section \ref{s:continuous_dependence}, see Corollary \ref{cor:estimate_up_to_0}.
In case $s=0$, we can let $L\to \infty$, and thus omit $\Gamma$.
Later on in Proposition \ref{prop:L_m_continuity}, we will see that
if $u_0\in L^{\zeta}(\T^d;\R^\ell)$ a.s., then one can show that the estimates \eqref{eq:aprioribounds1} and \eqref{eq:aprioribounds2} also hold with $s=0$.
Finally, continuity with respect to the initial data follows from Theorem \ref{t:continuity_general}.

The proof of the above result will be given in Subsection \ref{ss:global_general}. In the next subsections, we first discuss some sufficient conditions for Assumption \ref{ass:dissipation_general}.

Often, when dealing with scalar equations, i.e.\ $\ell=1$, we omit the subscript $i=1$.

\subsection{Sufficient conditions for $L^\m$-coercivity/dissipativity}
\label{ss:example_nonlinearities}
In this subsection, we discuss examples of nonlinearities satisfying Assumption \ref{ass:dissipation_general}.
Before doing so, we write down two classes of examples where Assumption \ref{ass:dissipation_general} is satisfied. In Lemma \ref{lem:dissipationI} we present a pointwise condition for $L^{\m}$-coercivity/dissipativity without assuming any smoothness on $(F_i,b_i,g_i)$. This condition is further weakened in Lemma \ref{lem:dissipationII} where the case of $\ell=1$ and smooth functions $(F,b,g)$ is considered. Further variations where only one or two functions have smoothness are also possible.

\begin{lemma}[Pointwise $L^\m$--coercivity/dissipativity without smoothness]
\label{lem:dissipationI}
Suppose that Assumption \ref{ass:reaction_diffusion_global} holds with $\m\geq 2$. For $\varepsilon_i>0$ let
\[N_i^{\varepsilon_i}(\cdot, y) = \frac{y_i f_i(\cdot, y)}{\m-1} +\frac12 \sum_{n\geq 1} |g_{n,i}(\cdot, y)|^2 + \frac{4}{\nu_i-\varepsilon_i} \Big(|F_i(\cdot, y)|+\sum_{n\geq 1} |b_{n,i}(\cdot)| \,  |g_{n,i}(\cdot,y)|\Big)^{2}.\]
If there exist $\varepsilon_i\in (0,\nu_i)$, $M>0$ and $\alpha_1, \ldots, \alpha_{\ell}>0$ such that a.e.\ in $[0,\infty)\times\Omega\times \T^d$ for all $y\in \R^{\ell}$
\begin{equation}\label{eq:dissipationIsystem}
 \sum_{i=1}^\ell \alpha_i |y_i|^{\m-2} N_i^{\varepsilon_i}(\cdot, y) \leq   M(|y|^\m+1),
\end{equation}
then Assumption \ref{ass:dissipation_general}$(\m)$ is satisfied. Moreover, if the above only holds with $y\in S\subseteq \R^{\ell}$, then  Assumption \ref{ass:dissipation_general}$(\m)$ holds for $S$-valued functions.
\end{lemma}

\begin{proof}
We prove the claim only for $\ell=1$ as the general case is analogous. We omit the subscript $i=1$, for notational convenience.

We claim that for all $z\in \R^d$ and $y\in \R$,
\begin{equation}\label{eq:dissipationhelp}
\am z\cdot z   + F(\cdot, y) \cdot z - \frac{y f(\cdot, y)}{\m-1} -\frac12 \sum_{n\geq 1} \big[b_n \cdot z + g_n(\cdot, y) \big]^2 \geq \theta |z|^2 - M(|y|^2+1).
\end{equation}
The claim implies Assumption \ref{ass:dissipation_general}$(\m)$ by applying \eqref{eq:dissipationhelp} with $y=u$ and $z = \nabla u$.

To prove the claim set $h(\cdot, y) = - \frac{y f(\cdot, y)}{\m-1}- \frac{1}{2}\sum_{n\geq 1} |g_n(\cdot,y)|^2$. Then
\begin{align*}
\text{LHS\eqref{eq:dissipationhelp}} &= \am z\cdot z - \frac12\sum_{n\geq 1} |b_n \cdot z|^2 + F(\cdot, y) \cdot z  - \sum_{n\geq 1} b_n \cdot z g_n(\cdot,y)   -h(y)
\\ & \stackrel{(i)}{\geq} \nu|z|^2 - |z|\, |F(\cdot, y)| - |z|\sum_{n\geq 1} |b_n| \,  |g_n(\cdot,y)| -h(y)
\\ & \stackrel{(ii)}{\geq} \varepsilon|z|^2 + \frac{4}{\nu-\varepsilon} \Big(|F(\cdot, y)|+\sum_{n\geq 1} |b_n| \,  |g_n(\cdot,y)|\Big)^{2} - h(y)
\\ & \stackrel{\eqref{eq:dissipationIsystem}}{\geq} \varepsilon |z|^2 - M(|y|^2+1)
\end{align*}
where in $(i)$ we used Assumption \ref{ass:reaction_diffusion_global}\eqref{it:ellipticity_reaction_diffusion} and in $(ii)$ we used $ab\leq (\nu-\varepsilon)a^2+\frac{1}{4(\nu-\varepsilon)}b^2$.
\end{proof}

In case $\ell=1$, then clearly \eqref{eq:dissipationIsystem} implies the existence of $M>0$ such that, a.e.\  in $[0,\infty)\times\Omega\times \T^d$,
\begin{align}\label{eq:dissneccon}
\frac{y f(\cdot, y)}{\m-1} +\frac12 \sum_{n\geq 1} |g_n(\cdot, y)|^2 \leq M(|y|^2+1), \ \ y\in \R.
\end{align}
In the next example, we will see that $f_i$ determines the admissible growth of $F_i$ and $g_{n,i}$ in order to satisfy Assumption \ref{ass:dissipation_general}.

\begin{example}[Strongly dissipative $f$]
\label{ex:strong_dissipation}
Let $h>1$ and $\m \geq 2$. Suppose that Assumption \ref{ass:reaction_diffusion_global} holds.
In the study of reaction-diffusion equations, the following dissipative condition on $f$ is often employed:

There exist $N_0, N_1>0$ such that, a.e.\ in $[0,\infty)\times\Omega\times \T^d$ and for all $i\in \{1,\dots,\ell\}$,
\begin{equation}
\label{eq:dissipative_nonlinearity}
f_i(\cdot,y)y_i\leq - N_0 |y_i|^{h+1}+ N_1 (|y|^2+1), \ \ y\in \R^{\ell},
\end{equation}
cf.\ \cite[Examples 4.2 and 4.5]{KvN12}.
In particular, \emph{odd} polynomial nonlinearities with negative leading coefficients satisfy \eqref{eq:dissipative_nonlinearity}, e.g.\
the Allen-Cahn nonlinearity: $f(u)=u-u^3$ with $\ell=1$.

In case $f_i$ satisfies \eqref{eq:dissipative_nonlinearity}, condition \eqref{eq:dissipationIsystem} holds if there exist  $\varepsilon\in (0,\nu)$ and $M>0$ such that, a.e.\ in $[0,\infty)\times\Omega\times \T^d$ and for all 
$y\in \R$, $i\in \{1,\dots,\ell\}$,
\begin{align}\label{eq:condFblemmaex}
\frac{1}{4(\nu-\varepsilon)}\Big(|F_i(\cdot,y)|
+\sum_{n\geq 1} |b_{n,i}(\cdot)|\, |g_{n,i}(\cdot,y)|\Big)^2 &+\frac12\|(g_{n,i}(\cdot,y))_{n\geq 1}\|_{\ell^2}^2
\\
\nonumber
 & \leq M (1+|y|^{2}) +\frac{N_0}{\m-1}|y_i|^{h+1}.
\end{align}
This shows that $F_i$ and $g_i$ can have growth of order $|y|^{\frac{h+1}{2}}$ for $|y|\to \infty$, but with a restriction on the constant which depends on the constant $N_0$ appearing in \eqref{eq:dissipative_nonlinearity}.

In particular, a standard Young's inequality argument shows that \eqref{eq:condFblemmaex} holds if there exist $M_1,M_2>0$ and $\delta\in (0,h+1)$ such that a.e.\ in $[0,\infty)\times\Omega\times\T^d$  and for all $y\in \R$, $i\in \{1,\dots,\ell\}$,
\begin{equation}
\label{eq:condFblemmaex2}
|F_i(\cdot,y)|^2+\|\rnoise_i(\cdot,y)\|_{\ell^2}^2
\leq M_1 (1+|y|^2)+M_2 |y_i|^{h+1-\delta}.
\end{equation}

Finally, if one can prove that the local solution provided by \eqref{t:reaction_diffusion_global_critical_spaces} takes values in $S\subseteq \R^{\ell}$ for some $S$, then Assumption \ref{ass:dissipation_general}$(\m)$ holds for $S$-valued functions if \eqref{eq:condFblemmaex} and \eqref{eq:condFblemmaex2} holds only for $y\in S$. In particular, many even polynomials satisfy \eqref{eq:dissipative_nonlinearity} for $y\geq 0$, e.g.\ for $\ell=1$ and the \emph{logistic} nonlinearity $f(u)=u(1-u)$.
\end{example}

Using smoothness assumptions on the coefficients and the nonlinearity, we formulate a simple sufficient condition for Assumption \ref{ass:dissipation_general}$(\m)$ in the case of scalar equations. A version of the results below also holds in the system case as well, but will not state this explicitly.

\begin{lemma}[Pointwise $L^\m$-coercivity/dissipativity with smoothness -- Scalar case]
\label{lem:dissipationII}
Assume that $\ell=1$ and that Assumption \ref{ass:reaction_diffusion_global} holds. Let $\m \geq 2$. Suppose that a.s.\ for every fixed $t\in [0,\infty)$ and $u\in \R$, for every $\phi\in \{F,g\}$, $x\mapsto \phi(t,x,u)$ and $x\mapsto \nabla_x \phi(t,x,u)$ are continuous (and thus periodic), and that $\div(b)\in L^\infty([0,\infty)\times \Omega\times\T^d;\ell^2)$.
If there exists an $M>0$ and $\varepsilon>0$ such that for all $u\in \R$, a.e.\ in $[0,\infty)\times\Omega\times \T^d$
\begin{equation}\label{eq:dissipationII}
\begin{aligned}
\frac{u f(\cdot, u)}{\m-1}+  \frac{1}{2} \|(g_n(\cdot,u))_{n\geq 1}\|_{\ell^2}^2  +
 \varepsilon \sup_{|y|\leq |u|}|\div_x F(\cdot,x,y)|^2&
\\  + \varepsilon\sum_{k\in \{0,1\}} \sup_{|y|\leq |u|} \|(\nabla_x^k g_n(\cdot, x,y))_{n\geq 1}\|_{\ell^2}^2
&\leq M(1+|u|^2),
\end{aligned}
\end{equation}
then Assumption \ref{ass:dissipation_general}$(\m)$ is satisfied. Moreover, if the above only holds for $S$-valued functions $u$ with $S\subseteq \R$, then Assumption \ref{ass:dissipation_general}$(\m)$ holds for all $S$-valued functions.
\end{lemma}
Note that in many cases the $\varepsilon$-terms vanish, and in that case, the condition reduces to \eqref{eq:dissneccon}. In particular, the proof below shows that the last term on the LHS\eqref{eq:dissipationII} can be omitted if $b=0$ or ($\div(b) = 0$ in $\D'(\T^d)$, and $g$ is $x$-independent). In order to keep the focus on global well-posedness, the tedious but elementary proof of Lemma \ref{lem:dissipationII} is given in Appendix \ref{app:proofs_lemmas}.

From Lemma \ref{lem:dissipationII} and the text below it we obtain the following:
\begin{example}
Let $\m\geq 2$. Suppose that $F$ is $x$-independent, and suppose that $b=0$ or ($\div(b) = 0$ and $g$ is $x$-independent).
Then Assumption \ref{ass:dissipation_general} holds provided there exist constants $N_0, N_1, M\geq 0$ such that, for all $y\in \R$ and a.e.\ on $\R_+\times \O$,
\begin{equation}
\label{eq:fyyN0}
\begin{aligned}
f(\cdot,y)y  \leq - N_0 |y|^{h+1}+ N_1 (|y|^2+1)
\ \  \text{ and } \ \
\|g(t,y)\|_{\ell^2}^2 \leq M (1+|y|^2)+ \frac{2N_0}{\m-1} |y|^{h+1}.
\end{aligned}
\end{equation}
As in Example \ref{ex:strong_dissipation}, if the local solution provided by \eqref{t:reaction_diffusion_global_critical_spaces} is positive, then it suffices to consider $y\geq 0$ in the above conditions.
\end{example}

\begin{remark}
\label{r:comparison}
\
\begin{enumerate}[(a)]
\item Many authors consider the ``one-sided Lipschitz condition'' (see \cite{DG15_boundedness, MaRo10, NeeSis,S21}):
\[(y_1-y_2)(f(y_1)- f(y_2))\leq C(y_1-y_2)^2, \ \ \ y_1, y_2\in \R.\]
Clearly, it implies \eqref{eq:fyyN0} with $N_0=0$.

\item\label{it:salins} The dissipative condition \eqref{eq:fyyN0} appears for instance in \cite{C03,F91,KvN12, S21_dissipative,WML20} and the references therein. In these works, $g$  usually has linear growth. An exception is \cite{S21_dissipative}, where the higher growth order of $\|g\|_{\ell^2}^2$ is allowed as long as it is strictly lower than that of $f$.
An important difference with \cite{S21_dissipative} is that we consider transport noise, and therefore it is assumed to be trace-class. Motivations for gradient noise are given in \cite{AVreaction-local}.

\item The use of positivity to check blow-up criteria for SPDEs with polynomial nonlinearities of even degree and negative leading coefficients seems to be new.
\end{enumerate}
\end{remark}

\subsection{Further applications: Coagulation dynamics and symbiotic Lotka-Volterra model}
\label{ss:applications_fully_dissipative}

\subsubsection{A model from coagulation dynamics}
\label{ss:coagulation}
Here we consider a model from \cite{FH21_coagulation} on coagulation dynamics. Coagulation is used in chemistry and rain formation. It is a chemical process used to neutralize charges and form a gelatinous mass. In \cite{FH21_coagulation} (probabilistically) weak solutions are obtained and pathwise uniqueness is proved
on $\R^d$. Below we will show that, if $\R^d$ is replaced by $\T^d$, then our setting can be used to obtain strong global well-posedness directly, and moreover deduce high order regularity of the solution.
More precisely, we consider the stochastic coagulation equations (with transport noise and superlinear diffusion), i.e.\ \eqref{eq:reaction_diffusion_system} with $\ell\geq 1$,
\begin{equation}
\label{eq:coagulation}
f_i(y) = \sum_{j=1}^{i-1} y_j y_{i-j} -2 \sum_{j=1}^{\ell} y_j y_i \quad  \text{ for } y\in \R^{\ell} \text{ and } i\in \{1,\dots,\ell\} ,
\end{equation}
$F_i\equiv 0$ and $(a_i,b_i,g_i)$ is as in Assumption \ref{as:gtermcoagulation} below. In the above we set $\sum_{j=1}^{0}:=0$. The unknown $u_i:[0,\infty)\times\O\times \T^d\to \R^{\ell}$ describes the density of particle $i\in \{1,\dots,\ell\}$.
As explained in \eqref{eq:Itostratintro}, the Stratonovich formulation is covered as well. A difficulty in proving global well-posedness for \eqref{eq:coagulation} is that the quadratic nonlinearity $f$ does not satisfy the usual one-sided Lipschitz condition (see Remark \ref{r:comparison}).

Consider the following conditions on the parameters and coefficients:
\begin{assumption}\label{as:gtermcoagulation}
Let $d\geq 1$, $\delta\in [1, 2)$, $h=2$, $q = \max\{d/2,2\}$, and let $p\geq \max\{\frac{2}{2-\delta},q\}$.
Suppose that Assumption \ref{ass:reaction_diffusion_global}$(p, q, h, \delta)$ holds with $F=0$.
\begin{enumerate}[{\rm(1)}]
\item\label{it:gtermcoagulation1}
Suppose that for all $i\in \{1,\dots,\ell\}$, $n\geq 1$, $y=(y_{i})_{i=1}^{\ell}\in [0,\infty)^{\ell}$ and a.e.\ on $\R_+\times \O$,
\[g_{n,i}(\cdot,y_1,\dots,y_{i-1},0,y_{i+1},\dots,y_{\ell}) = 0.\]

\item\label{it:gtermcoagulation2} Suppose that there is a $\lambda\in [0,2)$  such that a.e.\ on $\R_+\times \O$,
\[\limsup_{y\in [0,\infty)^\ell, |y|\to \infty} \frac{q-1}{|y|_2^2|y|_1} \esssup_{(t,\omega,x)} \sum_{n\geq 1} \sum_{i=1}^\ell |g_{n,i}(t,\omega,x, y)|^2  \leq  \lambda.\]
\end{enumerate}
\end{assumption}
In the above, we used the notation $|y|_p = \big(\sum_{j=1}^\ell |y_j|^p\big)^{1/p}$. The function $g$ does not appear in \cite{FH21_coagulation}, but under the above hypothesis, it can be added without much additional difficulty.  The smallness condition of Assumption \ref{as:gtermcoagulation}\eqref{it:gtermcoagulation2} on $g$ is often satisfied with $\lambda=0$.  Note that for $d\in \{1, 2, 3, 4\}$ one can take $q=2$.

Assumption \ref{ass:reaction_diffusion_global} is a condition on the coefficients $(a,b)$ and the nonlinearity $g$. Clearly, $f$ is quadratic and therefore, satisfies the local Lipschitz requirement with $h=2$.

Our main result concerning the equation \eqref{eq:coagulation} is the following well-posedness results.

\begin{theorem}[Global well-posedness -- Coagulation equations]
\label{thm:coagulation}
Suppose that Assumption \ref{as:gtermcoagulation} holds and set $\kappa_{\crit} = p(1-\frac{\delta}{2})-1$.
Then for every $u_{0}\in L^0_{\F_0}(\Omega;L^q(\T^d;\R^\ell))$ with $u_{0}\geq 0$ (componentwise), there exists a (unique) global $(p,\kappa_{\crit},q,\delta)$-solution $u:[0,\infty)\times\Omega\times\T^d\to [0,\infty)^{\ell}$ to the stochastic coagulation equations (as described near \eqref{eq:coagulation}). Moreover, a.s.,
\begin{align*}
u&\in C([0,\infty);B^0_{q,p}(\Tor^d;\R^\ell)),\\
 u&\in H^{\theta,r}_{\rm loc}(0,\infty;H^{1-2\theta,\zeta}(\Tor^d;\R^{\ell}))  \ \ \forall \theta\in  [0,1/2), \ \forall   r,\zeta\in (2,\infty),\\
u&\in C^{\theta_1,\theta_2}_{\rm loc}((0,\infty)\times \Tor^d;\R^{\ell})
\ \ \forall \theta_1\in  [0,1/2), \ \forall  \theta_2\in (0,1).
\end{align*}
Moreover, the energy estimates \eqref{eq:aprioribounds1}  and \eqref{eq:aprioribounds2} hold with $\zeta = q$. Moreover, one can take $s=0$ if for some $\varepsilon>0$, $u_0\in B^{\varepsilon}_{\zeta,\infty}(\T^d;\R^{\ell})$ a.s., or $u_0\in L^2(\T^d;\R^{\ell})$ a.s.\ and $\zeta =p=q= 2$.
\end{theorem}

Note that for $d\leq 4$ we can take $p=q=2$ and $\s=1$ (thus $\a_{\crit}=0$).

\begin{proof}
First consider $d\geq 4$.
Then Assumption \ref{ass:admissibleexp}$(p,q,h,\s)$ holds.
By Theorem \ref{t:reaction_diffusion_global_critical_spaces} there exists a $(p,\kappa_{\crit}, \delta, q)$-solution $(u,\sigma)$. Moreover, \eqref{eq:regularity_u_reaction_diffusion_critical_spaces}, \eqref{eq:reaction_diffusion_H_theta} and \eqref{eq:reaction_diffusion_C_alpha_beta} imply the required regularity if we can check that $\sigma=\infty$.

Since $u_0\geq 0$ and $g$ satisfies Assumption \ref{as:gtermcoagulation}\eqref{it:gtermcoagulation1}, it follows from Proposition  \ref{prop:positivity} that a.s.\ $u\geq 0$ on $[0,\sigma)$.
To prove global existence by Theorem \ref{t:global_well_posedness_general}  it is enough to check Assumption \ref{ass:dissipation_general}$(q)$ for $[0,\infty)^d$-valued functions. By Lemma \ref{lem:dissipationI} it is enough to show
\begin{align}\label{eq:toprovecoagulation}
\sum_{i=1}^\ell N_i^{\nu/2}(y)\leq C, \ \ \text{for all $y\in [0,\infty)^d$}.
\end{align}
Moreover, by the properties of $f$ and $g$ it suffices to prove \eqref{eq:toprovecoagulation} for $|y|_2\geq R$ for some $R>0$. By Assumption \ref{as:gtermcoagulation}\eqref{it:gtermcoagulation2} we can find $\eta>0$ and $R>0$ such that, for all $|y|_2\geq R$,
\begin{align}\label{eq:growthestgcoag}
\Big(\frac12+\eta\Big)\sum_{n\geq 1} \sum_{i=1}^\ell |g_{n,i}(\cdot,y)|^2  \leq  \frac{|y|_2^2|y|_1}{q-1}.
\end{align}

Since $\sum_{j=1}^{i-1} y_j y_{i-j} \leq \frac12\sum_{j=1}^{i-1} y_j^2  + y_{i-j}^2 \leq |y|^2_2$,
we obtain, for $y\in [0,\infty)^{\ell}$,
\begin{align*}
\sum_{i=1}^\ell y_i f_i(\cdot, y) \leq \sum_{i=1}^\ell y_i \Big[|y|^2 - 2 \sum_{j=1}^{\ell} y_j y_i\Big] = -|y|_2 |y|_1.
\end{align*}
Therefore, for all $|y|_2\geq R$,
\begin{align*}
\sum_{i=1}^\ell N_i^{\nu/2}(\cdot, y) &\leq -\frac{|y|_2^2 |y|_1}{q-1} + \frac12\sum_{i=1}^\ell \sum_{n\geq 1} |g_{n,i}(\cdot,y)|^2 + \frac{8}{\nu}\sum_{i=1}^\ell\sum_{n\geq 1} |b_{n,i}| |g_{n,i}(\cdot, y)|
\\ & \leq -\frac{|y|_2^2 |y|_1}{q-1} + \big(\frac12+\eta\big)\sum_{i=1}^\ell \sum_{n\geq 1} |g_{n,i}(\cdot,y)|^2 + C_{\eta,\nu,b}\leq C_{\eta,\nu,b},
\end{align*}
where we used \eqref{eq:growthestgcoag}. This proves \eqref{eq:toprovecoagulation}, and the global existence follows.

In case $d\in \{1, 2, 3\}$, we can introduce dummy variables to reduce to the case $d=4$  (see \cite[Remark 2.2(d)]{AVreaction-local}).
\end{proof}

\subsubsection{Symbiotic Lotka-Volterra equations}\label{subsec:Symbiotic Lotka-Volterra}
Here we consider the \emph{symbiotic} Lotka-Volterra equations, which are used in the case species actually benefit from each other. This differs from the classical predatory-prey system, which is mathematically more complicated and will be discussed in Subsection \ref{ss:Lotka_Volterra}.
The stochastic symbiotic Lotka-Volterra equations (with transport noise and superlinear diffusion) are of the form \eqref{eq:reaction_diffusion_system} with $\ell=2$,
\begin{equation}
\label{eq:symlotka_volterra}
f_i(\cdot,y)=-y^2_i + \chi_{i} y_1y_2+\lambda_i y_i \ \ \text{ for }y\in \R^2,\  i\in\{1,2\},
\end{equation}
$F_i\equiv 0$,
and $(a_i,b_i,g_i)$ is as in Assumption \ref{ass:condsymb}.  The unknowns $u_1,u_2:[0,\infty)\times \O\times \Tor^d \to \R$ are the population of the species.
The reader is referred to \cite{DLGS00,K92} for some comments on the deterministic model.
The presence of transport noise models the small-scale randomness in migration processes (see \cite[Subsection 1.3]{AVreaction-local}). The additional terms $g_{n,i}$ can model further random forces acting on the system.
A schematic idea of the model is depicted in Figure \ref{fig:lotka_volterra_symbotic} where $\lambda_i^{\prime}-\lambda_i^{\prime\prime}=\lambda_i$ for $i\in \{1,2\}$.
For simplicity, we take a unitary rate of self-interaction, but this is not essential.
\begin{figure}[h!]
\centering

\tikzstyle{io} = [ellipse,text width=2cm, minimum height=1cm, text centered, draw=black]
\tikzstyle{nn} = [rectangle,text width=1.8cm, minimum height=0.5cm, minimum width=2.5cm, text centered, draw=black]
\tikzstyle{nncircle} = [circle,text width=0.8cm, minimum height=0.5cm, minimum width=0.9cm, text centered, draw=black]
\tikzstyle{arrow} = [thick,->,>=stealth]

\begin{tikzpicture}[scale=0.95, transform shape, node distance=2.1cm]
\node (in1) [nn] {\small Species $u_2$};
\node (in2) [nn,below of = in1] {\small Species $u_1$};
\node (in3) [nncircle,left of = in1,yshift=-1cm,xshift=-2.8cm] {\small Food};
\node (in4) [nncircle,right of = in1,yshift=-1cm,xshift=2.8cm] {\small Death};

\path[thick,<-,>=stealth,shift left=0.2cm]
    (in1) edge node [left, xshift=-0.4cm]{$\chi_{1}u_{1}u_{2}\geq0$} (in2)
    (in2) edge node [right, xshift=0.4cm]{$\chi_{2}u_{1}u_{2}\geq0$}(in1);
\path [thick,<-,>=stealth](in1) edge [out=50,in=120,looseness=10] node [right, xshift=0.3cm]{$-u_{2}^{2}\leq0$}(in1);
\path [thick,<-,>=stealth](in2) edge [out=-50,in=-120,looseness=10] node [right, xshift=0.3cm]{$-u_{1}^{2}\leq0$} (in2);
\draw [thick,<-,>=stealth] (in1.east) -- node[above right]{$-\lambda_{2}^{\prime\prime}u_{2}\leq0$} ([yshift=0.1cm]in4.west);
\draw [thick,<-,>=stealth] (in2.east) -- node[below right]{$-\lambda^{\prime\prime}_{1}u_{1}\leq0$} ([yshift=-0.1cm]in4.west);
\draw [arrow] ([yshift=-0.1cm]in3.east) -- node[below left]{$\lambda_{1}^{\prime}u_{1}\geq0$} (in2.west);
\draw [arrow] ([yshift=0.1cm]in3.east) -- node[above left]{$\lambda_{2}^{\prime}u_{2}\geq0$} (in1.west);

\end{tikzpicture}
\caption{Scheme for the symbiotic Lotka-Volterra equations.}
\label{fig:lotka_volterra_symbotic}
\end{figure}
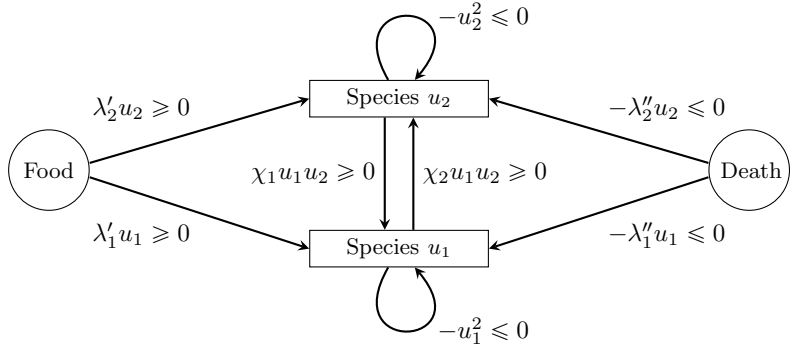

In order to state our global well-posedness result for $\varepsilon_1, \varepsilon_2>0$ we take $N_i^{\varepsilon_i}$ as in Lemma \ref{lem:dissipationI}:
\[N_i^{\varepsilon_i}(\cdot, y) = -y_i^3 +  \chi_{i} y_i y_1 y_2  +\frac12 \sum_{n\geq 1} |g_{n,i}(\cdot, y)|^2 + \frac{4}{\nu_i-\varepsilon_i} \Big(\sum_{n\geq 1} |b_{n,i}| \,  |g_{n,i}(\cdot,y)|\Big)^{2}.\]

\begin{assumption}\label{ass:condsymb}
Let $d\in \{1, 2, 3, 4\}$, $p=q=2$, $\kappa=\kappa_{\crit}=0$, $h=2$, $\delta = 1$, and Assumption \ref{ass:reaction_diffusion_global}$(p, q, h, \delta)$ holds with $F=0$ and $\ell=2$. Suppose
\begin{enumerate}[{\rm(1)}]
\item\label{it:condsymb1} The mappings $\lambda_1,\lambda_2, \chi_1, \chi_2:\R_+\times\O\times \Tor^d\to \R$ are bounded $\Progress\otimes \Borel(\Tor^d)$-measurable maps.
\item\label{it:condsymb3} There exist $\alpha_i>0,\varepsilon_i\in (0,\nu_i)$ for $i\in \{1, 2\}$ and  $M>0$ such that
\begin{align*}
 \sum_{i\in \{1, 2\}} \alpha_i N_i^{\varepsilon_i}(\cdot,y) \leq   M(|y|^2+1), \ \ y\geq 0.
\end{align*}
\item\label{it:condsymb4} $g_{n,1}(\cdot,0,y_2)=g_{n,2}(\cdot,y_1,0)=0$ for all $y\in \R^2$ and $n\geq 1$.
\end{enumerate}
\end{assumption}
Note that Assumption \ref{ass:condsymb}\eqref{it:condsymb4} is used to obtain the positivity of the solution $u$.

Now suppose that $g$ has at most linear growth. Using $v^2 w \leq \frac{2}{3}v^3  + \frac13 w^3$ one can check that
\[-y_1^3 +  \chi_{1} y_1 y_1 y_2  -y_2^3 +  \chi_{2} y_1y_2^2 \leq -(1-\tfrac{2}{3} \chi_1 - \tfrac13 \chi_2) y_1^3  -(1-\tfrac{1}{3} \chi_1 - \tfrac23 \chi_2) y_2^3.\]
Thus, Assumption \ref{ass:condsymb}\eqref{it:condsymb3} holds with $\alpha_1 = \alpha_2 = 1$ if $\chi_1,\chi_2\in [0,1]$. If $\chi_1,\chi_2\in [0,1)$ and can even take allow $g$ which has growth $\lesssim (1+|y|)^{\frac32-\gamma}$ for some $\gamma>0$. One can also allow large $\chi_1>1$ if $\chi_2$ is small enough, which follows by choosing $0<\alpha_1<\alpha_2$.

\begin{theorem}[Global well-posedness -- Symbiotic Lotka-Volterra]\label{thm:SymbLotka-Volterra}
Let $1\leq d\leq 4$ and Assumption \ref{ass:condsymb} be satisfied.
Then for every $u_{0}\in L^0_{\F_0}(\Omega;L^2( \T^d;\R^2))$ with $u_{0}\geq 0$ (componentwise), there exists a (unique) global $(2,0,1,2)$-solution $u:[0,\infty)\times\Omega\times\T^d\to [0,\infty)^2$ to the stochastic symbiotic Lotka-Volterra equations (as described near \eqref{eq:symlotka_volterra}). Moreover, a.s.,
\begin{align*}
u&\in L^2_{\loc}([0,\infty);H^{1}(\T^d;\R^2))\cap C([0,\infty);L^2(\Tor^d;\R^2)),\\
 u&\in H^{\theta,r}_{\rm loc}(0,\infty;H^{1-2\theta,\zeta}(\Tor^d;\R^{2}))  \ \ \forall \theta\in  [0,1/2), \ \forall r,\zeta\in (2,\infty),\\
u&\in C^{\theta_1,\theta_2}_{\rm loc}((0,\infty)\times \Tor^d;\R^{2})
\ \ \forall \theta_1\in  [0,1/2), \ \forall \theta_2\in (0,1).
\end{align*}
Finally, the energy estimates \eqref{eq:aprioribounds1}-\eqref{eq:aprioribounds2} hold for $\zeta=2$.
\end{theorem}

\begin{proof}
First, consider $d=4$. Assumptions \ref{ass:reaction_diffusion_global}$(2,2,2,1)$ and \ref{ass:admissibleexp}$(2,2,2,1)$ hold. From Theorem \ref{t:reaction_diffusion_global_critical_spaces} we obtain a local $(2,0,1,2)$-solution $(u,\sigma)$ which satisfies the regularity results \eqref{eq:reaction_diffusion_H_theta} and \eqref{eq:reaction_diffusion_C_alpha_beta}. Moreover, Proposition \ref{prop:positivity} implies $u\geq 0$.

It remains to check global existence and the a priori estimates. By Assumption \ref{ass:condsymb}\eqref{it:condsymb3} we can apply Lemma \ref{lem:dissipationI} to verify Assumption \ref{ass:dissipation_general}$(2)$. Therefore, Theorem \ref{t:global_well_posedness_general} gives $\sigma=\infty$.

In case $1\leq d\leq 3$, we can introduce dummy variables to reduce to the case $d=4$.
\end{proof}

\subsection{Proof of Theorem \ref{t:global_well_posedness_general} via energy estimates}
\label{ss:global_general}
To prove Theorem \ref{t:global_well_posedness_general} we check the blow-up criterion of Theorem \ref{thm:blow_up_criteria}. The key ingredients for this are the following energy estimates for $(p,\a_{\crit},q,\s)$-solutions to \eqref{eq:reaction_diffusion_system_intro}, where $\a_{\crit}$ is as in Theorem \ref{t:reaction_diffusion_global_critical_spaces}.

\begin{lemma}[Energy estimates]
\label{l:high_integrability_reaction_diffusion}
Let the assumptions of Theorem \ref{t:reaction_diffusion_global_critical_spaces} be satisfied. Let $(u,\sigma)$ be the local $(p,\a_{\crit},\s,q)$-solution to \eqref{eq:reaction_diffusion_system}. Suppose that Assumption \ref{ass:dissipation_general}$(\m)$ holds, where $\m\geq \frac{d(h-1)}{2}\vee 2$ is fixed.
Then, for all $0<s<T<\infty$ and $i\in \{1,\dots,\ell\}$,
\begin{align}
\label{eq:energy_estimates_lemma_reaction_diffusion}
\sup_{t\in [s,\sigma\wedge T)}\|u_i(t)\|_{L^{\m}}^{\m}&<\infty  \ \ \text{ a.s.\ on }\{\sigma>s\},\\
\label{eq:energy_estimates_lemma_reaction_diffusion_gradient}
\int_{s}^{\sigma\wedge T}\int_{\Tor^d}
|u_i|^{\m-2} |\nabla u_i|^{2}\,\dd x\, \dd t&<\infty \ \ \text{ a.s.\ on }\{\sigma>s\}.
\end{align}
Moreover, for all $\lambda\in (0,1)$ there exist $N_0, N_{0,\lambda}>0$ such that for all $0<s<T<\infty$, $L\geq 0$  and $i\in \{1,\dots,\ell\}$,
\begin{align}
\label{eq:energylambda1} \sup_{t\in [s,T]}\E \Big[\one_{[s,\sigma)}(t) \one_{\Gamma} \|u_i(t)\|_{L^{\m}}^{\m } \Big]+ \E \int_{s}^{\sigma\wedge T}\int_{\Tor^d} \one_{\Gamma} |u_i|^{\m-2} |\nabla u_i|^{2}\,\dd x\, \dd t&\\
\nonumber
\leq N_0\Big(1+\E\one_{\Gamma}\|u(s)\|_{L^{\m}}^{\m}\Big),&\\
\label{eq:energylambda2} \E \sup_{t\in [s,\sigma\wedge T)} \one_{\Gamma} \|u_i(t)\|_{L^{\m}}^{\m \lambda} 
+ \E \Big|\int_{s}^{\sigma\wedge T}\int_{\Tor^d} \one_{\Gamma} |u_i|^{\m-2} |\nabla u_i|^{2}\,\dd x\, \dd t\Big|^{\lambda}&
\\ \nonumber \leq N_{0,\lambda}\Big(1+\E\one_{\Gamma}\|u(s)\|_{L^{\m}}^{\m \lambda}\Big),&
\end{align}
where $\Gamma = \{\sigma>s,\, \|u(s)\|_{L^{\m}}\leq L\}$. Furthermore, one can take $s=0$ in \eqref{eq:energylambda1}-\eqref{eq:energylambda2} if for some $\varepsilon>0$, $u_0\in B^{\varepsilon}_{\zeta,\infty}(\T^d;\R^{\ell})$ a.s.,\ or $u_0\in L^2(\T^d;\R^{\ell})$ a.s.\ and $\zeta = p=q=2$ 
\end{lemma}

By the regularity results \eqref{eq:reaction_diffusion_H_theta}-\eqref{eq:reaction_diffusion_C_alpha_beta} and  $s>0$, the quantities in \eqref{eq:energy_estimates_lemma_reaction_diffusion}-\eqref{eq:energy_estimates_lemma_reaction_diffusion_gradient} are well-defined.

Before giving the proof of Lemma \ref{l:high_integrability_reaction_diffusion}, we show how it implies Theorem \ref{t:global_well_posedness_general}. The proof of Lemma \ref{l:high_integrability_reaction_diffusion} will be given in Subsection \ref{ss:reaction_diffusion_energy_estimate}.

\begin{proof}[Proof of Theorem \ref{t:global_well_posedness_general}]
We only consider the case $\ell=1$ as the general case is analogous. As above, we omit the subscript $i=1$.
Let $(u,\sigma)$ be $(p,\a_{\crit},\s,q)$-solution to \eqref{eq:reaction_diffusion_system} provided by Theorem \ref{t:reaction_diffusion_global_critical_spaces}. It remains to show that
$\sigma=\infty$ a.s.

{\em Step 1: Case $\m>\frac{d(h-1)}{2}\vee 2$}. Fix $0<s<T<\infty$.
By Lemma \ref{l:high_integrability_reaction_diffusion},
\begin{equation}
\label{eq:not_sharp_case_u_integrability}
u\in L^{\infty}(s,\sigma\wedge T;L^{\m})\quad \text{ a.s.\ on }\{\sigma>s\},
\end{equation}
and $u:[s,\sigma\wedge T)\to L^{\m}$ is continuous by \eqref{eq:reaction_diffusion_C_alpha_beta}. Let $h_0:=\max\{h,1+\frac{4}{d}\}$.
Applying  \eqref{eq:not_sharp_case_u_integrability} and then Theorem \ref{thm:blow_up_criteria}\eqref{it:blow_up_not_sharp_L} with $\zeta_1=\m$, $\delta_0 \in (1,\frac{h_0+1}{h_0}]$ small enough, and $p_0$ large enough, it follows that
\begin{equation}
\label{eq:combining_blow_up_and_energy_estimates}
\P(s<\sigma<T)
=
\P\Big(s<\sigma<T,\, \sup_{t\in [s,\sigma)} \|u(t)\|_{L^{\m}}<\infty\Big)=0.
\end{equation}
Since $\sigma>0$ a.s.\ by Theorem \ref{t:reaction_diffusion_global_critical_spaces}, letting $s\downarrow 0$ we obtain $\P(\sigma<T)=0$.
Letting $T\to \infty$, we find that $\sigma=\infty$ a.s.

{\em Step 2: Case $\m= \frac{d(h-1)}{2}\vee 2$}. In this case, we apply the sharper result of Theorem \ref{thm:blow_up_criteria}\eqref{it:blow_up_sharp_L} with a specific choice of the parameters $(p_0,q_0,\s_0)$, and $h_0$ as in Step 1.
Let $\zeta_0 = \frac{d}{2}(h_0-1)$. Then one can check that $\zeta_0 = \m$.
Choose $\s_0\in (1,\frac{h_0+1}{h_0})$. For $\varepsilon>0$ fixed, let
$q_0 = \frac{d(h_0-1)}{2}(1+\varepsilon)$.
For $\varepsilon>0$ small enough one has $q_0<\frac{d(h_0-1)}{h_0+1-\reg_0(h_0-1)}$. Moreover,
$q_0>\zeta_0$, and $q_0>2$. Define $p_0$ by
$\frac{2}{p_0}=\frac{2}{h_0-1}-\frac{d}{q_0}$. Then $p_0\to \infty$ as $\varepsilon\downarrow 0$. From this one can see that for $\varepsilon>0$ small enough
\begin{equation*}
p_0>\big(\frac{2}{\s_0-1}\big)\vee q_0
\qquad \text{ and }\qquad
 \frac{1}{p_0}+\frac{1}{2}\Big(\s_0+\frac{d}{q_0}\Big)\leq \frac{h_0}{h_0-1}.
\end{equation*}
In particular, Assumptions \ref{ass:reaction_diffusion_global}$(p_0,q_0,h_0,\s_0)$ and \ref{ass:admissibleexp}$(p_0,q_0,h_0,\s_0)$ hold.

We claim that it is enough to show that for all $0<s<T<\infty$ and a.s.\ on $\{\sigma>s\}$
\begin{align}
\label{eq:sup_B_L_sharp}
u\in L^{\infty}(s,\sigma\wedge T;L^{\m_0})  \quad \text{and} \quad  u \in L^{p_0}(s,\sigma\wedge T;L^{q_0}) .
\end{align}
Indeed, if \eqref{eq:sup_B_L_sharp} holds a.s.\ on $\{\sigma>s\}$, then the claim follows from Theorem \ref{thm:blow_up_criteria}\eqref{it:blow_up_sharp_L}, and an argument similar to the one used in \eqref{eq:combining_blow_up_and_energy_estimates} and below it.
Thus, it remains to prove \eqref{eq:sup_B_L_sharp}. To this end, fix $0<s<T<\infty$. By Lemma \ref{l:high_integrability_reaction_diffusion} and $\m=\zeta_0$, a.s.\ on $\{\sigma>s\}$
\begin{align}
\label{eq:sharp_case_u_integrability}
u&\in L^{\infty}(s,\sigma\wedge T;L^{\zeta_0}) \ \ \text{and}  \ \  |u|^{(\zeta_0-2)/2}|\nabla u| \in L^{2}(s,\sigma\wedge T;L^2).
\end{align}
Thus, the first part of \eqref{eq:sup_B_L_sharp} follows from the first part of \eqref{eq:sharp_case_u_integrability}. To show the second part of \eqref{eq:sup_B_L_sharp} we use an interpolation argument. By the chain rule and the smoothness of $u$ (see \eqref{eq:reaction_diffusion_H_theta}), a.e.\ on $ [s,\sigma)\times \{\sigma>s\}$ one has
\begin{equation*}
\frac{\zeta_0^2}{4}\int_{\Tor^d} |u|^{\zeta_0-2} |\nabla u|^2 \,\dd x
=\int_{\Tor^d} \Big|\nabla \big[|u|^{\zeta_0/2}\big]\Big|^2\,\dd x.
\end{equation*}
Thus, \eqref{eq:sharp_case_u_integrability} implies that a.s.\ on $\{\sigma>s\}$
\begin{align}
\label{eq:sharp_case_u_integrability_zeta_0}
|u|^{\zeta_0/2}&\in L^{\infty}(s,\sigma\wedge T;L^{2})\ \ \ \text{ and } \  \ \
|u|^{\zeta_0/2} \in L^{2}(s,\sigma\wedge T;H^1).
\end{align}
Standard interpolation inequalities and Sobolev embedding imply that
\begin{align*}
L^{\infty}(s,\sigma\wedge T;L^{2})\cap L^{2}(s,\sigma\wedge T;H^1)\hookrightarrow L^{2/\theta}(s,\sigma\wedge T;H^{\theta})\hookrightarrow L^{2/\theta}(s,\sigma\wedge T;L^{\xi}),
\end{align*}
where $\theta\in (0,1)$ and $\xi\in (2, \infty)$ satisfies $\theta  - \frac{d}{2} = -\frac{d}{\xi}$. Thus, by \eqref{eq:sharp_case_u_integrability_zeta_0},
\begin{align}\label{eq:normuzeta}
u\in L^{\zeta_0/\theta}(s,\sigma\wedge T;L^{\xi\zeta_0/2}) \ \text{ a.s.\ for all } \theta\in (0,1).
\end{align}
Taking $\theta = \zeta_0/p_0<1$ (here we use $p_0>q_0>\zeta_0$), we obtain
$\zeta_0/\theta = p_0$ and
\[\frac{2}{\xi\zeta_0} = \frac{2}{d\zeta_0} \frac{d}{\xi} = \frac{1}{\zeta_0} - \frac{2\theta}{d\zeta_0} = \frac{2}{d(h_0-1)}-\frac{2}{d p_0} = \frac{1}{q_0}.\]
Hence, the second part of \eqref{eq:sup_B_L_sharp} follows from \eqref{eq:normuzeta}.

{\em Step 3:}  The energy estimates \eqref{eq:aprioribounds1} and \eqref{eq:aprioribounds2} are immediate from the bounds in Lemma \ref{l:high_integrability_reaction_diffusion}.
\end{proof}

\subsection{Proof of Lemma \ref{l:high_integrability_reaction_diffusion} -- Energy estimates for \eqref{eq:reaction_diffusion_system}}
\label{ss:reaction_diffusion_energy_estimate}

The proof uses It\^o's formula for $\|u\|_{L^{\zeta}}^{\zeta}$ combined with the $L^\zeta$-coercivity of Assumption \ref{ass:dissipation_general}, and the stochastic Gronwall Lemma (see \cite[Corollary 5.4b)]{geiss2021sharp}). The validity of It\^o's formula heavily relies on the regularity of $u$ obtained in Theorem \ref{t:reaction_diffusion_global_critical_spaces}.

\begin{proof}[Proof of Lemma \ref{l:high_integrability_reaction_diffusion}]
Let $0<s<T<\infty$ be fixed. We split the proof into several steps. To keep the notation short, we consider the case $\ell=1$ in Steps 1-5. Comments on the case $\ell\geq 1$ are given in Step 6.

\emph{Step 1: Preparation -- case $\ell=1$.}
To obtain estimates leading to \eqref{eq:energy_estimates_lemma_reaction_diffusion}--\eqref{eq:energy_estimates_lemma_reaction_diffusion_gradient} we need a localization argument. Hence, we introduce stopping times $(\tau_j)_{j\geq 1}$.
To define $\tau_j$ we exploit the instantaneous regularization of $u$, see \eqref{eq:reaction_diffusion_H_theta}-\eqref{eq:reaction_diffusion_C_alpha_beta}. More precisely, for each $j\geq 1$,
we can define the stopping time $\tau_j$ by
\begin{equation*}
\tau_j=
\inf\big\{t\in [s,\sigma)\,:\, \|u(t)-u(s)\|_{C(\Tor^d)}+\|u\|_{L^2(s,t;H^{1}(\Tor^d))}\geq j\big\}\wedge T,
\end{equation*}
on $\{\sigma>s,\|u(s)\|_{C(\Tor^d)}\leq j-1\}$, and $\tau_j = s$ otherwise. Here we also set $\inf\emptyset:=\sigma\wedge T$.

By \eqref{eq:reaction_diffusion_H_theta}-\eqref{eq:reaction_diffusion_C_alpha_beta}, we have $\lim_{j\to \infty}\tau_j=\sigma$. In the following, $[s,\tau_j)\times \O$ serves as approximation of the stochastic interval $ [s,\sigma)\times \{\sigma>s\}$ on which we prove bounds.
Note that
$$
\|u\|_{ L^{\infty}((s,\tau_j)\times \T^d)}\leq 2j-1 \text{ a.s.\ on }\{\tau_j>s\}.
$$
Thus, Assumption \ref{ass:reaction_diffusion_global}\eqref{it:growth_nonlinearities} yields, for all $r\in (2,\infty)$, $i\in \{1,\dots,\ell\}$,
\begin{equation}
\label{eq:def_f_u_g_u}
\begin{aligned}
\one_{[s,\tau_j)\times \O}\big[\div(F_i(\cdot, u)) +f_i(\cdot, u)\big]&\in L^\infty(\Omega;L^{r}(s,\tau_j;H^{-1,r})),\\
\one_{[s,\tau_j)\times \O}(g_{n,i}(\cdot,u))_{n\geq 1}
&\in L^\infty(\Omega;L^{r}(s,\tau_j;L^{r}(\ell^2))).
\end{aligned}
\end{equation}
Combining \eqref{eq:def_f_u_g_u}, the stochastic maximal $L^p$-regularity estimates of \cite[Theorem 1.2]{AV21_SMR_torus} and the fact that $(u|_{[0,\tau_j)\times \O},\tau_j)$
is a local $(p,\a,\s,q)$-solution to \eqref{eq:reaction_diffusion_system}, we have
$u\in C([s,\tau_j];C(\Tor^d))$ a.s.

Fix $\Gamma\in \F_s$ such that $\Gamma\subseteq \{\sigma>s\}$.
Since it is convenient to work with processes defined on $\O\times [s,T]$ rather than $ [s,\tau_j)\times \Gamma$, we set
\begin{equation*}
u^{(j)}(t)=\one_{\Gamma} u(t\wedge \tau_j) \ \ \text{ on }\O\times [s,T].
\end{equation*}
We omit the dependence on $\Gamma$ for brevity. 
Since $(u,\sigma)$ is a $(p,\a_{\crit},\s,q)$-solution to \eqref{eq:reaction_diffusion_system}, we have
a.s.\ for all $t\in [s,T]$,
\begin{equation}\label{eq:Itoappliedto}
\begin{aligned}
u^{(j)}(t) - u^{(j)}(s) & =
\int_{s}^{t}  \one_{[s,\tau_j]\times \Gamma}\Big[\div(\am\cdot\nabla u)  + \big(  \div(F(\cdot, u)) +f(\cdot, u) \big)\Big] \,\dd r
\\ &\, + \sum_{n\geq 1}  \int_{s}^{t}  \one_{[s,\tau_j]\times \Gamma} \Big[(\bm_{n}\cdot \nabla) u+ \rnoise_{n}(\cdot,u)\Big]\,\dd w_r^n.
\end{aligned}
\end{equation}
From the definition of $\tau_j$, it is clear that
\begin{align}\label{eq:ujspacelem}
u^{(j)}\in L^2(\O;L^2(s,\tau_j;H^1)).
\end{align}
Now arguing as in the proof of \cite[Theorem 3.3]{AV22_variational} (see the argument below (4.5) there), by stochastic maximal $L^2$-regularity applied to \eqref{eq:Itoappliedto}, we also obtain that $u^{(j)}\in L^2(\O;C([s,T];L^2))$.

Let $\xi\in C^2_{\rm b}(\R)$ be such that $\xi(y) = |y|^{\m}$ for $|y|\leq 2j-1$.
By \eqref{eq:Itoappliedto} and \eqref{eq:def_f_u_g_u}, we may apply an extended form of It\^o's formula (see Lemma \ref{lem:Ito_generalized}) to $\xi(u)$ to obtain that a.s.\ for all $t\in [s,T]$,
\begin{align}\label{eq:identityIto}
\|u^{(j)}(t)\|_{L^{\m}}^{\m} = \|u^{(j)}(s)\|_{L^{\m}}^{\m} + \m(\m-1) \mathcal{D}(t) + \m \mathcal{S}(t),
\end{align}
where $\mathcal{D}$ and $\mathcal{S}$ stand for ``deterministic''  and ``stochastic'' part, respectively. They are defined by
\begin{align*}
\mathcal{D}(t) &= \int_s^t \int_{\T^d} \one_{[s,\tau_j]\times \Gamma} |u|^{\m-2} \Big(\frac{u f(\cdot, u)}{\m-1}  - \nabla u \cdot a\nabla u - \nabla u\cdot F(\cdot,u) \\ & \qquad \qquad\qquad\qquad\qquad\qquad
+\frac12 \sum_{n\geq 1} [(b_n\cdot\nabla) u  + g_n(\cdot, u)]^2   \Big) \, \dd x\,\dd r,
&
\\ \mathcal{S}(t)& = \sum_{n\geq 1} \int_s^t \int_{\T^d} \one_{[s,\tau_j]\times \Gamma}  |u|^{\m-2} u  [(b_n\cdot\nabla) u  + g_n(\cdot, u)]   \, \dd x \,\dd w^n_r.
\end{align*}
By Assumption \ref{ass:dissipation_general}$(\m)$, one has
\begin{align*}
\mathcal{D}(t)
&\leq C_0(t-s)-\theta \int_s^t \int_{\T^d}\one_{[s,\tau_j]\times \Gamma} |u|^{\m-2} |\nabla u|^2 \, \dd x\,\dd r+ M_0 \int_s^t \int_{\T^d} \one_{[s,\tau_j]\times \Gamma} |u|^{\m}  \,  \dd x \,\dd r,
\end{align*}
where $M_0 = \theta M$ and $C_0 = \theta C$. Therefore, by \eqref{eq:identityIto} for $t\in [s,T]$,
\begin{equation}\label{eq:Itoest}
\begin{aligned}
&\|u^{(j)}(t)\|_{L^{\m}}^{\m} +
\theta \m(\m-1)\int_s^t \int_{\T^d} \one_{[s,\tau_j]\times \Gamma} |u|^{\m-2} |\nabla u|^2 \, \dd x\,\dd r
\\ &\qquad
\leq\m(\m-1) C_0 (t-s) + \|u^{(j)}(s)\|_{L^{\m}}^{\m} + \m(\m-1) M_0 \int_s^t
\one_{[s,\tau_j]\times \Gamma}\|u(r)\|^{\m}_{L^{\m}}  \, \dd r +\m \mathcal{S}(t),
\end{aligned}
\end{equation}
where we note that $u = u^{(j)}$ on $[s,\tau_j]\times \Gamma$, so that $u$ can be replaced by $u^{(j)}$. The above estimate will be used to derive the assertion of the lemma in several steps.

{\em  Step 2 -- case $\ell=1$: Let $(\theta,M,C)$ be as in Assumption \ref{ass:dissipation_general}, and set $C_0=C\theta$, $M_0:=M \theta$. Then} 
\begin{equation}
\label{eq:step_1_main_estimate_proof}
\sup_{r\in [s,t]}\E\| u^{(j)}(r)\|_{L^{\m}}^{\m} \leq
e^{t\m(\m-1) M_0}\Big(\m(\m-1)C_0 t +\E[\one_{\Gamma}\| u(s)\|_{L^{\m}}^{\m}]\Big)\ \ \  \text{\emph{for all} }t\in [s,T].
\end{equation}
Note that the constants appearing on the RHS\eqref{eq:step_1_main_estimate_proof} are independent of $(j,s,\Gamma)$.

Taking expectations in \eqref{eq:Itoest} and using $\E \mathcal{S}(t) =0$ by \eqref{eq:def_f_u_g_u} and \eqref{eq:ujspacelem}, we obtain
\begin{equation}\label{eq:apriopriestIto1}
\begin{aligned}
\E\|u^{(j)}(t)\|_{L^{\m}}^{\m}
&+ \theta\m(\m-1) \E \int_s^t \int_{\T^d} \one_{[s,\tau_j]\times \Gamma} |u(r)|^{\m-2} |\nabla u|^2 \, \dd x\, \dd r
\\ & \leq \m(\m-1) C_0 (t-s)+\E\|u^{(j)}(s)\|_{L^{\m}}^{\m}
+ \m(\m-1) M_0 \int_s^t \E\|u^{(j)}(r)\|^{\m}_{L^{\m}}  \, \dd r,
\end{aligned}
\end{equation}
for all $t\in [s,T]$. Thus \eqref{eq:step_1_main_estimate_proof} follows from Gronwall's inequality applied to $t\mapsto \sup_{r\in [s,t]}\E\|u^{(j)}(r)\|^{\m}_{L^{\m}}$.

{\em Step 3  -- case $\ell=1$: There exists a constant $N\geq 1$ depending only on $(\theta,M,C,T,\m)$ such that}
\begin{align}\label{eq:aprioribothinu}
\E \int_s^{\sigma\wedge T} \one_{\Gamma} \|u(r)\|_{L^\m}^{\m}  \, \dd r
 &+
 \E \int_s^{\sigma\wedge T}\int_{\T^d}
 \one_{\Gamma} |u|^{\m-2} |\nabla u|^2 \, \dd x\,\dd r
 \leq N\big[1+\E\|\one_{\Gamma} u(s)\|_{L^{\m}}^{\m}\big].
\end{align}
To see this, note that by \eqref{eq:step_1_main_estimate_proof} and Fubini's theorem we obtain
\begin{align}\label{eq:aprioriintu}
 \E \int_s^T \|u^{(j)}(r)\|_{L^\m}^{\m}  \, \dd r
\leq \Big(\m(\m-1) C_0+\E\big[\one_{\Gamma}\|u(s)\|_{L^{\m}}^{\m}\big]\Big)  (1+T) e^{\m(\m-1)TM_0}.
\end{align}
Therefore, letting $j\to \infty$, by Fatou's lemma we find that
\begin{align*}
\E \int_s^{\sigma\wedge T} \one_{\Gamma}\| u(r)\|_{L^\m}^{\m}  \, \dd r & \leq
\liminf_{j\to \infty}
\E \int_s^{\tau_j}\one_{\Gamma} \| u(r)\|_{L^\m}^{\m}  \, \dd r
\\ &
\leq  \liminf_{j\to \infty}
\E \int_s^{T} \one_{\Gamma}\| u^{(j)}(r)\|_{L^\m}^{\m}  \, \dd r
\\ & \leq \Big(\m(\m-1) C_0+\E\big[\one_{\Gamma}\| u(s)\|_{L^{\m}}^{\m}\big]\Big) (1+T) e^{\m(\m-1)TM_0},
\end{align*}
which gives the estimate for the first term in \eqref{eq:aprioribothinu}.

From \eqref{eq:apriopriestIto1} and \eqref{eq:aprioriintu} we obtain
\begin{equation*}
 \theta\m(\m-1) \E \int_s^{\tau_j} \int_{\T^d} \one_{\Gamma} |u|^{\m-2} |\nabla u|^2 \, \dd x\,\dd r
 \leq M_1 \Big( \m(\m-1)C_0+\E\big[\one_{\Gamma}\|u(s)\|_{L^{\m}}^{\m}\big]\Big),
\end{equation*}
where $M_1 = M_0 (T+1) e^{\m(\m-1)TM_0}$. Letting $j\to \infty$ in the same way as before, we find that
\begin{equation*}
\theta\m(\m-1) \E \int_s^{\sigma\wedge T} \int_{\T^d} \one_{\Gamma} |u|^{\m-2} |\nabla u|^2 \, \dd x\,\dd r
\leq M_1 \Big( C_0+\E\big[\one_{\Gamma}\|u(s)\|_{L^{\m}}^{\m}\big]\Big),
\end{equation*}
which gives the estimate for the second term in \eqref{eq:aprioribothinu}.

{\em Step 4  -- case $\ell=1$: There exists $N_0>0$ depending only on $(\theta,C,M,T,\m, \lambda)$ such that}
\begin{equation*}
\E \Big[\one_{\Gamma} \sup_{t\in [s,\sigma\wedge T)} \|u(t)\|_{L^{\m}}^{\m \lambda} \Big]
+ \E \Big| \int_s^t \int_{\T^d} \one_{[s,\tau_j]\times \Gamma} |u|^{\m-2} |\nabla u|^2 \, \dd x\,\dd r\Big|^{\lambda}
\leq N_0\Big(1+\E\big[\one_{\Gamma}\|u(s)\|_{L^{\lambda\m}}^{\lambda\m}\big]\Big),
\end{equation*}
\emph{where $0<\lambda<1$.}
The latter estimate follows with $u$ replaced by $u^{(j)}$ by combining \eqref{eq:Itoest} with the stochastic Gronwall inequality \cite[Corollary 5.4b)]{geiss2021sharp}, and afterwards letting $j\to \infty$.

{\em Step  5 -- case $\ell=1$: Conclusions}.
By \eqref{eq:reaction_diffusion_C_alpha_beta}, for each $k\geq 1$ we set
\begin{equation*}
\Gamma_k :=\{\sigma>s,\, \|u(s)\|_{L^{\infty}}\leq k\}\in\F_s.
\end{equation*}
Using the estimates \eqref{eq:aprioribothinu} and Step 4, we find that for all $k\geq 1$
\begin{align*}
\sup_{t\in [s,\sigma\wedge T)}\|u(t)\|_{L^{\m}}^{\m}+\int_{s}^{\sigma\wedge T}\int_{\Tor^d}
|u|^{\m-2} |\nabla u|^{2}\,\dd x\,\dd r
&<\infty  \ \ \text{ a.s.\ on }\Gamma_k.
\end{align*}
Since $\P(\{\sigma>s\} \setminus \Gamma_k)\to 0$ as $k\to \infty$ by \eqref{eq:reaction_diffusion_C_alpha_beta}, the boundedness of \eqref{eq:energy_estimates_lemma_reaction_diffusion} and \eqref{eq:energy_estimates_lemma_reaction_diffusion_gradient} follow.

It remains to prove the estimates \eqref{eq:energylambda1} and \eqref{eq:energylambda2} with $\Gamma = \{\sigma>s\, \|u(s)\|_{L^{\m}}\leq L\}$. These follow from Steps 2, 3 and 4 by letting $j\to \infty$. If additionally, $u_0\in B^{\varepsilon}_{\zeta,\infty}$ a.s.\ or ($u_0\in L^2$ a.s.\ and $\zeta = 2$), then it follows from \cite[Propositions 3.1 and 7.1]{AVreaction-local} that $u\in C([0,\sigma);L^{\zeta})$ a.s. Therefore, we can let $s\downarrow 0$ in \eqref{eq:energylambda1} and \eqref{eq:energylambda2}.

{\em Step  6: The system case $\ell>1$}.
One can apply It\^o's formula to $\|u_i^{(j)}\|_{L^{\m}}^{\m}$ for each $i\in \{1,\dots,\ell\}$ to obtain an analogue of \eqref{eq:identityIto} for the $i$-th equation of the system. Multiplying these equations by $\alpha_i$ and summing them up over $i$, one can use Assumption \ref{ass:dissipation_general} to obtain the following analogue of \eqref{eq:Itoest}:
\begin{equation*}
\begin{aligned}
&\sum_{i=1}^\ell \|u^{(j)}_i(t)\|_{L^{\m}}^{\m} + \theta \sum_{i=1}^\ell \int_s^t \int_{\T^d} \one_{[s,\tau_j]\times\Gamma}(r) |u_i(r)|^{\m-2} |\nabla u_i(r)|^2 \, \dd x\,\dd r
\\ & \leq C_1 + \sum_{i=1}^\ell  \alpha_i \|u_i^{(j)}(s)\|_{L^{\m}}^{\m} + M \sum_{i=1}^\ell \int_s^t  \one_{[s,\tau_j]\times\Gamma}(r)  \|u_i^{(j)}(r)\|^{\m}_{L^{\m}}  \, \dd r + \mathcal{S}(t),
\end{aligned}
\end{equation*}
where $\mathcal{S} = \sum_{i=1}^\ell \alpha_i\mathcal{S}_i$, and $\mathcal{S}_i$ is given by
\begin{align*}
\mathcal{S}_i(t)  = \sum_{n\geq 1}  \int_s^t \int_{\T^d} \one_{[s,\tau_j,]\times\Gamma}  |u_i^{(j)}|^{\m-2} u_i  [(b_{n,i}\cdot\nabla) u_i^{(j)}  + g_n(\cdot, u^{(j)})]   \, \dd x\, \dd w^n_r.
\end{align*}
After that Steps 1-5 can be repeated almost verbatim.
\end{proof}

\section{Global existence and uniqueness for weakly dissipative systems}
\label{s:global_system}
In this section, we establish global well-posedness for certain \emph{weakly} dissipative systems. In the case of $2\times 2$ systems, this means that there is a dissipative nonlinearity in one of the equations but not in both. Therefore, the corresponding system is only `weakly' dissipative and the results of the previous section cannot be applied.
More precisely, here we investigate the following cases:
\begin{itemize}
\item Lotka-Volterra equations (Subsection \ref{ss:Lotka_Volterra}).
\item The Brusselator system (Subsection \ref{ss:brusselator}).
\end{itemize}
Moreover, we show that Subsection \ref{ss:Lotka_Volterra} also covers the SIR model (susceptible-infected-removed), and  Subsection \ref{ss:brusselator}
also covers the Gray--Scott model.

In the case of weakly dissipative systems, it does not seem possible to formulate a general `weak' dissipation condition which extends Assumption \ref{ass:dissipation_general}. At the moment, we can only argue `case by case' by possibly relying on the precise structure of the equations, and on some of the analysis done in the scalar case.
For the sake of simplicity, all the above examples are $2$$\times$$2$ systems. However, our arguments can be generalized in certain situations to $\ell$$\times$$\ell$ ones with $\ell\geq 3$ as well, see e.g.\ Remark \ref{r:multi_species}.
Our strategy to deal with weakly dissipative systems is to obtain suitable estimates for $u_1$ by using the dissipation effect of $f_1$ as done in Section \ref{s:scalar_global} for scalar equations. Subsequently, we estimate $u_2$ by exploiting the estimate for $u_1$. Since in the $u_2$-equation in \eqref{eq:reaction_diffusion_system} there is no dissipation, we cannot follow the arguments of Section \ref{s:scalar_global}. To prove estimates on $u_2$, we use a variant of Assumption \ref{ass:dissipation_general} (i.e.\ $L^{\m}$--coercivity/dissipativity) that we investigate in Subsection \ref{ss:suboptimal_coercivity} below.

\subsection{Energy estimates:  Iteration via $L^{\m}$--coercivity/dissipativity}
\label{ss:suboptimal_coercivity}

In this section, we prove the main energy estimates for $u_j$, where $j\in\{1,\dots,\ell\}$ is fixed, and $\ell\geq 2$ is general for the moment. The $j$--equation of the system \eqref{eq:reaction_diffusion_system}, is
\begin{equation*}
\left\{
\begin{aligned}
\dd u_j -\div(a_j\cdot\nabla u_j )\,\dd t
&= \big[\div(F_j(\cdot, u)) +f_j(\cdot, u)\big]\,\dd t
+ \sum_{n\geq 1}  \Big[(b_{n,j}\cdot \nabla) u_j+ \rnoise_{n,j}(\cdot,u) \Big]\,\dd w_t^n,\\
u_j(0)&=u_{0,j},
\end{aligned}
\right.
\end{equation*}
on $\Tor^d$.
The aim of this subsection is to provide estimates for $u_j$. Later on, this will be combined with bounds for $(u_i)_{i=1}^{j-1}$ to get a priori bounds. This eventually leads to global existence in the case of weakly dissipative systems.

Consider the following assumption:
\begin{assumption}[Random $L^{\m}$-coercivity/dissipativity]
\label{ass:dissipation_general_sub_optimal}
Suppose that the assumptions of Theorem \ref{t:reaction_diffusion_global_critical_spaces} are satisfied. Let $(u,\sigma)$ be the corresponding $(p,\a_{\crit},\s,q)$-solution.
For $0\leq s<T<\infty$, $j\in \{1,\dots,\ell\}$, and $\m\in [2,\infty)$, we say that Assumption \ref{ass:dissipation_general_sub_optimal}$(s,T,j,\m)$ holds if there exists a constant $\theta>0$ and there exist $\Progress \otimes \Borel(\Tor^d)$-measurable maps
$M_1,M_2:[s,T]\times \O\times \Tor^d\to [0,\infty)$ and $\mathcal{R}_j:[s,T]\times\Omega\to [0,\infty)$ satisfying
\begin{equation}
\label{eq:integrability_M_k}
M_{k}\in L^{\varphi_k}(s,T;L^{\psi_k}(\Tor^d))  \ \text{ a.s.\ for all }k\in \{1,2\},
\end{equation}
where $\psi_1:= \frac{\m}{2}$,  $\varphi_1:=1$, $\psi_2\in (\frac{d}{2}\vee 1,\infty)$, and $\varphi_2:=\frac{2\psi_2}{2\psi_2-d}$,
for which the following conditions hold a.e.\ in $(s,\sigma)\times \{\sigma>s\}$:
\begin{align*}
\int_{\T^d} |u_j|^{\m-2} \Big(\am_j \nabla u_j\cdot \nabla u_j + F_j(\cdot, u) \cdot \nabla u_j  -&\frac{u_j f_j(\cdot, u)}{\zeta-1}
-\frac12 \sum_{n\geq 1} \big[(b_{n,j} \cdot \nabla) u_j + g_{n,j}(\cdot, u) \big]^2 \Big) \, \dd x
\\ & \geq  \int_{\T^d} |u_j|^{\m-2}\Big( \theta |\nabla u_j|^2 - M_1- M_2|u_j|^2\Big) \dd x + \mathcal{R}_j.
\end{align*}
\end{assumption}

In the examples below $M_1$, $M_2$, and the possible extra term $\mathcal{R}_j$ will depend on $u$ as well, and often we will be able to bootstrap the integrability of $u$ in applications.

There is no direct connection between Assumption \ref{ass:dissipation_general_sub_optimal} and
Assumption \ref{ass:dissipation_general}. However, the following lemma can be proved in a similar way as Lemma \ref{lem:dissipationI}. Lemma \ref{lem:dissipationII} does not seem to have an analogue which is easy to state.

\begin{lemma}
\label{lem:pointwise_bound_suboptimal}
Suppose that Assumption \ref{ass:reaction_diffusion_global} holds.
Suppose that the assumptions of Theorem \ref{t:reaction_diffusion_global_critical_spaces} are satisfied. Let $(u,\sigma)$ be the corresponding $(p,\a_{\crit},\s,q)$-solution.
Let $0\leq s<T<\infty$, $j\in \{1,\dots,\ell\}$, and $\m\in [2,\infty)$. Then
Assumption \ref{ass:dissipation_general_sub_optimal}$(s,T,j,\m)$ holds if there exist  $\Progress\otimes \Borel(\Tor^d)$-measurable  $M_1,M_2:[s,T]\times \O\times\Tor^d\to [0,\infty)$ satisfying \eqref{eq:integrability_M_k}, and $\varepsilon\in (0,\nu)$ such that
\begin{equation*}
\begin{aligned}
\frac{u_j f_j(\cdot,u)}{\zeta-1}+ \frac{1}{4(\ellip_j-\varepsilon)}\Big(|F_j(\cdot,u)|
+\sum_{n\geq 1} |b_{n,j}|\, |g_{n,j}(\cdot,u)|\Big)^2
& +\frac12\|(g_{n,j}(\cdot,u))_{n\geq 1}\|_{\ell^2}^2 \\ & \leq M_1 + M_2|u_j|^2 -\mathcal{R}_j,
\end{aligned}
\end{equation*}
a.e.\ on $(s,\sigma)\times \{\sigma>s\}\times \Tor^d$.
\end{lemma}

We are ready to state a key lemma in this section.
\begin{lemma}[Energy estimates with dissipation]
\label{l:high_integrability_reaction_diffusion_with_dissipation}
Suppose that the assumptions of Theorem \ref{t:reaction_diffusion_global_critical_spaces} are satisfied. Let $(u,\sigma)$ be the corresponding $(p,\a_{\crit},\s,q)$-solution. Fix $0< s<T<\infty$, $j\in \{1,\dots,\ell\}$ and $\m\in [2,\infty)$ and suppose that Assumption \ref{ass:dissipation_general_sub_optimal}$(s,T,j,\m)$ holds.
For all $0\leq t_0<t_1<\infty$ let
\begin{align*}
\ej(t_0,t_1)
&:=\sup_{r\in (t_0, t_1\wedge \sigma)}\|u_j(r)\|_{L^{\m}}^{\m}+ \frac{\theta}{2}\int_{t_0}^{t_1\wedge \sigma }\int_{\Tor^d}
|u_j|^{\m-2} |\nabla u_j|^{2} \,\dd x\,\dd r + \int_{t_0}^{t_1} \mathcal{R}_j(r)  \, \dd r.
\end{align*}
Then
$\ej(s,T) <\infty\ \  \text{ a.s.\ on }\{\sigma>s\}$,
and there exists a constant $C$ only depending on $(d,\zeta,\psi_2)$ such that for all $R,\gamma,\lambda>0$
 \begin{equation}
 \label{eq:estimate_D_j}
 \begin{aligned}
 \P(\one_{\Gamma}\ej(s,T)>\g)
 \leq & \frac{e^{R}}{\g} \E(\one_{\Gamma}\xi_{1,M}\wedge \lambda) + \P(\one_{\Gamma}\xi_{1,M}> C^{-1}\lambda) + \P(\one_{\Gamma}\xi_{2,M}>C^{-1}R),
\end{aligned}
\end{equation}
where $\Gamma\in \F_s$ such that $\Gamma\subseteq \{\sigma>s\}$ is arbitrary,
\[\xi_{1, M} = \|u_{j}(s)\|_{L^{\m}}^{\m} + \|M_1\|_{L^{\varphi_1}(s,T;L^{\psi_1})}, \ \ \ \text{ and }\ \ \
\xi_{2,M} = \sum_{n\in\{1,2\}}\|M_n\|_{L^{\varphi_k}(s,T;L^{\psi_k})}.\]
Moreover, for every $\eta\in (0,1)$, there exists $\alpha_{\eta}>0$ such that
\begin{align}\label{eq:etagronwall}
\E[e^{-\xi_{2,M} \one_{\Gamma}} \one_{\Gamma}\ej(s,T)]^{\eta}\leq \alpha_{\eta}\,\E[\one_{\Gamma} \xi_{1,M}]^{\eta}.
\end{align}
Furthermore, one can take $s=0$ if for some $\varepsilon>0$, $u_0\in B^{\varepsilon}_{\zeta,\infty}(\T^d;\R^{\ell})$ a.s.,\ or $u_0\in L^2(\T^d;\R^{\ell})$ a.s.\ and $\zeta =p=q= 2$.
\end{lemma}

The above result is a variant of Lemma \ref{l:high_integrability_reaction_diffusion}, where additional processes $M_1$, $M_2$, and $\mathcal{R}_j$ appear which are crucial for applications to several of the concrete systems we consider below.  Note that in \eqref{eq:estimate_D_j} no integrability in $\Omega$ is required.

\begin{proof}[Proof of Lemma \ref{l:high_integrability_reaction_diffusion_with_dissipation}]
Fix $0<s<T<\infty$ and let $\Gamma\in \F_s$ be such that $\Gamma\subseteq \{\sigma>s\}$ and $\one_{\Gamma} u(s)\in L^{\m}(\O\times \Tor^d)$.
As in the proof of Lemma \ref{l:high_integrability_reaction_diffusion} we need a localization argument. For each $k\geq 1$, define the stopping time $\tau_k$ by
\begin{equation*}
\tau_k=
\inf\big\{t\in [s,\sigma)\,:\, \|u(t)-u(s)\|_{C(\Tor^d;\R^{\ell})}+\|u\|_{L^2(s,t;H^{1}(\Tor^d;\R^{\ell}))}\geq k\big\}\wedge T,
\end{equation*}
on $\{\sigma>s,\|u(s)\|_{C(\Tor^d;\R^{\ell})}\leq k-1\}$, and $\tau_k = s$ otherwise. Here we also set $\inf\emptyset:=\sigma\wedge T$. Note that $\tau_k$ is well defined due to \eqref{eq:reaction_diffusion_H_theta}--\eqref{eq:reaction_diffusion_C_alpha_beta}.

Set $u_{j}^{(k)}(t)= \one_{\Gamma} u_j(t\wedge \tau_k)$ on $\O\times [s,T]$.
Here to economize the notation we do not display the dependence on $\Gamma$.
Arguing as in \eqref{eq:def_f_u_g_u}-\eqref{eq:identityIto}, an application of the It\^{o} formula of Lemma \ref{lem:Ito_generalized} yields, a.e.\ on $[s,T]\times \O$,
\begin{align}\label{eq:identityIto_2}
\|u_j^{(k)}(t)\|_{L^{\m}}^{\m} = \|u_j(s)\|_{L^{\m}}^{\m} + \m(\m-1) \mathcal{D}(t) + \m \mathcal{S}(t),
\end{align}
where $\mathcal{D}$ and $\mathcal{S}$ are the deterministic and stochastic terms, and are given by
\begin{align*}
\mathcal{D}(t) &= \int_s^t \int_{\T^d} \one_{[s,\tau_k]\times\Gamma} |u_j|^{\m-2} \Big(\frac{u_j f_j(\cdot, u)}{\m-1}  - \nabla u_j \cdot a_j\nabla u_j - \nabla u_j\cdot F_j(\cdot,u) \\ & \qquad \qquad\qquad\qquad\qquad\qquad\quad+\frac12 \sum_{n\geq 1} [(b_{n,j}\cdot\nabla) u_j  + g_{n,j}(\cdot, u)]^2   \Big) \, \dd x \, \dd r,
&
\\ \mathcal{S}(t)& = \sum_{n\geq 1} \int_s^t \int_{\T^d} \one_{[s,\tau_k]\times\Gamma}  |u_j|^{\m-2} u_j  [(b_{n,j}\cdot\nabla) u_j  + g_{n,j}(\cdot, u)]   \, \dd x \,\dd w^n_r.
\end{align*}

To conclude,
it is suffices to show the existence of $C>0$ only depending on $(d,\zeta,\psi_2)$ such that
\begin{equation}
\label{eq:application_grownall_simple_bound}
\begin{aligned}
\one_{\Gamma} &\ej(s,t\wedge \tau_k)
\leq  \one_{\Gamma} \|u_j(s)\|_{L^{\m}}^{\m} + C \|M_1\|_{L^1(s,t;L^{\psi_1})} \\
 &+C\int_s^t\one_{[s,\tau_k]\times\Gamma}\Big(1+\|M_1(r)\|_{L^{\psi_1}}+\|M_2(r)\|_{L^{\psi_2}}^{\varphi_2}\Big)\|u_j(r)\|_{L^{\m}}^{\m} \,\dd r + C\mathcal{S}(t\wedge \tau_k).
\end{aligned}
\end{equation}
Indeed, \eqref{eq:application_grownall_simple_bound} implies
\begin{align*}
\one_{\Gamma} &\ej(s,t\wedge \tau_k)
\leq  \|u_j(s)\|_{L^{\m}}^{\m} + C \|M_1\|_{L^1(s,t;L^{\psi_1})}
\\ &  + \int_s^t C\Big(1+\|M_1(r)\|_{L^{\psi_1}}+\|M_2(r)\|_{L^{\psi_2}}^{ \varphi_2}\Big)\one_{\Gamma}\ej(s,r\wedge \tau_k) \,\dd r + C\mathcal{S}(t\wedge \tau_k).
\end{align*}
Recall that $\xi_{1,M}$ and $\xi_{2,M}$ are defined in Lemma \ref{l:high_integrability_reaction_diffusion_with_dissipation}.
By stochastic Gronwall inequality \cite[Corollary 5.4b)]{geiss2021sharp} applied to the process $t\mapsto \one_{\Gamma}\ej(s,t\wedge \tau_k)$, we find that, for all $\g,\lambda,R>0$,
\begin{align}
\label{eq:estimate_tail_probability_sub_optimal}
\P\Big(\Gamma \cap \big\{\ej(s,\tau_k)\geq \g\big\}\Big)
&\leq \frac{e^{R}}{\g}\E[\xi_{1,M}\wedge \lambda] + \P(\xi_{1,M}>C^{-1}\lambda)
+ \P(\xi_{2,M}> C^{-1}R),
\end{align}
and for all $\eta\in (0,1)$ there exists $\alpha_{\eta}>0$ depending only on $\eta$ such that
\begin{align}
\label{eq:estimate_tail_probability_sub_optimal2}
\E[e^{-\one_{\Gamma}\xi_{2,M}} \one_{\Gamma}\ej(s,\tau_k)]^{\eta}\leq
\alpha_{\eta}\E[\one_{\Gamma} \xi_{1,M}]^{\eta}.
\end{align}
Taking $k\to \infty$  in \eqref{eq:estimate_tail_probability_sub_optimal}, by monotone convergence, one sees that \eqref{eq:estimate_tail_probability_sub_optimal} and \eqref{eq:estimate_tail_probability_sub_optimal2} hold with $\tau_k$ replaced by $\sigma$.
Hence, \eqref{eq:estimate_D_j} and \eqref{eq:etagronwall} follow by replacing $\Gamma$ by  $\Gamma_N=\{\sigma>s,\ \|u_j(s)\|_{L^{\infty}}\leq N\}\cap \Gamma$ and letting $N\to \infty$. Finally, $\ej(s,T) <\infty$ on $\sigma>s$ follows by first choosing $R$ and $\lambda$ large enough, and then letting $\gamma\to \infty$.

Next, we prove \eqref{eq:application_grownall_simple_bound}. By Assumption \ref{ass:dissipation_general_sub_optimal}
\begin{equation}
\label{eq:deterministic_estimate_pointwise_sub_optimal_proof}
\begin{aligned}
\mathcal{D}(t)
&\leq -\theta \int_s^t \int_{\T^d}\one_{[s,\tau_k]\times\Gamma}  |u_j|^{\m-2} |\nabla u_j|^2 \, \dd x\, \dd r\\
& \ +  \theta \int_s^t \int_{\T^d}\one_{[s,\tau_k]\times\Gamma} |u_j|^{\m-2}(M_1+M_2|u_j|^2) \, \dd x\,\dd r - \int_s^t \one_{[s,\tau_k]\times\Gamma} \mathcal{R}_j(r) \, \dd r.
\end{aligned}
\end{equation}

We estimate the $M_i$-terms. One has a.e.\ on $[s,\tau_k]\times \O$,
\begin{align*}
\int_{\Tor^d}
M_2 |u_j|^{\m}\,\dd x
&\leq \|M_2\|_{L^{\psi_2}}\big\||u_j|^{\m}\big\|_{L^{\psi_2'}}\\
&=  \|M_2\|_{L^{\psi_2}} \big\||u_j|^{\m/2}\big\|_{L^{2 \psi_2'}}^2\\
&\stackrel{(i)}{\lesssim}  \|M_2\|_{L^{\psi_2}} \big\||u_j|^{\m/2}\big\|_{H^{\theta}}^2\\
&\lesssim \|M_2\|_{L^{\psi_2}} \big\||u_j|^{\m/2}\big\|_{L^2}^{2(1-\theta)}\big\||u_j|^{\m/2}\big\|_{H^1}^{2\theta}\\
&\lesssim \|M_2\|_{L^{\psi_2}} \big\||u_j|^{\m/2}\big\|_{L^2}^{2(1-\theta)}\big( \big\||u_j|^{\m/2}\big\|_{L^2}^{2\theta}+
 \big\|\nabla [|u_j|^{\m/2}]\big\|_{L^2}^{2\theta}\big)\\
 &\lesssim \|M_2\|_{L^{\psi_2}} \|u_j\|_{L^{\m}}^{\m(1-\theta)}\big( \|u_j\|_{L^{\m}}^{\m\theta}+
 \big\||u_j|^{\m/2-1}|\nabla u_j|\big\|_{L^2}^{2\theta}\big)\\
&\stackrel{(ii)}{\leq} \frac{1}{2\varepsilon}\big(\|M_2\|_{L^{\psi_2}}+\|M_2\|_{L^{\psi_2}}^{\varphi_2}\big)
\|u_j\|_{L^{\m}}^{\m}
+ \frac{\varepsilon}{2} \int_{\Tor^d} |u_j|^{\m-2}|\nabla u_j|^2\,\dd x,
\end{align*}
where in $(i)$ we used the embedding $H^{\theta}(\Tor^d)\embed L^{2 \psi_2'}(\Tor^d)$ with $\theta=\frac{d}{2}-\frac{d}{2 \psi_2'}=\frac{d}{2 \psi_2}\in (0,1)$, and in $(ii)$ Young's inequality as well as $\varphi_2=\frac{1}{1-\theta}$.

For $M_1$ we show that a.e.\ on $[s,\tau_k]\times \O$,
\begin{equation}
\label{eq:M_2_estimate_suboptimal}
\int_{\Tor^d} M_1  |u_j|^{\m-2}\,\dd x
\lesssim_{\m} \|M_1\|_{L^{\m/2}} (1+\|u_j\|_{L^{\m}}^{\m}).
\end{equation}
Note that \eqref{eq:M_2_estimate_suboptimal} is immediate in the case $\m=2$. In case $\m>2$, one has \eqref{eq:M_2_estimate_suboptimal} follows from
\begin{align*}
\int_{\Tor^d} M_1  |u_j|^{\m-2}\,\dd x
&\leq \|M_1\|_{L^{\m/2}} \big\||u_j|^{\m-2}\big\|_{L^{{\m}/(\m-2)}}= \|M_1\|_{L^{\m/2}} \|u_j\|_{L^{\m}}^{\m-2}\lesssim_{\m} \|M_1\|_{L^{\m/2}}(1+ \|u_j\|_{L^{\m}}^{\m}).
\end{align*}

From \eqref{eq:deterministic_estimate_pointwise_sub_optimal_proof}, and the above estimates for $M_1$ and $M_2$  with $\varepsilon = \theta$, we obtain
\begin{align*}
\mathcal{D}(t) &\leq -\frac{\theta}{2} \int_s^t \int_{\T^d}\one_{[s,\tau_k]\times\Gamma}  |u_j|^{\m-2} |\nabla u_j|^2 \, \dd x\,\dd r - \int_s^t \one_{[s,\tau_k]\times\Gamma} \mathcal{R}_j(r) \, \dd r \\ & \ + C_{\zeta,\theta}  \int_s^t \one_{[s,\tau_k]\times\Gamma} \Big[\|M_1\|_{L^{\m/2}}  + \big(1+\|M_1\|_{L^{\m/2}} + \|M_2\|_{L^{\psi_2}}^{\varphi_2}\big)
\|u_j\|_{L^{\m}}^{\m}\Big] \, \dd r
\end{align*}
a.e.\ on $[s,T]\times \O$.
This implies the estimate \eqref{eq:application_grownall_simple_bound}.

The final assertion concerning $s=0$ follows by taking $s\downarrow 0$ as in Step 5 of Lemma \ref{l:high_integrability_reaction_diffusion}.
\end{proof}

\subsection{Stochastic Lotka-Volterra equations}
\label{ss:Lotka_Volterra}
Lotka-Volterra equations (also known as \emph{predatory-prey} model) were initially proposed by Lotka in the theory of \emph{autocatalytic chemical reactions} in 1910 \cite{L10}. In the book \cite{L25}, the same author used this model to study \emph{predatory-prey} interactions. Around the same years, Volterra proposed the same equations to study similar systems \cite{V31}. The initially proposed model that consists of a system of ODEs for the unknown density of species, can be extended to a system of PDEs taking into account spatial inhomogeneity and diffusivity, see e.g.\ \cite{CL84,GL94}. Stochastic perturbations of such models take into account uncertainties in the determination of the external forces and/or parameters, see e.g.\ \cite{CE89,NY21} and the references therein. As explained before, \emph{transport noise} can be thought of as the `small-scale' effect of migration phenomena of the species. To the best of our knowledge, our result is the first global existence result for Lotka-Volterra equations with transport noise. But even if $b=0$, our results seem to be new.
Here we are mainly concerned with modelling of the dynamics of \emph{two} species. However, our arguments extend to certain Lotka-Volterra equations for multi-species. The reader is referred to Remark \ref{r:multi_species} for the precise description.
Moreover, in order to simplify the presentation, we only consider the low dimensional case $d\leq 4$ (therefore including the physical dimensions $d\in \{1,2,3\}$). The high-dimensional setting requires some modifications and additional conditions.

Here we consider the stochastic Lotka-Volterra equations (with transport noise and superlinear diffusion), i.e.\ \eqref{eq:reaction_diffusion_system} with $\ell=2$,
\begin{equation}
\label{eq:nonlinearity_lotka_volterra_last_section}
\begin{aligned}
f_1(\cdot,y)
&= \lambda_1 y_1-\chi_{1,1}y_1^2 -\chi_{1,2} y_1y_2 & \text{ for }\ &y\in\R^2,\\
f_2(\cdot,y)
&= \lambda_2 y_2-\chi_{2,2}y_2^2 +\chi_{2,1} y_1y_2 &\text{ for }\ &y\in\R^2,
\end{aligned}
\end{equation}
$F_i\equiv 0$,  and $(a,b,\lambda_i,\chi_{i,j},g)$ is  as in  Assumption \ref{ass:lotka_volterra} below. In particular $(\lambda_i,\chi_{i,j})$ are chosen so that $f_1$ and $f_2$ satisfy Assumption \ref{ass:reaction_diffusion_global}\eqref{it:growth_nonlinearities} for $h=2$.

The unknowns
 $u_1,u_2:[0,\infty)\times \O\times \Tor^d \to \R$ model the population of the $i$-th specie. More precisely, $u_1$ and $u_2$ denote the population of the \emph{prey} and \emph{predator}, respectively.
A schematic idea of the Lotka-Volterra equations is given in Figure \ref{fig:lotka_volterra}
where $\lambda'_1+\lambda''_1=\lambda_1$.

\begin{figure}[h!]
\centering

\tikzstyle{io} = [ellipse,text width=2cm, minimum height=1cm, text centered, draw=black]
\tikzstyle{nn} = [rectangle,text width=1.8cm, minimum height=0.5cm, minimum width=2.5cm, text centered, draw=black]
\tikzstyle{nncircle} = [circle,text width=0.8cm, minimum height=0.5cm, minimum width=0.9cm, text centered, draw=black]
\tikzstyle{arrow} = [thick,->,>=stealth]

\begin{tikzpicture}[scale=0.95, transform shape, node distance=2.1cm]
\node (in1) [nn] {\small Predator $u_2$};
\node (in2) [nn,below of = in1] {\small Prey $u_1$};
\node (in3) [nncircle,left of = in2,xshift=-2.8cm] {\small Food };
\node (in4) [nncircle,right of = in1,yshift=-1cm,xshift=2.8cm] {\small Death};

\path[thick,<-,>=stealth,shift left=0.2cm]
    (in1) edge node [left, xshift=-0.5cm]{$\chi_{2,1}u_{1}u_{2}\geq 0$} (in2)
    (in2) edge node [right, xshift=0.5cm]{$-\chi_{1,2}u_{1}u_{2}\leq0$}(in1);
\path [thick,<-,>=stealth](in1) edge [out=50,in=120,looseness=10] node [right, xshift=0.3cm]{$-\chi_{2,2}u_{2}^{2}\leq0$}(in1);
\path [thick,<-,>=stealth](in2) edge [out=-50,in=-120,looseness=10] node [right, xshift=0.3cm]{$-\chi_{1,1}u_{1}^{2}\leq0$} (in2);
\draw [arrow] (in3) -- node [below]{$\lambda^{\prime}_{1}u_{1}\geq0$} (in2);
\draw [arrow] ([yshift=0.05cm]in4.west) -- node[above right]{$\lambda_{2}u_{2}\leq0$} (in1.east);
\draw [arrow] ([yshift=-0.05cm]in4.west) -- node[below right]{$\lambda^{\prime\prime}_{1}u_{1}\leq0$} (in2.east);

\end{tikzpicture}
\caption{Scheme for the Lotka-Volterra equations.}
\label{fig:lotka_volterra}
\end{figure}
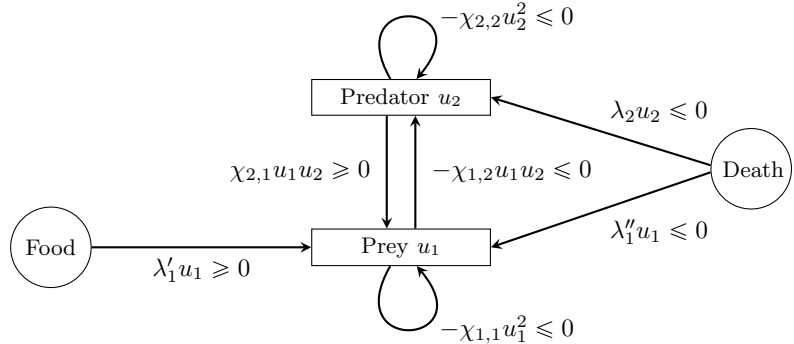

In particular, the signs in \eqref{eq:nonlinearity_lotka_volterra_last_section} are motivated from a modelling point of view.
Note that we allow $\chi_{i,j}\equiv 0$ for $i,j\in \{1,2\}$.
In particular, situations without self-interacting are covered, i.e. $\chi_{i,i}\equiv 0$.
From a mathematical viewpoint, the term  $+\chi_{2,1}  u_1u_2$ shows that the $u_2$-equation in the Lotka-Volterra equations is \emph{not} dissipative (below we are concerned with positive solutions). Hence, the Lotka-Volterra system is only \emph{weakly} dissipative and is \emph{not} covered by Section \ref{s:scalar_global}.

\begin{assumption}
\label{ass:lotka_volterra}
Let $p=q=2$, $\kappa = 0$, $h=2$, $1\leq d\leq 4$, $\ell=2$, $\delta=1$, and suppose that Assumption \ref{ass:reaction_diffusion_global} holds for $(a,b,g)$. Suppose that
\begin{enumerate}[{\rm(1)}]
\item\label{it:chi_sign}  $\lambda_i,\chi_{i,j}$ are bounded $\Progress\otimes \Borel(\Tor^d)$-measurable  and $\chi_{i,j}\geq 0$.
\item\label{it:lotka_volterra_positivity} a.s.\ for all $t\in \R_+$, $x\in \Tor^d$, $y_1, y_2\geq 0$, $n\geq 1$,
\begin{align*}
g_{n,1}(t,x,y_1,0)&=g_{n,2}(t,x,0,y_2)=0.
\end{align*}
\item\label{it:lotka_volterra_dissipation} There exist $M,\varepsilon>0$ such that for all $i\in \{1,2\}$, a.e.\ on $\R_+\times \O\times \Tor^d$ and  $y=(y_1,y_2)\in [0,\infty)^2$,
\begin{align*}
\frac{1}{4(\ellip_i -\varepsilon)}\Big(\sum_{n\geq 1} |b_{n,i}(\cdot)|\, |g_{n,i}(\cdot,y)|\Big)^2 +\frac12\|(g_{n,i}(\cdot,y))_{n\geq 1}\|_{\ell^2}^2\leq    N_i(t,x,y) & & \text{(growth condition),}
\end{align*}
where $\ellip_i$ are as in Assumption \ref{ass:reaction_diffusion_global}\eqref{it:ellipticity_reaction_diffusion} and
\begin{align*}
N_1(\cdot,y)&:= M \big(1+y_1^2 \big)+ \chi_{1,1} y_1^3+\chi_{1,2} y_1^2 y_2,\\
N_2(\cdot,y)&:= M \big[1 +  (1+ y_1)y_2^{2} + y^2_1 y_2+ y_1^3 \big] +\chi_{2,2} y_2^3.
\end{align*}
\end{enumerate}
\end{assumption}

Assumption \eqref{it:lotka_volterra_positivity} will be used to prove \emph{positivity}. This fact is consistent with the modelling interpretation of $u_i$ as the population of the $i$-th species. Condition \eqref{it:lotka_volterra_dissipation} shows that $g_2$ is allowed to be superlinear in $y=(y_1,y_2)$. The growth of $\|g_2(\cdot,y)\|_{\ell^2}$ is at most $|y|^{3/2}$  consistently with Assumption \ref{ass:reaction_diffusion_global}\eqref{it:growth_nonlinearities} in the case $h=2$.
On the other hand, $g_1$ is allowed to be (super)linear in $y_1$ (if $\chi_{1,1}\geq c_0>0$). In applications, there is always interaction between predator and prey, so that $\chi_{1,2}\geq c_0> 0$. In this case, the growth of $\|g_1(\cdot, y)\|_{\ell^2}$  is at most $|y_1| |y_2|^{1/2}$ for $|y|$ large. If additionally $\chi_{1,1}\geq c_0>0$, then we can allow growth $|y_1|^{3/2} + |y_1| |y_2|^{1/2}$ for $|y|$ large.

\begin{theorem}[Global well-posedness -- Lotka-Volterra]
\label{t:Lotka_Volterra}
Let $1\leq d\leq 4$. Suppose that Assumption \ref{ass:lotka_volterra} holds.
Then for every $u_{0}\in L^0_{\F_0}(\Omega;L^2( \T^d;\R^2))$ with $u_{0}\geq 0$ (componentwise), there exists a (unique) global $(2,0,1,2)$-solution $u:[0,\infty)\times\Omega\times\T^d\to [0,\infty)^2$ to the stochastic Lotka-Volterra equations (as described near \eqref{eq:nonlinearity_lotka_volterra_last_section}). Moreover, a.s.,
\begin{align*}
u&\in L^2_{\loc}(0,\infty;H^1(\T^d;\R^2)) \cap C([0,\infty);L^2(\Tor^d;\R^2)),\\
u&\in H^{\theta,r}_{\rm loc}(0,\infty;H^{1-2\theta,\zeta}(\Tor^d;\R^{2})) \ \ \forall\theta\in  [0,\tfrac{1}{2}), \ \forall  r,\zeta\in (2,\infty),
\\
u&\in C^{\theta_1,\theta_2}_{\rm loc}((0,\infty)\times \Tor^d;\R^{2})
\ \  \forall\theta_1\in  [0,1/2), \ \forall \theta_2\in (0,1),
\end{align*}
and the following estimates hold for all $0<T<\infty$ and $\gamma>0$:
\begin{align}
\label{eq:aprioridecayLV1}\P\Big(\sup_{t\in [0,T]} \|u(t)\|_{L^{2}}^{2}\geq \g\Big)
&\leq \psi(\g)(1+\E\|u_0\|_{L^{2}}^2),\\
\label{eq:aprioridecayLV2}\P\Big(\max_{1\leq i\leq 2} \int_{0}^{T}\int_{\Tor^d}|\nabla u_i|^2\,\dd x\,\dd t\geq \g\Big)
&\leq \psi(\g)(1+\E\|u_0\|_{L^{2}}^{2}),
\end{align}
where $\psi$ does not depend on $\gamma$ and $u$, and satisfies $\lim_{\gamma\to \infty}\psi(\gamma) = 0$.
\end{theorem}
Using Theorem \ref{t:continuity_general} one can further weaken the assumption on the initial data.

The main idea in the proof of Theorem \ref{t:Lotka_Volterra} is to exploit the dissipative effect of $\chi_{1,1}u_1^2$ and $\chi_{1,2}u_1u_2$ in the $u_1$-equation in the Lotka-Volterra equations.
To this end, we need the positivity of $u_1,u_2$ which is ensured by Proposition \ref{prop:positivity}, \eqref{it:lotka_volterra_positivity} and $u_{0}\geq 0$ a.s.\ on $\Tor^d$.

\begin{proof}[Proof of Theorem \ref{t:Lotka_Volterra}]
As in the proof of Theorem \ref{thm:SymbLotka-Volterra}, by adding dummy variables, it suffices to consider the case $d=4$. The existence of a $(2,0,1,2)$-solution $(u,\sigma)$ with the required regularity up to $\sigma$ follows from Theorem \ref{t:reaction_diffusion_global_critical_spaces}.
Moreover, the assumption \eqref{it:lotka_volterra_positivity} and Proposition \ref{prop:positivity} yield $u\geq 0$ on $\Tor^d$ a.s.\ for all $t\in[0,\sigma)$.
Hence, it remains to prove that $\sigma=\infty$ a.s. To this end, note that $\max\{\frac{d}{2}(h-1),2\}=2$ since $d=4$ and $h=2$.
Localizing in $\O$, it is enough to consider $u_0\in L^2(\O;L^2(\T^4;\R^2))$, see \cite[Proposition 4.13]{AV19_QSEE_2}.
Thus, as in Step 1 in the proof of Theorem \ref{t:global_well_posedness_general} (but now using Theorem \ref{thm:blow_up_criteria}\eqref{it:blow_up_sharp_2}), it is enough to show that, for every $0<T<\infty$,
\begin{align}
\label{eq:Lotka_volterra_bound}
\ej(0,T):=\sup_{t\in [0,\sigma\wedge T)} \|u_j(t)\|^2_{L^{2}}+\int_0^{\sigma\wedge T}\int_{\Tor^d}|\nabla u_j(t)|^2\,\dd x\,\dd t&<\infty \ \
\text{ a.s.}
\end{align}

From Assumption  \ref{ass:lotka_volterra}\eqref{it:lotka_volterra_dissipation} and Lemma \ref{lem:dissipationI} we see that $L^{\zeta}$-coercivity for $\zeta=2$ holds for the equation for $u_1$. Thus, Lemma \ref{l:high_integrability_reaction_diffusion} implies that for all $0\leq s<T<\infty$ and $j=1$, \eqref{eq:Lotka_volterra_bound}  holds, and
\begin{align}
\nonumber
\sup_{t\in [0,T]}\E \one_{[0,\sigma)}(t) \|u_1(t)\|_{L^{2}}^{2} + \E \int_{0}^{\sigma\wedge T}\int_{\Tor^d} |\nabla u_1|^{2}\,\dd x\,\dd t&\leq N_0\Big(1+\E\big[\|u_{1,0}\|_{L^{2}}^{2}\big]\Big),
\\
\label{eq:energylambda2LV} \E \sup_{t\in [0,\sigma\wedge T)} \|u_1(t)\|_{L^{2}}^{2 \lambda} + \E \Big|\int_{0}^{\sigma\wedge T}\int_{\Tor^d} |\nabla u_1|^{2}\,\dd x\,\dd t\Big|^{\lambda}
& \leq N_{0,\lambda}\Big(1+\E\big[\|u_{1,0}\|_{L^{2}}^{2 \lambda}\big]\Big).
\end{align}

To check \eqref{eq:Lotka_volterra_bound} for $j=2$ note by Assumption \ref{ass:lotka_volterra}\eqref{it:lotka_volterra_dissipation} a.e.\ on $[0,\sigma)\times \O$,
\begin{equation*}
\begin{aligned}
{f}_2(\cdot,u)u_2
+\frac{1}{4(\ellip_2-\varepsilon) }\Big(\sum_{n\geq 1} |b_{n,2}(\cdot)|\, |g_{n,2}(\cdot,u)|\Big)^2
&+\frac12\|g_{1}(\cdot,u)\|_{\ell^2}^2\leq M_0( M_1 |u_2|^2+ M_2),
\end{aligned}
\end{equation*}
where we used that $u_1^2 u_2= u_1^{3/2} (u_1^{1/2} u_2)\lesssim u_1^3 +u_1 u^2_2$, where $ M_0\geq 0$ is a constant, and $M_1$ and $M_2$ are given by
\begin{equation*}
M_1:=\one_{[0,\sigma)}|u_1|^3 + 1
\quad \text{ and }\quad
M_2:=\one_{[0,\sigma)} |u_1| + 1.
\end{equation*}

We claim $M_1\in L^1(0,T;L^1(\Tor^4))$ and $M_2\in L^{2}(0,T;L^{4}(\Tor^4))$  a.s.  Indeed, from \eqref{eq:Lotka_volterra_bound} for $j=1$ we obtain
\begin{equation}
\label{eq:integrability_u_1_lotka_volterra}
u_1\in L^{\infty}(0,\sigma\wedge T;L^{2}(\Tor^4))\cap L^2(0,\sigma\wedge T;H^1(\Tor^4)) \  \text{ a.s.}
\end{equation}
By Sobolev embedding, interpolation, and \eqref{eq:integrability_u_1_lotka_volterra}, we have
a.s.\ on $\{\sigma>s\}$,
\begin{align*}
\|u_1\|_{L^{3} (0,\sigma\wedge T;L^{3}(\Tor^4))} & \lesssim \|u_1\|_{L^{3}(0,\sigma;H^{2/3}(\Tor^4))}
\leq \sup_{t\in [0,\sigma\wedge T)}\|u_1(t)\|_{L^{2}(\Tor^4)} + \|u_1\|_{L^2(0,\sigma\wedge T;H^1(\Tor^4))},
\end{align*}
which already gives the first part of the claim. Moreover, taking moments, by  \eqref{eq:energylambda2LV} (with $\lambda = 1/2$ which is non-optimal) we obtain
\begin{equation}\label{eq:suboptimalestu1LV}
\begin{aligned}
\E\|u_1\|_{L^{3} (0,\sigma\wedge T;L^{3}(\Tor^4))}
 & \lesssim \E\sup_{t\in [0,\sigma\wedge T)}\|u_1(t)\|_{L^{2}(\Tor^4)} + \E\|u_1\|_{L^2(0,\sigma\wedge T;H^1(\Tor^4))}
 \\ & \lesssim 1+\E\|u_{1,0}\|_{L^{2}}.
\end{aligned}
\end{equation}
The second part of the claim follows from \eqref{eq:integrability_u_1_lotka_volterra} and $H^1(\Tor^4)\embed L^4(\Tor^4)$. Moreover,
\begin{align}\label{eq:suboptimalestu2LV}
\E\|u_1\|_{L^{2} (0,\sigma\wedge T;L^{4}(\Tor^4))} &\leq \E\|u_1\|_{L^{2} (0,\sigma\wedge T;H^1(\Tor^4))}
\lesssim 1+\E\|u_{1,0}\|_{L^{2}}.
\end{align}

The claim can be combined with Lemmas \ref{lem:pointwise_bound_suboptimal} and \ref{l:high_integrability_reaction_diffusion_with_dissipation} to obtain \eqref{eq:Lotka_volterra_bound} for $i=2$. Moreover, this also implies the following estimate for all $\gamma>0$ and $\lambda=R\geq \max\{4C,1\}$ (say):
\begin{align*}
 \P(\mathcal{E}_{2}(0,T)>\g)
 &\leq  \frac{R e^{R}}{\g}   + \P\big(\|u_{2,0}\|_{L^{2}}^{2} + C \|M_1\|_{L^{1}(0,T;L^{1})}\geq R\big) \\ & \qquad + \P\big(C\|M_1\|_{L^{1}(0,T;L^{1})} + C\|M_2\|_{L^{2}(0,T;L^{4})}\geq R\big)
 \\ & \leq \frac{e^{2R}}{\g} + \P\big(\|u_{2,0}\|_{L^{2}}^{2}\geq R/2\big) + 2\P\big( C \|u_1\|_{L^{3}(0,T;L^{3})}^3+C\geq R/2\big) \\ & \quad + \P\big(C\|u_1\|_{L^{2}(0,T;L^{4})} + C \geq R/2\big)
\\ & \lesssim \frac{e^{2R}}{\g} + \frac{\E \|u_{2,0}\|_{L^{2}}^{2}}{R} + \frac{\E \|u_{1}\|_{L^{3}(0,T;L^{3})}}{R^{1/3}} + \frac{\E \|u_{1}\|_{L^{2}(0,T;L^{4})}}{R}
\\ & \lesssim \Big(\E \|u_0\|_{L^{2}}^{2}+1\Big) \Big(\frac{e^{2R}}{\g} + \frac{1}{R^{1/3}}\Big),
\end{align*}
where we used \eqref{eq:suboptimalestu1LV} and \eqref{eq:suboptimalestu2LV}.  It remains to take $R = \max\{4C,1, \log(\gamma^{1/4})\}$ to obtain
\[\P(\mathcal{E}_{2}(0,T)\geq \g)\leq \psi(\gamma) \big(\E \|u_0\|_{L^{2}}^{2}+1\big),\]
where $\lim_{\gamma\to \infty} \psi(\gamma) = 0$.   Combined with \eqref{eq:energylambda2LV}, this the energy estimates \eqref{eq:aprioridecayLV1} and \eqref{eq:aprioridecayLV2}.
\end{proof}

We conclude this subsection with several remarks.
\begin{remark}[Stochastic Lotka-Volterra equations for multi-species]
\label{r:multi_species}
In the study of multi-species systems, one may consider the following extension of the Lotka-Volterra equations:

For $i\in \{1,\dots,\ell\}$ and $\ell\geq 1$,
\begin{align*}
\dd u_i -\div(a_i\cdot\nabla u_i) \,\dd t
= & \Big[  \lambda_i u_i + u_i\sum_{1\leq j<i} \chi_{i,j} u_j - u_i\sum_{j\geq i} \chi_{i,j} u_j\Big]\,\dd t
+ \sum_{n\geq 1}  \Big[(b_{n,i}\cdot \nabla) u_i+ g_{n,i}(\cdot,u) \Big]\,\dd w_t^n
\end{align*}
on $\Tor^d$. As above, $u_i$ denotes the density of the $i$-th species,  $(\chi_{i,j})_{i,j=1}^{\ell}:(0,\infty)\times\O\times \Tor^d \to [0,\infty)^{\ell \times \ell}$ is a matrix with bounded $\Progress\otimes \Borel(\Tor^d)$-measurable entries which models the interaction of the species.
One can readily check that
Theorem \ref{t:Lotka_Volterra} extends to the above system
with appropriate assumptions on the $(u_0,a,b,g)$ (i.e.\ extending \eqref{it:lotka_volterra_positivity}-\eqref{it:lotka_volterra_dissipation} to the above setting). In particular, $g$ can be chosen to be superlinear in $u$.
We leave the details to the interested reader.
We conclude by mentioning that the corresponding condition for global existence is satisfied for the natural SPDE version of the Lotka-Volterra system for three species as analyzed in \cite{HRY15}.
\end{remark}

\begin{example}[The stochastic diffusive SIR model] In the deterministic setting, the SIR model (susceptible-infected-removed)
is a widely used model in epidemiology \cite{M89_SIR,H89_SIR,KM27_SIR}. Taking into account spatial diffusivity and small-scale fluctuations, stochastic perturbations of the SIR model (with constant population $N>0$) are of the form \eqref{eq:reaction_diffusion_system} with $\ell=2$, $F=0$,
$f_1(\cdot,u)= -r_1 u_1u_2$, and $f_2(\cdot,u)=r_2 u_1u_2+ r_3 u_2$.
The quantities $u_1$ and $u_2$ are the population of the stock of susceptible and infected subjects, respectively; while the stock of removed population $u_3$ (either by death or recovery) is given by $N-u_1-u_2$. Finally, $(r_1,r_2)$  are $\Progress\otimes \Borel(\Tor^d)$-measurable maps satisfying $r_i\geq 0$ a.e.\ on $\R_+\times \O\times \Tor^d$ and are determined experimentally.
Of course, as a sub-case the above model contains the classical SIR model in the case of $(\om,x)$-independent solutions.

The stochastic diffusive SIR model is a special case of the Lotka-Volterra equations where
one takes $\chi_{1,2}=r_1$, $\chi_{2,1}=r_2$, $\chi_{i,i}\equiv 0$, $\lambda_1=0$ and $\lambda_2=r_3$.
Thus, under suitable assumptions on $(a,b,g,u_0)$, the SIR model is included in our setting, and global well-posedness follows from Theorem \ref{t:Lotka_Volterra}.
\end{example}

\subsection{The stochastic Brusselator}
\label{ss:brusselator}
Here we study a stochastic perturbation of the so-called \emph{Brusselator}. This model appears in the study of \emph{chemical morphogenetic processes} (see e.g.\ \cite{PL68,PN71,Tu90}) and in \emph{autocatalytic reactions} (see e.g.\ \cite{AO78,B95}). For historical notes and examples of autocatalytic reactions that can be modelled by using the Brusselator, the reader is referred to the introduction of \cite{YZ12}. Stochastic perturbations of this model can be used to model thermal fluctuations and/or small-scale turbulence.
Here we consider the stochastic Brusselator  (with transport noise and superlinear diffusion), i.e.\ \eqref{eq:reaction_diffusion_system} with $\ell=2$,
\begin{equation}
\label{eq:brusselator}
\begin{aligned}
f_1(\cdot,y)&=
- u_1 u_2^2+ \alpha_{1} y_1+\alpha_{2}y_2+\alpha_{0} &\text{ for }\ &y\in\R^2,\\
f_2(\cdot,y)&=+  y_1y_2^2 +\beta_1 y_1+\beta_2y_2+\beta_0&\text{ for }\ &y\in\R^2,
\end{aligned}
\end{equation}
$F_i\equiv 0$ and $(a_i,b_i,\alpha_i,\beta_j,g_i)$ is described below, cf.\ Assumptions \ref{ass:brusselator_two_dimensions} and \ref{ass:brusselator_three} below. The unknowns $u_1,u_2:[0,\infty)\times \O\times \Tor^d \to \R$ model the unknown concentrations.

Note that $f$ satisfies Assumption \ref{ass:reaction_diffusion_global}\eqref{it:growth_nonlinearities} for all $h\geq 3$.
In the following, we assume that $(a,b,g)$ satisfies Assumption \ref{ass:reaction_diffusion_global} with $h=3$.
We say that \emph{$(u,\sigma)$ is a (unique) $(p,\a,q,\s)$-solution} to \eqref{eq:brusselator} if $(u,\sigma)$ is a $(p,\a,q,\s,3)$-solution to \eqref{eq:reaction_diffusion_system} with the above choice of $(a,b,f,F,g)$, see Definition \ref{def:solution}. A $(p,\a,q,\s)$-solution $(u,\sigma)$ to \eqref{eq:brusselator} is said to be \emph{global} if $\sigma=\infty$ a.s. In this case, we simply write $u$ instead of $(u,\sigma)$.

In this subsection, we prove \emph{global well-posedness} of \eqref{eq:brusselator} in the physical dimensions $ d\leq 3$.
The analysis of \eqref{eq:brusselator} for $d\in \{1,2\}$ is easier than the three-dimensional case.
In particular, in three dimensions, we need positivity of solutions (see Proposition \ref{prop:positivity}).
For the reader's convenience, we split the argument into the cases $d\in \{1,2\}$ and $d=3$, see Subsection \ref{sss:brusselator_two_dimensional} and \ref{sss:brusselator_three_dimensional} respectively.

We believe that the arguments below can be extended to study stochastic perturbations of the \emph{extended} Brusselator (see e.g.\ \cite{YZ12}). For the sake of brevity, we do not pursue this here.

\subsubsection{The stochastic Brusselator in one and two dimensions}
\label{sss:brusselator_two_dimensional}
We begin by listing the assumptions needed in this subsection.

\begin{assumption}
\label{ass:brusselator_two_dimensions}
Suppose that Assumption \ref{ass:reaction_diffusion_global}\eqref{it:ellipticity_reaction_diffusion}-\eqref{it:growth_nonlinearities} hold as well as the following:
\begin{enumerate}[{\rm(1)}]
\item The mappings $\alpha_i,\beta_i:\R_+\times \O\times \Tor^d\to \R$ are bounded and $\Progress\otimes \Borel(\Tor^d)$-measurable.
\item\label{it:brusselator_two_dimensions_2} There exist $M,\varepsilon>0$ such that for all $i\in \{1,2\}$, a.e.\ on $\R_+\times \O\times \Tor^d$ and  $y=(y_1,y_2)\in [0,\infty)^2$,
\begin{align*}
\frac{1}{4(\ellip_i -\varepsilon)}\Big(\sum_{n\geq 1} |b_{n,i}(\cdot)|\, |g_{n,i}(\cdot,y)|\Big)^2 +\frac12\|(g_{n,i}(\cdot,y))_{n\geq 1}\|_{\ell^2}^2\leq    N_{i}(t,x,y) & & \text{(growth condition),}
\end{align*}
where $\ellip_i$ are as in Assumption \ref{ass:reaction_diffusion_global}\eqref{it:ellipticity_reaction_diffusion} and
\begin{align*}
N_{1}(\cdot,y)
&:=M [1+|y_1|^2] +  (1-\varepsilon) |y_1|^2|y_2|^2,\\
N_{2}(\cdot,y)
&:=M [1+(1  + |y_1|^2)|y_2|^2 +|y_1||y_2|^{3}+ |y_1|^4].
\end{align*}
\end{enumerate}

\end{assumption}

Note that $N_{i}$ grows like $|y|^4$ consistently with Assumption \ref{ass:reaction_diffusion_global}\eqref{it:growth_nonlinearities} in case $h=3$.
However, the condition for $g_1$ is quite restrictive since only the factor $|y_1|^2|y_2|^2$ is allowed. Instead, the one for $g_2$ allows additional terms of quadratic growth and with a constant $M$ that can be large.

\begin{theorem}[Brusselator in one and two dimensions]
\label{t:Brusselator}
Let $d\in \{1, 2\}$. Suppose that Assumption \ref{ass:brusselator_two_dimensions} holds.
Fix $\s\in (1, 2)$, $q\in (2,\frac{2}{2-\delta})$,  and $p\in [\frac{2}{2-\delta},\infty)$. Suppose that  Assumption \ref{ass:reaction_diffusion_global}$(p,q,h,\s)$ holds with $h=3$ and $\a_{0}:=p(1-\frac{\s}{2})-1$.
Then for all
$
u_0\in L^0_{\F_0}(\O;L^{q}(\Tor^d;\R^2))
$,
the stochastic Brusselator (as described near \eqref{eq:brusselator}) has
\emph{a (unique) global $(p,\a_0,q,\s)$-solution} $u$, and a.s.,
\begin{align*}
u& \in H^{\theta,p}_{\rm loc}([0, \infty),w_{\a_0};H^{2-\s-2\theta,q}(\Tor^d;\R^2))\cap C([0,\infty);B^{0}_{q,p}(\Tor^d;\R^2))  ,\\
u& \in H^{\theta,r}_{\rm loc}(0, \infty;H^{1-2\theta,\zeta}(\Tor^d;\R^2))\ \ \forall \theta\in  [0,\tfrac{1}{2}), \ \forall r,\zeta\in (2,\infty),\\
u&\in C^{\theta_1,\theta_2}_{\rm loc}((0,\infty)\times \Tor^d;\R^{2})
\ \  \forall\theta_1\in  [0,1/2), \ \forall \theta_2\in (0,1).
\end{align*}
Moreover, there exists a $\zeta>2$ such that for all $0<s<T<\infty$, $\gamma>0$ and $L\geq 0$,
\begin{align}
\label{eq:tailestBrus1} \P\Big(\sup_{t\in [s,T]} \one_{\Gamma}\|u(t)\|_{L^{\zeta}}^{\zeta}\geq \g\Big)
&\leq \psi(\g)(1+\E\one_{\Gamma}\|u(s)\|_{L^{\zeta}}^{\zeta}),\\
\label{eq:tailestBrus2} \P\Big(\max_{i\in \{1, 2\}} \one_{\Gamma}\int_{s}^{T}\int_{\Tor^d}|u_i|^{\zeta-2}|\nabla u_i|^2\,\dd x\, \dd r\geq \g\Big)
&\leq \psi(\g)(1+\E\one_{\Gamma}\|u(s)\|_{L^{\zeta}}^{\zeta}),
\end{align}
where $\Gamma = \{\|u(s)\|_{L^{\zeta}}\leq L\}$ and $\psi$ does not depend on $L\geq 1$, $s$, $\gamma$ and $u$ and satisfies $\lim_{\gamma\to \infty}\psi(\gamma)=0$, and one can take $s=0$ if $u_0\in B^{\varepsilon}_{\zeta,\infty}(\T^d;\R^2)$ a.s.\ for some $\varepsilon>0$.

Furthermore, if $u_0\geq 0$ (component-wise) a.e.\ on $\O\times \Tor^d$, $\alpha_{2},\alpha_{0}, \beta_{1},\beta_{0}\geq 0$ a.e.\ on $\R_+\times\O\times \Tor^d$ and $g_{n,1}(\cdot,0,y_2)=g_{n,2}(\cdot,y_1,0)=0$ a.e.\ on $\R_+\times\O\times \Tor^d$  for all $y_1, y_2\in [0,\infty)$, then
$$
u\geq 0 \text{ (component-wise) a.e.\ on } (0,\sigma)\times \O\times \Tor^d.
$$
\end{theorem}
Using Theorem \ref{t:continuity_general} one can further weaken the assumption on the initial data.

The main idea in the proof of Theorem \ref{t:Brusselator} is to exploit the dissipative effect of the nonlinearity $-u_1 u_2^2 $ appearing in the $u_1$-equation of \eqref{eq:brusselator} to estimate certain quantities involving $u_1$ and the product $u_1u_2$. Then we use this information to estimate $u_2$. Note that there are no dissipative effects in the $u_2$-equation.

\begin{proof}[Proof of Theorem \ref{t:Brusselator}]
We only consider $d=2$, since $d=1$ can be obtained by adding a dummy variable. Let $h_0 = q+1$.

{\em Step 1: Local well-posedness.}
One can check that Assumption \ref{ass:reaction_diffusion_global}$(p,q,h_0,\s)$ and \ref{ass:admissibleexp}$(p,q,h_0,\s)$ hold. To obtain a $(p, \kappa_0, q, \delta)$-solution one can apply Theorem \ref{t:reaction_diffusion_global_critical_spaces} with $h$ replaced by $h_0$. At the same time, this gives the regularity properties for the paths of $u$ stated in Theorem \ref{t:Brusselator} on $(0,\sigma)$ instead of $(0,\infty)$. It remains to check $\sigma=\infty$.

{\em Step 2: Energy estimate.}
We claim that there is a $\zeta>2$ such that for all $0<s<T<\infty$
\begin{equation}
\label{eq:energy_estimate_brusselator}
 \mathcal{E}_{i}(s,T) := \sup_{t\in [s,\sigma\wedge T)} \|u_i(t)\|_{L^{\zeta}}^{\zeta} + \int_s^{\sigma\wedge T} |u_i|^{\m-2}|\nabla u_i|^2\,\dd x\,\dd r <\infty \  \ \text{ a.s.\ on }\{\sigma>s\},
\end{equation}
for $i\in \{1, 2\}$, and for $\Gamma = \{\sigma>s, \|u(s)\|_{L^{\zeta}}\leq L\}$, one has
\begin{equation}
\label{eq:energy_estimate_brusselator2}
\P(\one_{\Gamma}\mathcal{E}_{i}(s,T)>\g) \leq \psi(\gamma) \big(\E\one_{\Gamma}\|u(s)\|_{L^{\m}}^{\m}+1\big)\quad \text{ for }\    i\in \{1, 2\}.
\end{equation}
The proofs of \eqref{eq:energy_estimate_brusselator} and \eqref{eq:energy_estimate_brusselator2} are given in the next two sub-steps.

\emph{Step 2a: There exists $\m>2$ such that a.s.\ on $\{\sigma>s\}$,}
\begin{align*}
\wt{\mathcal{E}}_{1}(s,T):= \sup_{t\in [s,\sigma\wedge T)}\|u_1(t)\|_{L^{\m}}^{\m}+
\int_{s}^{\sigma\wedge T} \int_{\Tor^2} |u_1|^{\m-2}|\nabla u_1|^2\,\dd x\,\dd r+
\int_{s}^{\sigma\wedge T} \int_{\Tor^2}u_1^2 u_2^2\,\dd x\,\dd r<\infty.
\end{align*}
\emph{In particular, \eqref{eq:energy_estimate_brusselator} holds for $i=1$. Moreover, for all $\eta\in (0,1)$ there is a constant $M_{\eta}$ independent of $s$, $k$,  and $u$ such that}
\begin{align}\label{eq:MetaestBrus2d}
\E |\one_{\Gamma}\wt{\mathcal{E}}_{1}(s,T)|^{\eta} \leq M_{\eta} (1+\E \one_{\Gamma}\|u_1(s)\|_{L^\zeta}^{\zeta \eta}).
\end{align}

From Assumption \ref{ass:brusselator_two_dimensions}\eqref{it:brusselator_two_dimensions_2} and \eqref{eq:brusselator} it is elementary to check that there exists $M,\varepsilon_0,\varepsilon_1>0$ such that for every $\m\in \{2,2+\varepsilon_0\}$, $y\in\R^2$ and a.e.\ on $\R_+\times \O$,
\begin{equation*}
\begin{aligned}
\frac{{f}_1(\cdot,y)y_1 }{\m-1}
+\frac{1}{4(\ellip_1-\varepsilon) }\Big(\sum_{n\geq 1} |b_{n,1}(\cdot)| \, |g_{n,1}(\cdot,y)|\Big)^2
&+\frac12\|(g_{n,1}(\cdot,y))_{n\geq 1}\|_{\ell^2}^2 \leq M_0(1+|y|^2)- \varepsilon_1|y_1|^2|y_2|^2.
\end{aligned}
\end{equation*}
Moreover, we may suppose that $2+\varepsilon_0<\frac{2}{2-\delta}$.
Thus, by Lemma \ref{lem:pointwise_bound_suboptimal}, Assumption \ref{ass:dissipation_general_sub_optimal}$(\m)$ holds for $\m \in\{2,2+\varepsilon_0\}$. Now the boundedness of the first two terms in Step 2a follows by applying Lemmas \ref{lem:pointwise_bound_suboptimal} and \ref{l:high_integrability_reaction_diffusion_with_dissipation} with $\m = 2+\varepsilon_0$, $M_1 = M_2 = M_0$, and $\mathcal{R}_1 = 0$. The boundedness of the last term in Step 2a follows by applying Lemmas \ref{lem:pointwise_bound_suboptimal} and \ref{l:high_integrability_reaction_diffusion_with_dissipation} with $\m=2$, $M_1 = M_2 = M_0$, and
\[\mathcal{R}_1(t) = \int_{\T^2} \varepsilon_1 \one_{(0,\sigma)}(t)|u_1(t)|^2 |u_2(t)|^2 \, \dd x.\]
Moreover, \eqref{eq:MetaestBrus2d} (and thus \eqref{eq:energy_estimate_brusselator2} for $i=1$) follows from \eqref{eq:etagronwall} and the fact that $M_1=M_2 = M_0$ are constant.

\emph{Step 2b: \eqref{eq:energy_estimate_brusselator} for $i=2$ holds with $\m = 2+\varepsilon_0$, where $\varepsilon_0$ is as in Step 2a}.
By Assumption  \ref{ass:brusselator_two_dimensions}\eqref{it:brusselator_two_dimensions_2} we have, a.e.\ on $[0,\sigma)\times \O$,
\begin{align}
\label{eq:estimate_brusselator_two_d_u_2}
\frac{{f}_2(\cdot,u)u_2}{\zeta-1}
+\frac{1}{4(\ellip_2-\varepsilon) }\Big(\sum_{n\geq 1} |b_{n,2}(\cdot)|\, |g_{n,2}(\cdot,u)|\Big)^2
&+\frac12\|(g_{n,2}(\cdot,u))_{n\geq 1}\|_{\ell^2}^2\\
\nonumber
&\leq M_0( M_1 |u_2|^2+ M_2),
\end{align}
where $M_0>0$ is a constant, $M_1:=\one_{[s,\sigma)} |u_1|^4 + 1$ and
$M_2:= \one_{[s,\sigma)}|u_1| |u_2| +1$.
In the above estimate, we used that
$|u_1|^2|u_2|^2=|u_1|^{4/3} \big(|u_1|^{2/3} |u_2|^2\big)\lesssim  |u_1|^4+ |u_1| |u_2|^{3}$
as it follows from the Young inequality with exponents $(\frac{3}{2},3)$.

We claim that a.s.\
\begin{equation}
\label{eq:M_brusselator_estimates}
M_1\in L^{1}(s,T;L^{\zeta/2}(\Tor^2)) \quad \text{ and }\quad
M_2\in L^2(s,T;L^2(\Tor^2)) .
\end{equation}
We begin by noticing that the second assertion of \eqref{eq:M_brusselator_estimates} follows from Step 2a, and moreover by \eqref{eq:MetaestBrus2d} with $\eta=1/2$ we obtain
\begin{align}\label{eq:estM2Brusd2}
\E (\one_{\Gamma}\|M_2\|_{L^2(s,T;L^2(\T^2))})\lesssim  (1+\E \one_{\Gamma}\|u_1(s)\|_{L^\zeta}^{\zeta/2}).
\end{align}

For the first assertion of \eqref{eq:M_brusselator_estimates},
note that reasoning as in \eqref{eq:sharp_case_u_integrability}-\eqref{eq:sharp_case_u_integrability_zeta_0} we find that  a.s.
\begin{align*}
\||u_1|^{\zeta/2}\|_{L^{4}(s,\sigma\wedge T;L^{4}(\Tor^2))}&\lesssim  \||u_1|^{\zeta/2}\|_{L^{4}(s,\sigma\wedge T;H^{1/2}(\Tor^2))}
\\ &  \lesssim \||u_1|^{\zeta/2}\|_{L^{4}(s,\sigma\wedge T;H^{1/2}(\Tor^2))}
\\ &  \lesssim \sup_{t\in [s,\sigma\wedge T)}\||u_1(t)|^{\zeta/2}\|_{L^2(\Tor^2)} + \||u_1|^{\zeta/2}\|_{L^{2}(s,\sigma\wedge T;H^{1}(\Tor^2))}
\\ &  \lesssim \sup_{t\in [s,\sigma\wedge T)}\|u_1(t)\|_{L^\zeta(\Tor^2)}^{\zeta/2} + \| |u|^{(\zeta-2)/2} |\nabla u_1|\|_{L^{2}(s,\sigma\wedge T;L^2(\Tor^2))}<\infty.
\end{align*}
In particular, taking ${1/2}$-moments it follows from \eqref{eq:MetaestBrus2d} with $\eta = 2/\zeta$ that
\begin{equation}\label{eq:estM2Brusd1}
\begin{aligned}
\E\|\one_{\Gamma} M_1&\|_{L^1(s,T;L^{\zeta/2}(\T^2))}^{1/2}\lesssim 1+ \E\| \one_{\Gamma} |u_1|^{\zeta/2}\|_{L^{4}(s,\sigma\wedge T;L^{4}(\Tor^2))}^{4/\zeta}
\\ & \lesssim 1+
\E \sup_{t\in [s,\sigma\wedge T)} \one_{\Gamma} \|u_1(t)\|_{L^\zeta(\Tor^2)}^{2} + \E\|\one_{\Gamma}  |u|^{(\zeta-2)/2} |\nabla u_1|\|_{L^{2}(s,\sigma\wedge T;L^2(\Tor^2))}^{4/\zeta}
\\ & \lesssim 1+\E \one_{\Gamma}\|u_1(s)\|_{L^\zeta}^{2}.
\end{aligned}
\end{equation}

By Lemma \ref{lem:pointwise_bound_suboptimal},  \eqref{eq:estimate_brusselator_two_d_u_2}, \eqref{eq:M_brusselator_estimates}, and Lemma \ref{l:high_integrability_reaction_diffusion_with_dissipation} it follows that \eqref{eq:energy_estimate_brusselator} holds for $i=2$, and by \eqref{eq:estM2Brusd2} and \eqref{eq:estM2Brusd1} for all $\gamma>0$ and $R\geq \max\{4C,1\}$ we obtain
\begin{align*}
 \P(\one_{\Gamma}\mathcal{E}_{2}(s,T)>\g)
 & \leq  \frac{R e^{R}}{\g} + \P\big(\one_{\Gamma}\|u_{2}(s)\|_{L^{\m}}^{\m}+C \one_{\Gamma}\|M_1\|_{L^{1}(s,T;L^{\zeta/2}(\Tor^2))}\geq R\big) \\ & \ \ +  \P\big(C\one_{\Gamma}\|M_1\|_{L^{1}(s,T;L^{\zeta/2}(\Tor^2))}+C\one_{\Gamma}\|M_2\|_{L^2(s,T;L^2(\Tor^2))}\geq R\big)
 \\ & \lesssim  \frac{e^{2R}}{\g} + \frac{\E\one_{\Gamma}\|u_{2}(s)\|_{L^{\m}}^{\m}}{R^{1/2}} + \frac{\E\one_{\Gamma}\|M_1\|_{L^{1}(s,T;L^{\zeta/2}(\Tor^2))}^{1/2} }{R^{1/2}} + \frac{\E\one_{\Gamma}\|M_2\|_{L^2(s,T;L^2(\Tor^2))}}{R}
 \\ & \lesssim  \Big(\frac{e^{2R}}{\g}+\frac{1}{R^{1/2}}\Big) \Big(\E\one_{\Gamma}\|u(s)\|_{L^{\m}}^{\m}+1\Big).
\end{align*}
Taking $R = \max\{\log(\gamma^{1/4}),4C,1\}$, this gives \eqref{eq:energy_estimate_brusselator2} for $i=2$.

{\em Step 3: Global well-posedness.} Let $\varepsilon_0$ be as in Step 2a, let $q_0\in (2, 2+\varepsilon_0)$, and take $\zeta = 2+\varepsilon_0$.  By Theorem \ref{thm:blow_up_criteria}\eqref{it:blow_up_not_sharp_L} with $\zeta_1=\m$, $\delta_0 \in (1,\frac{h_0+1}{h_0}]$ small enough, and $p_0$ large enough, it follows that
\begin{equation*}
\P(s<\sigma<T)
=
\P\Big(s<\sigma<T,\, \sup_{t\in [s,\sigma)} \|u(t)\|_{L^{\m}}<\infty\Big)=0,
\end{equation*}
where we used \eqref{eq:energy_estimate_brusselator}.
Since $\sigma>0$ a.s.\ by Theorem \ref{t:reaction_diffusion_global_critical_spaces}, letting $s\downarrow 0$ we obtain $\P(\sigma<T)=0$.
Letting $T\to \infty$, we find that $\sigma=\infty$ a.s. Moreover, \eqref{eq:energy_estimate_brusselator2} implies the tail estimates \eqref{eq:tailestBrus1} and \eqref{eq:tailestBrus2} for $s>0$. The case $s=0$ for smooth initial data can be obtained as in Lemma \ref{l:high_integrability_reaction_diffusion} by letting $s\downarrow 0$.

{\em Step 4: Positivity}.
This is immediate from Proposition \ref{prop:positivity}.
\end{proof}

\subsubsection{The stochastic Brusselator in three dimensions}
\label{sss:brusselator_three_dimensional}
In this subsection, we prove global existence \eqref{eq:brusselator} in the case $d=3$. This will require a more detailed analysis compared to the case $d\in \{1, 2\}$ treated in Subsection \ref{sss:brusselator_two_dimensional} because
we cannot apply Theorem \ref{thm:blow_up_criteria}\eqref{it:blow_up_not_sharp_L} as this would require $\sup_{t\in [s,\sigma\wedge T)} \|u(t)\|_{L^{\zeta_1}}<\infty$ a.s. for some $\zeta_1>3$. For $u_1$ such an estimate holds, but for $u_2$ it seems that we can only deduce such an a priori energy estimate for $\zeta_1=3$. Instead, we will apply the sharper blow-up criterion of Theorem \ref{thm:blow_up_criteria}\eqref{it:blow_up_sharp_L}. This will require a more detailed analysis with energy estimates which are mixed in space and time.
To prove these new energy estimates for $u_2$, the positivity of $u_1$ and $u_2$ plays an essential role. Therefore, unlike in the case $d\in \{1, 2\}$ we need to restrict ourselves to those $g$ which ensure $u\geq 0$ from the beginning. We start by listing our assumptions.

\begin{assumption}
\label{ass:brusselator_three}
Let $d=3$, $p\in (2,\infty)$ and suppose that following:
\begin{enumerate}[{\rm(1)}]
\item\label{it:brusselator_three_1} Assumption \ref{ass:reaction_diffusion_global}$(p,q,h,\s)$ for $q=h=3$ and some $\s>1$.
\item\label{it:brusselator_three_alphabeta}  The mappings $\alpha_i,\beta_i:\R_+\times \O\times \Tor^3\to \R$ are bounded and $\Progress\otimes \Borel(\Tor^3)$-measurable.
\item\label{it:brusselator_three_2} a.e.\ on $\R_+\times \O\times \Tor^3$ and for all $y_1, y_2\geq 0$,
\begin{align*}
\alpha_2,\alpha_0,\beta_1,\beta_0 \ \ \text{are non-negative and} \ \
g_{n,1}(\cdot,y_1,0)=g_{n,2}(\cdot,0,y_2)=0.&
\end{align*}
\item\label{it:brusselator_three_3} There exist $M,\varepsilon>0$ such that for all $i\in \{1,2\}$, a.e.\ on $\R_+\times \O\times \Tor^3$ and  $y=(y_1,y_2)\in [0,\infty)^2$,
\begin{align*}
\frac{1}{4(\ellip_i -\varepsilon)}\Big(\sum_{n\geq 1} |b_{n,i}(\cdot)|\, |g_{n,i}(\cdot,y)|\Big)^2 +\frac12\|(g_{n,i}(\cdot,y))_{n\geq 1}\|_{\ell^2}^2\leq    N_i(t,x,y) & & \text{(growth condition),}
\end{align*}
where $\ellip_i$ are as in Assumption \ref{ass:reaction_diffusion_global}\eqref{it:ellipticity_reaction_diffusion} and
\begin{align*}
N_1(\cdot,y)
&:=M \big[1+|y_1|^2\big] + \frac{|y_1|^2|y_2|^2}{5}  ,\\
N_2(\cdot,y)
&:=M\big[1+(1  + |y_1|^2)|y_2|^2 +|y_1||y_2|^{3}+ |y_1|^4\big].
\end{align*}
\end{enumerate}
\end{assumption}

The main result of this subsection reads as follows.

\begin{theorem}[Brusselator in three dimensions]
\label{t:brusselator_three}
Suppose that Assumption \ref{ass:brusselator_three} holds.
Assume that $\s\in (1,\frac{4}{3}]$, $p\geq 3\vee \frac{2}{2-\s}$ and set $\a_{0}:=p(1-\frac{\s}{2})-1$.
Then for every $u_0\in L^0_{\F_0}(\O;L^3(\Tor^3;\R^2))$
with $u_0\geq 0$ (component-wise), there exists
a (unique) \emph{global} $(p,\a_0,q,\s)$-solution $u:[0,\infty)\times \O\times \Tor^3\to \R^2$ to the stochastic Brusselator (as described near \eqref{eq:brusselator}) satisfying $u\geq 0$ (component-wise) a.e.\ on $\R_+\times \O\times \Tor^3$, and a.s.,
\begin{align*}
u& \in  H^{\theta,p}_{\rm loc}([0, \infty),w_{\a_0};H^{2-\s-2\theta,q}(\Tor^3;\R^2))\cap C([0,\infty);B^{0}_{3,p}(\Tor^3;\R^2)),\\
u& \in  H^{\theta,r}_{\rm loc}(0, \infty;H^{1-2\theta,\zeta}(\Tor^3;\R^2)) \ \ \forall \theta\in  [0,\tfrac{1}{2}), \ \forall r,\zeta\in (2,\infty),\\
u&\in C^{\theta_1,\theta_2}_{\rm loc}((0,\infty)\times \Tor^3;\R^{2})
\ \ \forall\theta_1\in  [0,1/2), \ \forall \theta_2\in (0,1).
\end{align*}
Moreover, there for all $0<s<T<\infty$, $\gamma>0$ and $L\geq 0$
\begin{align}
\label{eq:tailestBrus13d}
\P\Big(\sup_{t\in [s,T]} \one_{\Gamma}\|u(t)\|_{L^{3}}^{3}\geq \g\Big)
&\leq \psi(\g)(1+\E\one_{\Gamma}\|u(s)\|_{L^{6}}^{6}),\\
\label{eq:tailestBrus23d} \P\Big(\max_{i\in \{1, 2\}} \one_{\Gamma}\int_{s}^{T}\int_{\Tor^d}|u_i||\nabla u_i|^2\,\dd x\,\dd r\geq \g\Big)
&\leq \psi(\g)(1+\E\one_{\Gamma}\|u(s)\|_{L^{6}}^{6}),
\end{align}
where $\Gamma = \{\|u(s)\|_{L^{6}}\leq L\}$ and $\psi$ does not depend on $L\geq 1$, $s$, $\gamma$ and $u$, and satisfies $\lim_{\gamma\to \infty}\psi(\gamma)=0$, and one can take $s=0$ if $u_0\in B^{\varepsilon}_{6,\infty}(\T^3;\R^2)$ a.s.\ for some $\varepsilon>0$.
\end{theorem}
Using Theorem \ref{t:continuity_general} one can further weaken the assumption on the initial data.

Before we start with the proof of Theorem \ref{t:brusselator_three} we show the a.s.\ boundedness of three consecutive parts: {\em the good}, {\em the bad}, and {\em the ugly part}.

\begin{itemize}
\item
{\em The good part:}
 $\sup_{t}\|u_1\|_{L^6_x}^6$, $\int_{(t,x)} |u_1|^4|\nabla u_1|^2\,\dd x\,\dd t$ and $\int_{(t,x)}u_1^2u_2^2\,\dd x\,\dd t$;
\item
{\em The bad part:}
$\sup_{t}\|u_2\|_{L^2_x}^2$ and $\int_{(t,x)} |\nabla u_2|^2\,\dd x\,\dd t$;
\item
{\em The ugly part:} $\sup_{t}\|u_2\|_{L^3_x}^3$ and $\int_{(t,x)} |u_2| |\nabla u_2|^2\,\dd x\,\dd t$.
\end{itemize}

The $L^3_x$-norm in the {\em ugly part} appears because we use the blow-up criterion of Theorem \ref{thm:blow_up_criteria} with $d=h=3$ to check global existence. The estimates in {\em the good part} and {\em the bad part} will be needed to estimate $(M_1,M_2)$ in {\em the ugly part}, see \eqref{eq:M_1_L_6_u_1} below.
As we will see {\em the good part} and {\em the ugly part} follow from Lemma \ref{l:high_integrability_reaction_diffusion} and \ref{l:high_integrability_reaction_diffusion_with_dissipation}, respectively. The proof of {\em the bad part} requires an additional argument based on the dissipative effect of $- u_1 u_2^2$ in the $u_1$-equation of \eqref{eq:brusselator}.

The existence of a (unique) $(p,\a_{\crit},\s,q)$-solution $(u,\sigma)$ to \eqref{eq:brusselator} with the required regularity up to the explosion time $\sigma$ follows from Theorem \ref{t:reaction_diffusion_global_critical_spaces} and \cite[Remark 2.8(c)]{AVreaction-local} with $h=d=3$. Moreover, by Assumption \ref{ass:brusselator_three}\eqref{it:brusselator_three_2} and Proposition \ref{prop:positivity}, we have
\begin{equation}
\label{eq:positivity_u_brusselator_3D}
u\geq 0\ \text{ (component-wise)}  \ \text{ a.e.\ on }[0,\sigma)\times \O\times \Tor^3.
 \end{equation}

\begin{proof}[Boundedness of the good part]
We will show that a.s.\ on $\{\sigma>s\}$,
\begin{align}
\label{eq:brusselator_three_claim_Step_1_1}
 \mathcal{E}_{11}(s,T):= \sup_{t\in [s,\sigma\wedge T) }\|u_1(t)\|_{L^{6}}^{6} +
\int_s^{\sigma\wedge T}\int_{\Tor^3} |u_1|^{4} |\nabla u_1|^2 \,\dd x\, \dd r&<\infty,\\
\label{eq:brusselator_three_claim_Step_1_2}
\mathcal{E}_{12}(s,T):=\int_s^{\sigma\wedge T }\int_{\Tor^3} |\nabla u_1|^2\,\dd x\,\dd r
+ \int_s^{\sigma\wedge T }\int_{\Tor^3} u_1^{2}u_2^2\,\dd x\,\dd r&<\infty,
\end{align}
and moreover, for every $\eta\in (0,1)$ there is a constant $M_{\eta}$ independent of $s$, $u$ and $L$ such that
\begin{align}\label{eq:MetaestBrus3d1}
\E |\one_{\Gamma}\mathcal{E}_{11}(s,T)|^{\eta} & \leq M_{\eta} (1+\E \one_{\Gamma}\|u_1(s)\|_{L^6}^{6\eta}),
\\ \label{eq:MetaestBrus3d2} \E |\one_{\Gamma}\mathcal{E}_{12}(s,T)|^{\eta} &\leq M_{\eta} (1+\E \one_{\Gamma}\|u_1(s)\|_{L^2}^{2\eta}).
\end{align}
The bounds \eqref{eq:brusselator_three_claim_Step_1_1}, \eqref{eq:brusselator_three_claim_Step_1_2}, and \eqref{eq:MetaestBrus2d} are proved in a similar way as in Theorem \ref{t:Brusselator}. Indeed, from Assumption \ref{ass:brusselator_three}\eqref{it:brusselator_three_3} and Lemma \ref{lem:pointwise_bound_suboptimal} one can check that
Assumption \ref{ass:dissipation_general_sub_optimal}$(\zeta)$ holds for $j=1$ and with $\zeta=6$, $M_1, M_2$ constant and $\mathcal{R}_1 = 0$.
Therefore, \eqref{eq:brusselator_three_claim_Step_1_1} and \eqref{eq:MetaestBrus3d1} are immediate from
Lemma \ref{l:high_integrability_reaction_diffusion_with_dissipation}.

For \eqref{eq:brusselator_three_claim_Step_1_2}, again by Assumption \ref{ass:brusselator_three}\eqref{it:brusselator_three_3} and Lemma \ref{lem:pointwise_bound_suboptimal} one has
that
Assumption \ref{ass:dissipation_general_sub_optimal}$(\zeta)$ holds for $j=1$, but this time with $\zeta=2$, $M_1, M_2$ constant and
\[\mathcal{R}_1(t) := \frac{4}{5}\int_{\Tor^3} u_1^2(t,x)u_2^2(t,x)\,\dd x.\]
Here the constant $\frac45$ is not important and can be replaced by a lower number. Therefore, the bounds \eqref{eq:brusselator_three_claim_Step_1_2} and \eqref{eq:MetaestBrus3d2} follow from Lemma \ref{l:high_integrability_reaction_diffusion_with_dissipation}.
\end{proof}

\begin{proof}[Boundedness of the bad part]
We show that for all $0<s<T<\infty$
\begin{equation}
\label{eq:brusselator_three_claim_Step_2}
\mathcal{E}_{21}:=\int_s^{\sigma\wedge T}\int_{\Tor^3} |u_2|^2+ |\nabla u_2|^2 \,\dd x\,\dd r<\infty,  \ \text{ a.s.\ on $\{\sigma>s\}$},
\end{equation}
and moreover, for every $\eta\in (0,1)$ there is a constant $M_{\eta}$ independent of $s$, $u$ and $L$ such that
\begin{align}\label{eq:brusselator_three_claim_Step_2est}
\E|\one_{\Gamma}\mathcal{E}_{21}|^\eta\leq M_{\eta} (1+\E \one_{\Gamma}\|u(s)\|_{L^6}^{4}).
\end{align}

Since there is no strong dissipation in the equation for $u_2$ due to the term $u_1 u_2^2$, this term is viewed as the bad term.
To prove \eqref{eq:brusselator_three_claim_Step_2} we need a new energy estimate, which requires a localization argument and It\^o's formula once more.

Fix $0<s<T<\infty$, $k\geq 1$, and let $(\Gamma,\tau_k)$ be as in the proof of  Lemma \ref{l:high_integrability_reaction_diffusion_with_dissipation}.
For $j\in \{1,2\}$, let
\begin{equation*}
u_j^{(k)}(t)=
\one_{\Gamma} u_j(t\wedge \tau_k) \ \ \text{ on }\O\times [s,T].
\end{equation*}
Note that $u_j^{(k)}\geq 0$ a.e.\ on $[s,T]\times \O\times \Tor^3$ by \eqref{eq:positivity_u_brusselator_3D}.

We claim that there exist $\theta,c_0>0$ such that for all $t\in [s,T]$,
\begin{equation}
\begin{aligned}
\label{eq:claim_Step_2_brusselator_three_d}
\one_{\Gamma} \|u_{2}^{(k)}(t)\|_{L^2}^2
&+
\theta \int_{s}^t \int_{\Tor^3} \one_{[s,\tau_{k}]\times \Gamma}  |\nabla u_2|^2\,\dd x\,\dd r
\leq  \one_{\Gamma}\|u(s)\|^2_{L^2}
\\
&\qquad
+ c_0
\int_{s}^t \int_{\Tor^3} \one_{[s,\tau_{k}]\times \Gamma}  | u_2|^2\,\dd x\,\dd r+ c_0\HH(t) + \MM(t),
\end{aligned}
\end{equation}
where $\MM$ is a continuous local martingale (which depends on $k$) and such that $\MM(s)=0$ and
\begin{equation}
\label{eq:def_H_3d_brusselator}
\HH(t):=
\int_s^t \int_{\Tor^3}  \one_{[s,\sigma) \times \Gamma} \big( 1+u_1^2u_2^2 + u_1^4 +
|\nabla u_1|^2\big)\,\dd x\,\dd r.
\end{equation}

To see that the claim \eqref{eq:claim_Step_2_brusselator_three_d} yields \eqref{eq:brusselator_three_claim_Step_2} one can argue as follows. The boundedness of the good part implies that $\HH(T)<\infty$ a.s.\ on $\{\sigma>s\}$. Therefore, \eqref{eq:brusselator_three_claim_Step_2}, \eqref{eq:brusselator_three_claim_Step_2est} follow from the stochastic Gronwall lemma \cite[Corollary 5.4b)]{geiss2021sharp}, and \eqref{eq:claim_Step_2_brusselator_three_d} and embedding $L^6(\T^3)\hookrightarrow L^\zeta(\T^3)$ for $\zeta\leq 6$. We divide the proof of \eqref{eq:claim_Step_2_brusselator_three_d} into three steps.

\emph{Step 1: There exists $\theta, c_1>0$ such that for all $t\in [s,T]$ and $\Gamma$ as above,}
\begin{align}
\label{eq:claim_Substep_2a_brusselator}
\one_{\Gamma} \|u_{2}^{(k)}(t)\|_{L^2}^2
&+
\theta \int_{s}^t \int_{\Tor^3} \one_{[s,\tau_{k}]\times \Gamma}  |\nabla u_2|^2\,\dd x\,\dd r
 \leq \one_{\Gamma}\|u_2(s)\|_{L^2}^2\\
 \nonumber
&+ c_1 \int_s^t \int_{\Tor^3} \one_{[s,\tau_{k}]\times \Gamma}
(u_2^2+u_1u_2^3)\,\dd x\, \dd r
 +\HH(t)+\MM_1(t),
\end{align}
\emph{where $\MM_1$ is a continuous local martingale (which depends on $k$) such that $\MM_1(s)=0$.}

As in \eqref{eq:identityIto_2}, by Assumption \ref{ass:reaction_diffusion_global}\eqref{it:ellipticity_reaction_diffusion} and applying the It\^{o} formula to $t\mapsto \|u_2^{(k)}(t)\|_{L^2}^2$ it follows that, a.s.\ for all $t\in [s,T]$,
\begin{align}
\label{eq:Ito_L_2_brusselator_three_d}
\one_{\Gamma}\|u_{2}^{(k)}(t)\|_{L^{2}}^{2}
+ \ellip_2 \int_{s}^t\int_{\Tor^3}\one_{[s,\tau_{k}]\times \Gamma}|\nabla u_2|^2\,\dd x\,\dd r
\leq\one_{\Gamma}\|u_2(s)\|_{L^2}^2+ I_1(t)+I_2(t)+\MM_1(t),
\end{align}
where
$\MM_1(t)$ is a local continuous martingale such that $\MM_1(s)=0$ and
\begin{align*}
I_1(t)&:=
2\int_{s}^{t} \int_{\Tor^d} \one_{[s,\tau_{k}]\times \Gamma}
 f_2(\cdot,u)u_2\,\dd x\,\dd r, \\
I_2(t)&:=
\sum_{n\geq 1}\int_{s}^{t} \int_{\Tor^d}\one_{[s,\tau_{k}]\times \Gamma} \Big(
 | g_{n,2}(\cdot,u)|^2
+ 2\sum_{n\geq 1} | (b_{n,2}\cdot \nabla) u_2| |g_{n,2}(\cdot,u)|\Big)\,
 \dd x\,\dd r.
\end{align*}
We estimate $I_1$ and $I_2$ separately. It is easy to see that a.s.\ for all $t\in [s,T]$ (using \eqref{eq:positivity_u_brusselator_3D}),
\begin{equation*}
I_1(t)
\lesssim \int_{s}^{t} \int_{\Tor^3} \one_{[s,\tau_{k}]\times \Gamma} (u^2_2+ u_1 u_2^3)\,\dd x\,\dd r+ \HH(t).
\end{equation*}
Moreover, by Assumption \ref{ass:brusselator_three}\eqref{it:brusselator_three_3}, one can check that also
\begin{align*}
I_2(t)\lesssim \int_{s}^{t} \int_{\Tor^d}\one_{[s,\tau_{k}]\times \Gamma} (1+|\nabla u_2|) N_2(\cdot, u)^{1/2} \dd  x\, \dd r \lesssim \int_{s}^{t} \int_{\Tor^3} \one_{[s,\tau_{k}]\times \Gamma} (u^2_2+ u_1 u_2^3)\,\dd x\, \dd r+  \HH(t).
\end{align*}
Hence, \eqref{eq:claim_Substep_2a_brusselator} follows by combining the previous estimates with \eqref{eq:Ito_L_2_brusselator_three_d}.

\emph{Step 2: For each $\varepsilon_0>0$, there exists a local continuous martingale $\MM_2$ (depending on $k$) such that  $\MM_2(s)=0$, and a.s.\ for all $t\in [s,T]$,}
\begin{equation}
\label{eq:claim_Substep_2b_brusselator}
\begin{aligned}
\int_{s}^t \int_{\Tor^3} \one_{[s,\tau_{k}]\times \Gamma}  u_1u_2^3\,\dd x\, \dd r
&\leq\one_{\Gamma} \|u_2(s)\|_{L^2}^2
+ C_{\varepsilon_0} \HH(t) +\MM_2(t)\\
& \ +
\int_{s}^t \int_{\Tor^3} \one_{[s,\tau_{k}]\times \Gamma}  \Big[C_0u_2^2 + \varepsilon_0|\nabla u_2|^2\Big]\,\dd x\,\dd r.
\end{aligned}
\end{equation}

Applying It\^o's formula to
$
(u_1,u_2)\mapsto \int_{\Tor^3} u_1 u_2\,\dd x
$
gives that a.s.\ for all $t\in [s,T]$,
\begin{align}
\label{eq:Ito_brusselator_step_2b}
\int_{\Tor^3} u_{1}^{(k)}(t,x) u_{2}^{(k)}(t,x)\,\dd x
\leq \one_{\Gamma}\int_{\Tor^3} u_{1}(s,x) u_{2}(s,x)\,\dd x
+\sum_{i=1}^3 J_i(t)  +\MM_2(t),
\end{align}
where
$\MM_2(t)$ is a local continuous martingale such that $\MM_2(s)=0$, and
\begin{align*}
J_1(t)&:=- \int_{s}^t  \int_{\Tor^3}\one_{[s,\tau_{k}]\times \Gamma} \Big[(a_1 \cdot \nabla u_1)\cdot\nabla u_2 + (a_2 \cdot \nabla u_2)\cdot\nabla u_1 \Big]\,\dd x\, \dd r,\\
J_2(t)&:= \int_{s}^t\int_{\Tor^3}  \one_{[s,\tau_{k}]\times \Gamma} \Big[f_1(\cdot,u)u_2+f_2(\cdot,u)u_1 \Big]\,\dd x\,\dd r,\\
J_3(t)&:= \sum_{n\geq 1} \int_{s}^t \int_{\Tor^3}  \one_{[s,\tau_{k}]\times \Gamma}
\big[(b_{n,1} \cdot\nabla) u_1+g_{n,1}(\cdot,u)\big]\big[(b_{n,2} \cdot\nabla) u_2+g_{n,2}(\cdot,u)\big]\,\dd x\,\dd r.
\end{align*}

The first term on the RHS\eqref{eq:Ito_brusselator_step_2b} can be estimates as:
$
\int_{\Tor^3} u_{1}(s,x) u_{2}(s,x)\,\dd x\leq \|u_1(s)\|^2_{L^2}+\|u_2(s)\|^2_{L^2}
$
a.s.\ on $\Gamma$.
Next, we estimate the remaining terms.
We begin with $J_1$. Since $a_i$ is bounded, it follows that for all $\varepsilon_0\in (0,1)$,
\begin{align*}
J_1(t)
\leq \frac{\varepsilon_0}{2}\int_{s}^t  \int_{\Tor^3}\one_{[s,\tau_{k}]\times \Gamma}
|\nabla u_2|^2\,\dd x\, \dd r
+ C_{\varepsilon_0} \HH(t).
\end{align*}
Next, we estimate $J_2$. From the definition of $f=(f_i)_{i=1}^2$ (see \eqref{eq:brusselator}) and boundedness of the coefficients $\alpha_i, \beta_i$ (see Assumption \ref{ass:brusselator_three}\eqref{it:brusselator_three_alphabeta}) we obtain that a.s.\ for all $t\in [s,T]$,
\begin{align*}
J_2(t)
&\leq - \int_{s}^t\int_{\Tor^3} \one_{[s,\tau_{k}]\times \Gamma}  u_1u_2^3\,\dd x\, \dd r
+C_0  \int_{s}^t\int_{\Tor^3} \one_{[s,\tau_{k}]\times \Gamma}  (1+ u_1^2u_2^2 + u_1^2+u_2^2)\,\dd x\,\dd r\\
&\stackrel{\eqref{eq:def_H_3d_brusselator}}{\leq} -\int_{s}^t\int_{\Tor^3} \one_{[s,\tau_{k}]\times \Gamma}  u_1u_2^3\,\dd x\, \dd r +C_0 \int_{s}^t\int_{\Tor^3} \one_{[s,\tau_{k}]\times \Gamma}  u_2^2\,\dd x\, \dd r
 + C_0 \HH(t).
\end{align*}

Finally, to estimate $J_3$, note that by the Cauchy-Schwartz' inequality, for all $t\in [s,T]$ and $\varepsilon_0\in (0,1)$
\begin{align*}
J_3(t)
&\leq \frac{\varepsilon_0}{2} \int_{s}^t  \int_{\Tor^3}\one_{[s,\tau_{k}]\times \Gamma}
(|\nabla u_2|^2+  \|(g_{n,2}(\cdot,u))_{n\geq 1}\|_{\ell^2}^2)\,\dd x\,\dd r\\
& \qquad +C_{\varepsilon_0}
\int_{s}^t  \int_{\Tor^3}\one_{[s,\tau_{k}]\times \Gamma}
(|\nabla u_1|^2+  \|(g_{n,1}(\cdot,u))_{n\geq 1}\|_{\ell^2}^2)\,\dd x\, \dd r
\\
&\stackrel{(i)}{\leq} \frac{\varepsilon_0}{2}  \int_{s}^t  \int_{\Tor^3}\one_{[s,\tau_{k}]\times \Gamma}
(u_2^2 +  |\nabla u_2|^2+  u_1 u_2^3)\,\dd x\,\dd r+ C_{\varepsilon_0}\HH(t),
\end{align*}
where in $(i)$ we used Assumption \ref{ass:brusselator_three}\eqref{it:brusselator_three_3} as well as \eqref{eq:def_H_3d_brusselator}.

Now \eqref{eq:claim_Substep_2b_brusselator} follows by combining the previous estimates with \eqref{eq:Ito_brusselator_step_2b}, and additionally using that $\int_{\Tor^3} u_{1}^{(k)}(t,x) u_{2}^{(k)}(t,x)\,\dd x\geq 0$ a.s.\ on $[s,T]$ as $u\geq 0$ (see \eqref{eq:positivity_u_brusselator_3D}).

\emph{Step 3: Proof of \eqref{eq:claim_Step_2_brusselator_three_d}}. Let $c_1>0$ be as in \eqref{eq:claim_Substep_2a_brusselator}. To prove \eqref{eq:claim_Step_2_brusselator_three_d}, it is suffices to use  \eqref{eq:claim_Substep_2b_brusselator} with $\varepsilon_0=(2c_1)^{-1}$ in \eqref{eq:claim_Substep_2a_brusselator}. Note that on the RHS of the corresponding estimate the term $\frac{1}{2}\int_{s}^t \int_{\Tor^3} \one_{[s,\tau_{k}]\times \Gamma}  |\nabla u_2|^2\,\dd x\,\dd r$ can be absorbed on the LHS as $\int_{s}^{\tau_{k}} \int_{\Tor^3} |\nabla u_2|^2\,\dd x\,\dd r\leq k$ a.s.\ on $\Gamma$ by the definition of $\tau_k$. This completes the proof of \eqref{eq:claim_Step_2_brusselator_three_d}.
\end{proof}

After these preparations, we now estimate the ugly part needed for the proof of Theorem \ref{t:brusselator_three}.
\begin{proof}[Boundedness of the ugly part]
We show that for all $0<s<T<\infty$,
\begin{equation}
\label{eq:energy_estimate_brusselator_three}
\mathcal{E}_{22}:= \sup_{t\in [s,\sigma\wedge T)} \|u_2(t)\|^3_{L^{3}}+\int_s^{\sigma\wedge T}\int_{\Tor^d} |u_2||\nabla u_2|^2\,\dd x\,\dd r<\infty \ \  \text{ a.s.\ on }\{\sigma>s\}.
\end{equation}
and for all $\gamma>0$
\begin{equation}\label{eq:energy_estimate_brusselator_threetail}
\P(\one_{\Gamma}\mathcal{E}_{22}(s,T)>\g) \leq \psi(\gamma) \big(\E\one_{\Gamma}\|u(s)\|_{L^{6}}^{4}+1\big).
\end{equation}
Indeed, we will deduce this from Lemma \ref{l:high_integrability_reaction_diffusion_with_dissipation} with $\zeta=3$ and $j=2$.

First observe that by Assumption \ref{ass:brusselator_three}\eqref{it:brusselator_three_3} on $(s,\sigma)\times \{\sigma>s\}\times \Tor^d$ it holds that
\begin{align*}
\frac{u_2 f_2(\cdot,u)}{2}&+ \frac{1}{4(\ellip_2-\varepsilon)}\Big(\sum_{n\geq 1} |b_{n,2}|\, |g_{n,2}(\cdot,u)|\Big)^2
 +\frac12\|(g_{n,2}(\cdot,u))_{n\geq 1}\|_{\ell^2}^2  \\ & \leq \frac{u_2 f_2(\cdot,u)}{2} + N_2(u)
\\ & \leq \frac12 \big[u_1u_2^3 + \beta_1 u_1 u_2 + \beta_2 u_2^2 + \beta_3u_2\big] + M\big[1+(1  + u_1^2)u_2^2 +u_1u_2^{3}+ u_1^4\big]
\\ & \leq M_1 + M_2|u|^2,
\end{align*}
where we take $M_1
:=
C \one_{[s,\sigma\wedge T)}(u_1^4+1)$,
and $M_2
:=
C\one_{[s,\sigma\wedge T)}(u_1 u_2+1)$ for a suitable constant $C$ depending on $\beta_i$ and $M$.

We claim that
$M_1 \in L^1(s,T;L^{3/2}(\Tor^3))$ and $M_2 \in L^{2}(s,T;L^{3}(\Tor^3))$.
The assertion for $M_1$ is equivalent to $u_1\in L^{4}(s,\sigma\wedge T ;L^{6}(\Tor^3))$ a.s., which is clear from the boundedness of the good part (see \eqref{eq:brusselator_three_claim_Step_1_1}),
and moreover by \eqref{eq:MetaestBrus3d1} with $\eta=2/3$, we see
\begin{align}\label{eq:M1bound3d}
\E \one_{\Gamma}\| M_1\|_{L^1(s,T;L^{3/2}(\Tor^3))} \lesssim (1+\E \one_{\Gamma}\|u_1(s)\|_{L^6}^{4}).
\end{align}
The required integrability of $M_2$ follows from
\begin{equation}
\label{eq:M_1_L_6_u_1}
\begin{aligned}
\|u_1 u_2\|_{L^2(s,T;L^3)}
&\leq \one_{\{\sigma> s\}}
\|u_1\|_{L^{\infty}(s,\sigma\wedge T;L^{6})}
\|u_2\|_{L^{2}(s,\sigma\wedge T;L^{6})} \\
&\stackrel{(i)}{\lesssim}
\one_{\{\sigma>s\}} \|u_1\|_{L^{\infty}(s,\sigma\wedge T;L^{6})}
\|u_2\|_{L^{2}(s,\sigma\wedge T;H^1)}\stackrel{(ii)}{<}\infty \ \text{a.s.},
\end{aligned}
\end{equation}
where in $(i)$ we used the Sobolev embedding $H^1(\Tor^3)\embed L^6(\Tor^3)$, and in $(ii)$ we used \eqref{eq:brusselator_three_claim_Step_1_1} and \eqref{eq:brusselator_three_claim_Step_2}. Moreover,  using \eqref{eq:MetaestBrus3d1}, and \eqref{eq:brusselator_three_claim_Step_2est} with $\eta = 2/3$ in \eqref{eq:M_1_L_6_u_1}, we find that
\begin{equation}\label{eq:M2bound3d}
\begin{aligned}
\E \one_{\Gamma}\| M_2\|_{L^2(s,T;L^{3}(\Tor^3))} &\lesssim
1+ \E\one_{\Gamma} \sup_{t\in [s, \sigma\wedge T)}\|u_1(t)\|_{L^{6}}^4 +
\E\one_{\Gamma}  \|u_2\|_{L^{2}(s,\sigma\wedge T;H^1)}^{4/3}
\\ & \lesssim
1+\E \one_{\Gamma}\|u(s)\|_{L^6}^{4}.
\end{aligned}
\end{equation}

From the above, Lemma \ref{lem:pointwise_bound_suboptimal} and Assumption \ref{ass:brusselator_three}\eqref{it:brusselator_three_3}, we see that Assumption \ref{ass:dissipation_general_sub_optimal} holds for $\zeta=3$ and $j=2$ with $R_2 = 0$. Therefore, by Lemma \ref{l:high_integrability_reaction_diffusion_with_dissipation}, \eqref{eq:energy_estimate_brusselator_three} holds. Moreover, from \eqref{eq:estimate_D_j}, \eqref{eq:M1bound3d}, and \eqref{eq:M2bound3d}, we obtain that for all $\gamma>0$, $R\geq\max\{4C, 1\}$
\begin{align*}
\P(\one_{\Gamma}\mathcal{E}_{22}(s,T)>\g) &\leq \frac{R e^{R}}{\g} + \P\big(\one_{\Gamma}\|u_2(s)\|_{L^{3}}^{3}+C \one_{\Gamma}\|M_1\|_{L^{1}(s,T;L^{3/2})}\geq R\big) \\ & \ \ +  \P\big(C\one_{\Gamma}\|M_1\|_{L^{1}(s,T;L^{3/2})}+C\one_{\Gamma}\|M_2\|_{L^2(s,T;L^3)}\geq R\big)
\\ & \lesssim \frac{e^{2R}}{\g} + \frac{\E\one_{\Gamma} \|u_2(s)\|_{L^{3}}^{3}}{R} + \frac{\E \one_{\Gamma}\|M_1\|_{L^{1}(s,T;L^{3/2})}}{R} + \frac{\E \one_{\Gamma}\|M_2\|_{L^{2}(s,T;L^{3})}}{R}
\\ & \lesssim \frac{e^{2R}}{\g} + \frac{1+\E\one_{\Gamma} \|u_2(s)\|_{L^{6}}^{4}}{R}.
\end{align*}
Taking $R = \max\{\log(\gamma^{1/4}),4C,1\}$, this gives \eqref{eq:energy_estimate_brusselator_threetail}.
\end{proof}

\begin{proof}[Proof of Theorem \ref{t:brusselator_three}] By \eqref{eq:brusselator_three_claim_Step_1_1}, \eqref{eq:brusselator_three_claim_Step_1_2} and \eqref{eq:energy_estimate_brusselator_three}, it holds that, for all $0<s<T<\infty$ and $i\in\{1,2\}$,
\begin{equation*}
 \mathcal{E}_{i}(s,T) := \sup_{t\in [s,\sigma\wedge T)} \|u_i(t)\|^3_{L^{3}}
 +\int_s^{\sigma\wedge T}\int_{\Tor^d} |u_i||\nabla u_i|^2\,\dd x\,\dd r<\infty \ \  \text{ a.s.\ on }\{\sigma>s\}.
\end{equation*}
For $i=1$ note that we used $|u_1|\leq |u_1|^4  + 1$.

Next, we will apply Theorem  \ref{thm:blow_up_criteria}\eqref{it:blow_up_sharp_L} with $d=3$ and $h_0=h=3$ (thus $\zeta_0 = \frac{d(h-1)}{2}=3$) in a similar way as in Step 2 of Theorem \ref{t:global_well_posedness_general}. We first claim that for all $\eta\in [0,1]$, $s>0$ and a.s.\
\begin{align}\label{eq:energy_estimate_brusselator_threealt}
u_i\in L^{3/\eta}(s,\sigma\wedge T;L^{9/(3-2\eta)}),  \ \ i\in \{1, 2\}.
\end{align}

Fix $i\in \{1,2\}$ and observe that from \eqref{eq:energy_estimate_brusselator_three} we see that for $v_i:=\one_{(s,\sigma\wedge T)}|u_i|^{3/2}\in L^\infty(s,T;L^2)\cap L^2(s,T;H^1)$ a.s.
By interpolation and the Sobolev embedding $H^1\embed L^6$, we find that
\begin{align*}
v_i\in
L^{\infty}(s,T;L^{2})\cap L^{2}(s,T;L^6)
\hookrightarrow L^{2/\eta}(s,T;L^{6/(3-2\eta)}) \ \  \text{ for }\eta\in [0,1],
\end{align*}
which implies the claim by rewriting the integrability in terms of $u_i$ again.

Now let $\delta_0\in (1, \frac53)$ be such that Assumption \ref{ass:reaction_diffusion_global}\eqref{it:regularity_coefficients_reaction_diffusion} holds.
Then one can readily check that there exist $q_0\in (3,\frac{3}{2-\s_0})$ and $p_0>\frac{2}{\s_0-1}\vee q_0$ such that $\frac{2}{p_0}+\frac{3}{q_0}=1$ and
Assumptions \ref{ass:reaction_diffusion_global}$(p_0,q_0,h_0,\s_0)$ and \ref{ass:admissibleexp}$(p_0,q_0,h_0,\s_0)$ hold. Observe that by the restriction on $\delta_0$ we have $q_0\in (3, 9)$. Hence, there exists $\eta_0\in (0,1)$ such that $q_0=\frac{9}{3-2\eta_0}$ and therefore $\frac{2}{p_0}=1-\frac{3}{q_0}=\frac{2\eta_0}{3}$, and thus $p_0 = \frac{3}{\eta_0}$.
From \eqref{eq:energy_estimate_brusselator_threealt} with $\eta\in \{0,\eta_0\}$ it follows that
\[\P(s<\sigma<T) = \P\Big(s<\sigma<T,\, \sup_{t\in [s,\sigma)}\|u(t)\|_{L^{3}}+
\|u\|_{L^{p_0}(s,\sigma;L^{q_0})} <\infty\Big)=0,\]
where in the last equality we applied Theorem \ref{thm:blow_up_criteria}\eqref{it:blow_up_sharp_L} with $h_0=d=3$ (and thus $\frac{d(h_0-1)}{2}=3$).
Letting $s\downarrow 0$ and using $\sigma>0$ a.s., we find that $\P(\sigma<T) = 0$. Letting $T\to \infty$ we get $\sigma = \infty$ a.s.

Finally, note that \eqref{eq:tailestBrus13d} follows from \eqref{eq:MetaestBrus3d1} and \eqref{eq:energy_estimate_brusselator_threetail}.
Similarly, \eqref{eq:tailestBrus23d} for $i=2$ follows from \eqref{eq:energy_estimate_brusselator_threetail}. For $i=1$ it follows from
\eqref{eq:MetaestBrus3d1} and \eqref{eq:MetaestBrus3d2} and the estimate $|u_1|\leq |u_1|^4  + 1$.
\end{proof}

\begin{example}[Stochastic Gray--Scott model]
In the deterministic setting, the Gray--Scott model has been proposed by P.\ Gray and S.K.\ Scott in \cite{GS83_model,GS84_model,GS85_model}
in the study of autocatalytic reactions.
Taking into account spatial diffusivity and small-scale fluctuations, stochastic perturbations of the Gray--Scott model are of the form \eqref{eq:reaction_diffusion_system} with $\ell=2$, $F\equiv 0$,
and $f_1(\cdot,u)= - u_1 u_2^2+ \g_1 u_1 + \eta_1$ and $f_2(\cdot,u)=u_1u_2^2+\g_2 u_2+ \eta_2$.
Here $u_1$ and $u_2$ denote the concentration of reactants in certain chemical reactions. Finally, $\g_i,\eta_i$ are $\Progress\otimes \Borel(\Tor^d)$-measurable maps to be determined experimentally.

It is easy to see that the stochastic diffusive Gray--Scott model is a special case of the Brusselator system \eqref{eq:brusselator}. Hence, under suitable assumptions on $(a,b,g,u_0)$, global well-posedness for this model in $d\in \{1, 2, 3\}$ follow from Theorems \ref{t:Brusselator} and \ref{t:brusselator_three} as well.
\end{example}

\section{Continuous dependence on initial data}
\label{s:continuous_dependence}
The aim of this section is to prove the continuity of global solutions of reaction-diffusion equations with respect to $u_0$ assuming (relatively weak) energy estimates. As we have shown in  Sections \ref{s:scalar_global}--\ref{s:global_system}, the validity of such energy estimates depends on the fine structure of the SPDE under consideration. However, they hold in many situations as we have seen in the applications of Sections \ref{s:scalar_global}--\ref{s:global_system}.

As in Section \ref{s:main_results}, let us consider the following system of SPDEs on $\T^d$ with initial data at time $s\geq 0$:
\begin{equation}
\label{eq:reaction_diffusion_system_continuity}
\left\{
\begin{aligned}
\dd u_i -\div(a_i\cdot\nabla u_i) \,\dd t&= \Big[\div(F_i(\cdot, u)) +f_i(\cdot, u)\Big]\,\dd t + \sum_{n\geq 1}  \Big[(b_{n,i}\cdot \nabla) u_i+ \rnoise_{n,i}(\cdot,u) \Big]\,\dd w_t^n,\\
u_i(s)&=u_{s,i},
\end{aligned}\right.
\end{equation}
where $i\in \{1,\dots,\ell\}$ and $\ell\geq 1$ is an integer. Note that \eqref{eq:reaction_diffusion_system_continuity} coincide with \eqref{eq:reaction_diffusion_system} in the case $s=0$.

\subsection{Main results}
We start by formulating the main assumption used in this section.
\begin{assumption}[Global energy estimate]
\label{ass:global_s}
Suppose the assumptions of Theorem \ref{t:reaction_diffusion_global_critical_spaces} are satisfied for $(p,q,h,\s)$ and with $\a_{\crit}$ as in Assumption \ref{ass:admissibleexp}.
Set $q_0:=\frac{d(h-1)}{2}\vee 2$.
Let $\zeta_0\in [q_0,\infty]$ be such that $2-\s-\frac{d}{q}>-\frac{d}{\zeta_0}$ and let $s_0\geq 0$.
We say that Assumption \ref{ass:global_s}$(s_0,\zeta_0)$ holds if we have the following:

For all initial data such that
$$
u_{s_0}\in L^{0}_{\F_{s_0}}(\O;C^{\theta}(\Tor^d;\R^{\ell}))\   \text{ for some $\theta>0$ and }  \ u_{s_0}\in L^{\zeta_0}(\O\times \Tor^d;\R^{\ell}),
$$
there exists a global $(p,\a_{\crit},\s,q)$-solution to \eqref{eq:reaction_diffusion_system_continuity}, and for all $T>s_0$ and $\gamma>0$
\begin{align*}
\P\Big(\sup_{t\in (s_0,T]} \|u(t)\|_{L^{q_0}}^{q_0}\geq \g\Big)
&\leq \psi_{s_0}(\g)(1+\E\|u_{s_0}\|_{L^{\zeta_0}}^{\zeta_0}),\\
\P\Big(\max_{1\leq i\leq \ell} \int_{s_0}^{T}\int_{\Tor^d}|u_i|^{q_0-2}|\nabla u_i|^2\,\dd x\,\dd t\geq \g\Big)
&\leq \psi_{s_0}(\g)(1+\E\|u_{s_0}\|_{L^{\zeta_0}}^{\zeta_0}),
\end{align*}
where $\psi$ does not depend on $\gamma$ and $u$, and satisfies $\lim_{\gamma\to \infty}\psi_{s_0}(\gamma)=0$.

Moreover, given $S\subseteq \R^{\ell}$ we say that Assumption \ref{ass:global_s}$(s_0,\zeta_0)$ holds for $S$-valued functions if the above is satisfies for all $u_{s_0}$ with values in $S$.
\end{assumption}

Due to the estimates proved in Section \ref{s:scalar_global}-\ref{s:global_system} and a translation argument, all (system of) SPDEs analyzed in these sections satisfy
Assumption \ref{ass:global_s}$(s_0,\zeta_0)$ for all $s_0\geq 0$ and some $\zeta_0\in [q_0,\infty]$, where in some cases we need to take $S = [0,\infty)^{\ell}$. Moreover, in all but one case, one can take $\zeta_0=q_0$. The exception is the 3d Brusselator of Subsection \ref{sss:brusselator_three_dimensional} where we need $6=\zeta_0>q_0=q=3$ (see \eqref{eq:tailestBrus13d}-\eqref{eq:tailestBrus23d}). In the latter case, $2-\s-\frac{d}{q}>-\frac{d}{\zeta_0}$ still holds, where without loss of generality, we can restrict ourselves to $\s<\frac{3}{2}$.
In particular, the results below apply to all SPDEs considered in the paper. Finally, an inspection of the following proofs shows that in Assumption \ref{ass:global_s} we may replace $\E\|u_{s_0}\|_{L^{\zeta_0}}^{\zeta_0}$ by $\E\|u_{s_0}\|_{L^{\zeta_0}}^{r_0}$ where $r_0\in (0,\infty)$ is an additional parameter which can be chosen independently of $\zeta_0$.

Next, we formulate the main results of this section.

\begin{theorem}[Continuous dependence on initial data]
\label{t:continuity_general}
Let the assumptions of Theorem \ref{t:reaction_diffusion_global_critical_spaces} be satisfied with $q,p\geq \frac{d(h-1)}{2}\vee 2$ and $p\geq q$. Let Assumption \ref{ass:global_s}$(s_0,\zeta_0)$ be satisfied for all $s_0\in (0,\infty)$ and some $S\subseteq \R^{\ell}$.
Let $u_0$ and $(u_{0}^{(n)})_{n\geq 1}$ be $S$ valued and such that
\[u_{0}^{(n)}\to u_0 \ \ \text{in} \ \ L^0_{\F_0}(\Omega;B^{\frac{d}{q}-\frac{2}{h-1}}_{q,p}(\Tor^d;\R^{\ell})).\]
Then there exist global $(p,\a_{\crit},\s,q)$-solutions $u$ and $u^{(n)}$ to \eqref{eq:reaction_diffusion_system_continuity} with initial data $u_0$ and $u_{0}^{(n)}$ at time $s=0$, respectively, and for all $T\in (0,\infty)$,
$$
u^{(n)}\to u \ \  \text{in} \ \ L^0(\Omega;C([0,T];B^{\frac{d}{q}-\frac{2}{h-1}}_{q,p}(\Tor^d;\R^{\ell}))).
$$

\end{theorem}

Note that the assumption $q\geq \frac{d(h-1)}{2}\vee 2$ implies that the critical space considered in Theorem \ref{t:continuity_general} has smoothness $\leq 0$. It is likely that the restriction $p\geq \frac{d(h-1)}{2}\vee 2$ can be removed, but it does not lead to any serious limitations.

Next, we state a similar result for critical Lebesgue spaces under the further restriction that Assumption \ref{ass:global_s}$(s_0,\zeta_0)$ holds for $s_0=0$ as well (which is the case in all applications). Recall that Theorem \ref{t:reaction_diffusion_global_critical_spaces} applies for $p$ large enough (see  \cite[Remark 2.8(c)]{AVreaction-local}).

\begin{proposition}[Continuous dependence on initial data in $L^q$]
\label{prop:L_m_continuity}
Let the
Assumptions \ref{ass:reaction_diffusion_global}$(p,q,h,\s)$ and \ref{ass:admissibleexp}$(p,q,h,\s)$ be satisfied with $q=\frac{d(h-1)}{2}\vee 2$ and for some $\s>1$, $p\geq q\vee \frac{2}{2-\s}$. Set $\a:=\a_{\crit}:=p\big(1-\frac{\s}{2}\big)-1$. Assume that Assumption \ref{ass:global_s}$(s,\zeta)$ holds and for all $s\geq 0$ and for some $\zeta\in [q,\infty]$, $S\subseteq \R^\ell$.
Let $u_0, u_0^{(n)}\in L^0_{\F_0}(\Omega;L^{\m}(\Tor^d;\R^{\ell}))$ be $S$-valued and such that
\[
u_0^{(n)}\to u_0 \ \ \text{in} \ \ L^0(\Omega;L^{\m}(\Tor^d;\R^{\ell})).
\]
Then there exist global $(p,\a_{\crit},\s,q)$-solutions $u$ and $u^{(n)}$  to \eqref{eq:reaction_diffusion_system_continuity} with initial data $u_0$ and $u_0^{(n)}$ at time $s=0$, respectively, and for all and $T\in (0,\infty)$
\begin{align*}
u^{(n)}&\to u \ \ L^0(\Omega;C([0,T];L^q(\Tor^d;\R^{\ell}))).
\end{align*}
\end{proposition}

Let us stress that $u^{(n)}, u\in C([0,T];L^q(\Tor^d;\R^{\ell}))$ a.s.\ is part of the assertion in the above.  It does not follow from Theorem \ref{t:reaction_diffusion_global_critical_spaces} because we only have $L^q\subsetneq B^{0}_{q,p}$ for $p\geq q$.

In particular, Proposition \ref{prop:L_m_continuity} implies that $\lim_{t\downarrow 0}u(t)= u_0$ in $L^0(\O;L^q(\Tor^d;\R^{\ell}))$. As a consequence, the energy estimates of Sections \ref{s:scalar_global} and \ref{s:global_system} extend to $s=0$ for initial data in $L^\zeta$:
\begin{corollary}
\label{cor:estimate_up_to_0}
The following energy estimates extend to $s=0$ for the initial data as stated:
\begin{enumerate}[{\rm(1)}]
\item\label{it:estimate_up_to_0_Lzeta} \eqref{eq:aprioribounds1} and \eqref{eq:aprioribounds2} with $\Gamma=\O$ for $u_0\in L^{\zeta}(\T^d;\R^{\ell})$.
\item \eqref{eq:tailestBrus1} and \eqref{eq:tailestBrus2} with $\Gamma=\O$ for $u_0\in L^{\zeta}(\T^d;\R^2)$.
\item \eqref{eq:tailestBrus13d} and \eqref{eq:tailestBrus23d} with $\Gamma=\O$ for $u_0\in L^{6}(\T^3;\R^{2})$.
\end{enumerate}
\end{corollary}
Before turning to the proofs of the above results we collect some observations.

\begin{remark}[More on continuity in Lebesgue spaces]
\label{r:continuity_Lq}
Suppose that the conditions of Proposition \ref{prop:L_m_continuity} are satisfied (in particular $q=\frac{d(h-1)}{2}\vee 2$). Let $T>0$ and $i\in \{1, \ldots, \ell\}$. Proposition \ref{prop:L_m_continuity} implies that $\lim_{t\downarrow 0}u_i(t)=u_{0,i}$ in $L^0(\Omega;L^q(\T^d))$. As a consequence, the following improvement can be obtained under further restrictions
\begin{enumerate}[{\rm(a)}]
\item\label{it:Lq_continuity_in_norma} If $r_0>0$ and $\dps \sup_{s\in [0,T]}\E \|u_i(s)\|_{L^{q}(\T^d)}^{r_0}<\infty$,
 then $\dps\lim_{t\downarrow 0}u(t)=u_{0}$ in $L^r(\O;L^q(\Tor^d))$ for all $r\in (0,r_0)$.
\item\label{it:Lq_continuity_in_normb}
If $\dps\limsup_{s\downarrow 0} \E \|u_i(s)\|_{L^q(\T^d)}^q\leq \E\|u_{0,i}\|_{L^q(\T^d)}^q<\infty$,
then $\dps\lim_{t\downarrow 0}u_i(t)=u_{0,i} $ in $L^q(\O;L^q( \Tor^d))$.
\end{enumerate}
Part \eqref{it:Lq_continuity_in_norma} follows from \cite[Theorem 5.12]{Kal}. The condition is satisfied in many of the applications.

To prove \eqref{it:Lq_continuity_in_normb} note that by Fatou's lemma and the convergence in probability
\[\E\|u_{0,i}\|_{L^q(\T^d)}^q  \leq \liminf_{s\downarrow 0} \E \|u_i(s)\|_{L^q(\T^d)}^q \leq \limsup_{s\downarrow 0} \E \|u_i(s)\|_{L^q(\T^d)}^q\leq \E\|u_{0,i}\|_{L^q(\T^d)}^q.\]
Hence, $\lim_{s\downarrow 0}\E\|u_i(s)\|_{L^q(\T^d)}^q = \E\|u_{0,i}\|_{L^q(\T^d)}^q$ and the required result follows again from \cite[Theorem 5.12]{Kal}. The condition in \eqref{it:Lq_continuity_in_normb} is satisfied in the case of fully dissipative systems discussed in Section \ref{s:scalar_global} in case where Assumption \ref{ass:dissipation_general}$(q)$ holds, cf.\ \eqref{eq:aprioribothinu} and Step 6 in the proof of Lemma \ref{l:high_integrability_reaction_diffusion} since $\sigma=\infty$ a.s.\ by Theorem \ref{t:global_well_posedness_general}.
\end{remark}

\subsection{Proofs of Theorem \ref{t:continuity_general}, Proposition \ref{prop:L_m_continuity} and Corollary \ref{cor:estimate_up_to_0}}
\label{ss:proofs_continuity_data}
We begin by proving the following key lemma. The following should be compared with \cite[Proposition 4.5]{AV22_variational}.

\begin{lemma}
\label{l:tail_estimate}
Let the assumptions of Theorem \ref{t:reaction_diffusion_global_critical_spaces} be satisfied, in particular $(p,\s,q)$ satisfy Assumption \ref{ass:admissibleexp} and $\a_{\crit}:=p\big(\frac{h}{h-1}-\frac{1}{2}(\reg+\frac{d}{q})\big)-1$.
 Set $q_0:=\frac{d(h-1)}{2}\vee 2$.
Let $s_0\geq 0$ be fixed and suppose that Assumption \ref{ass:global_s}$(s_0,\zeta_0)$ holds for some $S\subseteq \R^{\ell}$.
Fix
$$
u_{s_0},v_{s_0}\in L^{0}_{\F_{s_0}}(\O;C^{\theta}(\Tor^d;\R^{\ell}))\   \text{ for some $\theta>0$ such that }  \ u_{s_0},v_{s_0}\in L^{q_0}(\O\times \Tor^d;\R^{\ell}).
$$
Let $u$ and $v$ be the $(p,\a_{\crit},\s,q)$-solution to \eqref{eq:reaction_diffusion_system_continuity} with $S$-valued initial data $u_{s_0}$ and $v_{s_0}$, respectively.
Then for all  $T>s$,  $R>0$, there are constants $C_R$ and $C_T$ such that $\gamma>0$
\begin{equation}\label{eq:tail_estimateFeller}
\P\Big(\max_{i\in\{1, \ldots, \ell\}} \sup_{t\in [s,T]}\|u_i(t)-v_i(t)\|_{L^{q_0}}^{q_0}\geq \g  \Big)
 \leq
\frac{2e^{R}}{\g}\E\|u_{s_0}-v_{s_0}\|_{L^{q_0}}^{q_0}+ \wt{\psi}_{s_0}(R) (1+\E\|u_{s_0}\|_{L^{\zeta_0}}^{\zeta_0}+\E\|v_{s_0}\|_{L^{\zeta_0}}^{\zeta_0}),
\end{equation}
where $\wt{\psi}_{s_0}$ only depends on $\psi_{s_0}$, $T$, $q_0$, $\nu_i$ and $d$, and satisfies $\lim_{R\to \infty} \wt{\psi}_{s_0}(R) = 0$.
\end{lemma}
The tail estimate \eqref{eq:tail_estimateFeller} is only useful if $\E\|u_{s_0}-v_{s_0}\|_{L^{q_0}}^{q_0}$ is very small. For instance this is the case if $v_{s_0} = u_{s_0}^{(n)}$ converges to $u_{s_0}$ in $L^{q_0}(\Omega\times\T^d)$. After letting $n\to \infty$, we can let $R\to \infty$ to obtain a convergence result for the corresponding solutions.

\begin{proof}
By a translation argument, we can suppose $s_0=0$. From Assumption \ref{ass:global_s} and arguing as in Step 2 of Theorem \ref{t:global_well_posedness_general}, we obtain global existence and uniqueness for \eqref{eq:reaction_diffusion_system_continuity} with initial data $u_0$ and $v_0$. Moreover, by \cite[Proposition 3.1]{AVreaction-local} with large integrability parameters $(\wt{p}, \wt{q})$, $\delta=1$ and $\kappa=0$, we find that $u$ and $v$ are in $C([0,T]\times \T^d;\R^\ell)\cap L^{\wt{p}}(0,T;H^{1,\wt{q}}(\T^d;\R^\ell))$ a.s. Note that here we also used the compatibility result \cite[Proposition 3.5]{AVreaction-local}.

Let $w_i:=u_i-v_i$ and $w_{0,i}:=u_{0,i}-v_{0,i}$ for all $i\in \{1,\dots,\ell\}$. By It\^o's formula applied to $w_i\mapsto \|w_i\|_{L^{q_0}}^{q_0}$ (and an approximation argument as in \eqref{eq:identityIto}) we have, a.s.\ for all $t\in [0,T]$,
\begin{equation}
\label{eq:Ito_difference_continuity}
\|w_i(t)\|_{L^{q_0}}^{q_0}
+ \frac{\ellip_i q_0(q_0-1)}{2} \int_0^t \int_{\Tor^d} |w_i|^{q_0-2} |\nabla w_i|^2\,\dd x\,\dd s\leq
\|w_0\|_{L^{q_0}}^{q_0} + C_0\sum_{j=0}^2\int_0^t  J_j\,\dd s + M(t),
\end{equation}
where $M$ is a continuous local martingale such that $M(0)=0$, $C_0$ is a constant independent of $u_0,v_0$ and $J_j:[0,T]\times \Omega\to \R$ for $j\in \{0, 1, 2\}$ are given by
\begin{align*}
J_0 &:=
 \int_{\Tor^d}|w_i|^{q_0-1} |f_i(\cdot,u)-f_i(\cdot,v)|\,\dd x ,\\
J_1&:= \int_{\Tor^d}|F_i(\cdot,u)-F_i(\cdot,v)| |w|^{q_0-2} |\nabla w|\,\dd x, \\
J_2&:=\sum_{n\geq 1} \int_{\Tor^d} |w_i|^{q_0-2}|g_{n,i}(\cdot,u)-g_{n,i}(\cdot,v)|^{2}\,\dd x.
\end{align*}
Here and below we omitted the $(t,x)$-dependency for notational brevity.

\emph{Step 1: For all $\varepsilon>0$ there exists $C_{\varepsilon}>0$ such that a.s.\ pointwise in $[0,T]$,}
\begin{equation*}
\begin{aligned}
\sum_{j=0}^2 J_{j}
&\leq \varepsilon \max_{1\leq i\leq \ell} \int_{\Tor^d} |w_i|^{q_0-2}|\nabla w_i|^2\,\dd x +  C_{\varepsilon}(1+ \|u\|_{L^{q_0+h_0-1}(\Tor^d;\R^{\ell})}^{q_0+h_0-1}+ \|v\|_{L^{q_0+h_0-1}(\Tor^d;\R^{\ell})}^{q_0+h_0-1}) \|w\|_{L^{q_0}(\Tor^d;\R^{\ell})}^{q_0},
\end{aligned}
\end{equation*}
\emph{where $h_0:=1+\frac{2}{d}q_0\geq h$ and $q_0=\frac{d(h-1)}{2}\vee 2$.}

We begin by estimating $J_0$. Assumption \ref{ass:reaction_diffusion_global}\eqref{it:growth_nonlinearities}  and H\"{o}lder's inequality gives
\begin{equation}
\begin{aligned}
\label{eq:estimate_J0}
J_0
&\lesssim
 \int_{\Tor^d} (1+ |u|^{h_0-1}+|v|^{h_0-1})|w|^{q_0}\,\dd x\\
&\lesssim (1+\|u\|_{L^{q_0+h_0-1}}^{h_0-1}
+\|v\|_{L^{q_0+h_0-1}}^{h_0-1})\|w\|_{L^{q_0+h_0-1}}^{q_0}\\
&\lesssim (1+\|u\|_{L^{q_0+h_0-1}}^{h_0-1}
+\|v\|_{L^{q_0+h_0-1}}^{h_0-1})\big(\max_{1\leq i\leq \ell}\|w_i\|_{L^{q_0+h_0-1}}^{q_0}\big).
\end{aligned}
\end{equation}
Note that $q_0+h_0-1=(\frac{d+2}{2})(h_0-1)=\frac{d+2}{d}q_0$. Fix $i\in \{1,\dots,\ell\}$. By Sobolev embedding and interpolation estimates we find that with $\theta=\frac{d}{d+2}$,
\begin{equation}
\begin{aligned}
\label{eq:estimate_w_i}
\|w_i\|_{L^{q_0+h-1}(\Tor^d)}^{q_0}
&=\big\||w_i|^{q_0/2}\big\|_{L^{\frac{2(d+2)}{d}}}^{2}
\\
&\lesssim \big\||w_i|^{q_0/2}\big\|_{H^{\theta}}^{2}
\\
&\leq
 \big\||w_i|^{q_0/2}\big\|_{L^2}^{2(1-\theta)}\big\||w_i|^{q_0/2}\big\|_{H^{1}}^{2\theta}\\
&\lesssim \|w_i\|_{L^{q_0}}^{q_0} +  \|w_i\|_{L^{q_0}}^{q_0(1-\theta)}
\Big(\int_{\Tor^d} |w_i|^{q_0-2}|\nabla w_i|^2\,\dd x\Big)^{2\theta}.
\end{aligned}
\end{equation}
Since $\frac{h_0-1}{1-\theta}=q_0+h_0-1$, the $J_0$-part of Step 1 follows by combining \eqref{eq:estimate_w_i} in \eqref{eq:estimate_J0} with Young's inequality. To estimate $J_1,J_2$, note that by Assumption \ref{ass:reaction_diffusion_global}\eqref{it:growth_nonlinearities} and Cauchy-Schwarz' inequality we have
\begin{align*}
J_1+J_2\leq \frac{\varepsilon}{4} \int_{\Tor^d} |w_i|^{q_0-2}|\nabla w_i|^2\,\dd x+
C_{\varepsilon}  \int_{\Tor^d} (1+|u|^{h_0-1}+|v|^{h_0-1})|w|^{q_0}\,\dd x.
\end{align*}
Hence, we can apply the previous argument to estimate the last term in the above estimate.

\emph{Step 2: Let $(\psi_0,\zeta_0)$ be as in Assumption \ref{ass:global_s}$(s_0,\zeta_0)$ with $s_0=0$. Then there exists a constant $K>0$ depending on $q_0,h_0,d$ such that  for each $z\in \{u, v\}$ and all $r>0$}
\begin{align*}
\P(\|z\|_{L^{q_0+h_0-1}((0,T)\times \Tor^d;\R^\ell)}^{q_0+h_0-1}> r)\leq K \psi_{0}(r^{\theta}) (1+\E\|z_0\|_{L^{\zeta_0}}^{\zeta_0}),
\end{align*}
{\em where $\theta = \frac{d}{d+2}$.}
We argue as in \eqref{eq:estimate_w_i} with an additional integration in time, and use Young's inequality to obtain
\begin{align*}
\|u_i\|^{q_0}_{L^{q_0+h_0-1}((0,T)\times \Tor^d)}
&= \big\||u_i|^{q_0/2}\big\|_{L^{\frac{2(d+2)}{d}}((0,T)\times \Tor^d)}^{2}\\
&\lesssim \big\||u_i|^{q_0/2}\big\|_{L^{\infty}(0,T;L^2)}^{2(1-\theta)}
\big\||u_i|^{q_0/2}\big\|_{L^{2}(0,T;H^1)}^{2\theta}\\
&\lesssim \|u_i\|_{L^{\infty}(0,T;L^{q_0})}^{q_0}+
\int_0^T\int_{\Tor^d} |u_i|^{q_0-2}|\nabla u_i|^{2}\,\dd x\,\dd s.
\end{align*}
Now the claim of this step follows from Assumption \ref{ass:global_s}$(0,\zeta_0)$ and $\frac{q_0+h_0-1}{q_0} = \frac{1}{\theta}$.

\emph{Step 3: Conclusion}.
Combining the estimates of Steps 1 in \eqref{eq:Ito_difference_continuity} and choosing $\varepsilon>0$ small enough, we find that
\begin{align*}
X^w(t) &: = \max_{i\in \{1, \ldots, \ell\}}\|w_i(t)\|_{L^{q_0}}^{q_0}
+ \max_{i\in \{1, \ldots, \ell\}} \frac{\ellip_i q_0(q_0-1)}{4} \int_0^t \int_{\Tor^d} |w_i|^{q_0-2} |\nabla w_i|^2\,\dd x\,\dd s
\\ & \leq
\|w_0\|_{L^{q_0}}^{q_0} + C_{\varepsilon} \int_0^t (1+ a^{u}(s) + a^v(s))X^w(s) \, \dd s + M(t),
\end{align*}
where $a^z = \|z\|_{L^{q_0+h_0-1}(\Tor^d;\R^{\ell})}^{q_0+h_0-1}$ for $z\in \{u,v\}$. Therefore, by the stochastic Gronwall lemma \cite[Corollary 5.4b)]{geiss2021sharp} we obtain that for all $\gamma>0$, $\lambda>0$, $R>0$
\begin{align*}
\P(X^{w}(t)>\gamma)\leq \frac{e^{R}}{\gamma} \E(\|w_0\|_{L^{q_0}}^{q_0}\wedge \lambda) + \P(\|w_0\|_{L^{q_0}}^{q_0}>\lambda) + \P\Big(C_{\varepsilon}\int_0^T (1+ a^{u}(t) + a^v(t))\,\dd t>R\Big).
\end{align*}
By Step 2 and Assumption \ref{ass:global_s} for each $z\in \{u,v\}$ and all $r>0$ we can estimate
\begin{align*}
\P\Big(\int_0^T a^{z}(t)\, \dd t>r\Big) \leq K \psi_0(r^{\theta}) (1+\E\|z_0\|_{L^{\zeta_0}}^{\zeta_0}).
\end{align*}
After setting $r = \frac{R}{2C_{\varepsilon}} - \frac{T}{2}$ we obtain that
\begin{align*}
\P\Big(C_{\varepsilon}\int_0^T (1+ a^{u}(t) + a^v(t))\, \dd t>R\Big) & \leq \sum_{z\in \{u,v\}} \P\Big(\int_0^T a^{z}(t)\, \dd t>\frac{R}{2C_{\varepsilon}} - \frac{T}{2}\Big) \\ & \leq K \psi_0(r^{\theta}) (1+\E\|u_0\|_{L^{\zeta_0}}^{\zeta_0} + \E\|v_0\|_{L^{\zeta_0}}^{\zeta_0}).
\end{align*}
Therefore, taking $\lambda = \gamma/e^{R}$ we can conclude
\begin{align*}
\P(X^{w}(t)>\gamma) \leq \frac{2e^{R}}{\gamma} \E\|w_0\|_{L^{q_0}}^{q_0} + K \psi_0(r^{\theta}) (1+\E\|u_0\|_{L^{\zeta_0}}^{\zeta_0} + \E\|v_0\|_{L^{\zeta_0}}^{\zeta_0}).
\end{align*}
Hence, the required assertion is proved with $\wt{\psi}_0(R) =  K\psi_0(r^{\theta}) = K\psi_0((\tfrac{R}{2C_{\varepsilon}} - \frac{T}{2})^{\theta})$ for $R\geq C_{\varepsilon}T$ and $\wt{\psi}(R) = 1$ otherwise.
\end{proof}

We are ready to prove Theorem \ref{t:continuity_general}. The main idea is to combine local well-posedness, say on $[0,T_0]$ for $T_0>0$ (see e.g.\ \cite[Proposition 2.9]{AVreaction-local}), and then to use Lemma \ref{l:tail_estimate} on the time interval $[t,T]$ where $0<t<T_0$.
The advantage is that on $[t,T_0]$ we may use the instantaneous regularization of anisotropic weighted function spaces (see e.g.\ \cite[Theorem 1.2]{ALV21} with positive weights) to go from the (possible) negative smoothness of the space $B^{\frac{d}{q}-\frac{2}{h-1}}_{q,p}$ if $q>\frac{d(h-1)}{2}$ to a space with positive smoothness so that Lemma \ref{l:tail_estimate} is applicable.

\begin{proof}[Proof of Theorem \ref{t:continuity_general}]
To economize the notation we let
$
\Y_p:=B^{\frac{d}{q}-\frac{2}{h-1}}_{q,p}.
$
Let us begin by noticing that, as Assumption \ref{ass:global_s}$(s,\zeta_0)$ holds for all $s>0$, the SPDE \eqref{eq:reaction_diffusion_system_continuity} has a \emph{global} $(p,\a_{\crit},\s,q)$-solution $u$ for all $u_0\in L^0_{\F_0}(\O;\Y_p)$. To see this, recall that the existence of a $(p,\a_{\crit},\s,q)$-solution $(u,\sigma)$ follows from Theorem \ref{t:reaction_diffusion_global_critical_spaces} with $\sigma>0$ a.s.\ and that \eqref{eq:reaction_diffusion_C_alpha_beta} holds. In particular, $\one_{\{\sigma>s\}}u(s)\in L^0_{\F_s}(\O;C^{\theta})$ for all $\theta\in (0,1)$. Now Assumption \ref{ass:global_s}$(s,\zeta_0)$ for all $s>0$ and the proof of Step 2 of Theorem \ref{t:global_well_posedness_general}, shows that there exists a \emph{global} $(p,\a_{\crit},\s,q)$-solution $v$ to \eqref{eq:reaction_diffusion_system_continuity} on $[s,\infty)$ with initial data $\one_{\{\sigma>s\}} u(s)$. Moreover, $v$ has the regularity stated in \eqref{eq:reaction_diffusion_H_theta}-\eqref{eq:reaction_diffusion_C_alpha_beta}.
Clearly, $(u\one_{\{\sigma>s\}}, \sigma\vee s)$ is a $(p,\a_{\crit},\s,q)$-solution to \eqref{eq:reaction_diffusion_system_continuity} on $[s,\infty)$ as well.
By maximality of the $(p,\a_{\crit},\s,q)$-solution $v$, we have $v=u$ on $[s,\sigma)$. Let $\beta_0 = \frac{d}{q} - \frac{2}{h-1}$ and $\gamma_0 = \frac{d}{q}+\frac{2}{p}-\frac{2}{h-1}$.
Then using the regularity of $v$ and $u=v$ on $[s,\sigma)$, we find that
\begin{align*}
\P(s<\sigma<T)
&= \P\Big(s<\sigma<T, \sup_{t\in [0,\sigma)}\|v(t)\|_{L^{\zeta_0}}<\infty\Big)\leq \P\Big(s<\sigma<T,\, \sup_{t\in [s,\sigma)}\|u(t)\|_{\zeta_0}<\infty\Big) = 0,
\end{align*}
where in the last step we used \cite[Theorem 2.10(1)]{AVreaction-local} .
Letting $T\to \infty$ and $s\downarrow 0$ we find that $\P(0<\sigma<\infty) = 0$. Since $\sigma>0$ a.s., we can conclude $\sigma = \infty$.

\emph{Step 1: Reduction to the case where $(u_0^{(n)})_{n\geq 1}$ are uniformly bounded as $\Y_p$-valued random variables}. In this step, we assume the statement of Theorem \ref{t:continuity_general} holds for bounded initial data.
Let $(u_0^{(n)})_{n\geq 1}$ be as in the statement of Theorem \ref{t:continuity_general}. Let $\eta>0$. Using Egorov's theorem, one can see that there exists an $\O_{0}\in \F_0$ and a constant $K$ such that $\P(\Omega_0)>1-\eta$ and
$\|u_0^{(n)}\|_{\Y_p}\leq K$ on $\Omega_0$. In particular, $\|u_0\|_{\Y_p}\leq K$ a.s.\ on $\Omega_0$.

Let $\wt{u}$ and $\wt{u}^{(n)}$ be the global $(p,\a_{\crit},\s,q)$-solution to \eqref{eq:reaction_diffusion_system}  with data $\one_{\O_0}u_0$ and $\one_{\O_0}u_0^{(n)}$, respectively. By localization (see \cite[Theorem 4.7(4)]{AV19_QSEE_1}),
$$
\wt{u}= u\quad  \text{ and }\quad  \wt{u}^{(n)}=u^{(n)} \ \ \text{a.s.\ on $[0,\infty)\times \O_0$.}
$$
Thus
\begin{align*}
\P\Big(\sup_{t\in [0,T]}\|u(t)-u^{(n)}(t)\|_{\Y_p}\geq \varepsilon\Big)
&\leq
\P\Big(\sup_{t\in [0,T]}\|\wt{u}(t)-\wt{u}^{(n)}(t)\|_{\Y_p}\geq \varepsilon\Big)+
\P(\O\setminus \O_0)\\
&\leq
\P\Big(\sup_{t\in [0,T]}\|\wt{u}(t)-\wt{u}^{(n)}(t)\|_{\Y_p}\geq \varepsilon\Big)
+\eta.
\end{align*}
Since $(\one_{\O_0}u_0^{(n)})_{n\geq 1}$ is uniformly bounded, the assumption of Step 1 implies
\[\limsup_{n\to \infty}\P\Big(\sup_{t\in [0,T]}\|u(t)-u^{(n)}(t)\|_{\Y_p}\geq \varepsilon\Big)\leq \eta. \]
Since $\eta>0$ was arbitrary, this completes step 1.

\emph{Step 2: Splitting argument}. By Step 1 we may suppose that there is a constant $K$ such that $\|u_0^{(n)}\|_{\Y_p}\leq K$ a.s. Therefore, $u_{0}^{n}\to u_0$ in $L^{r}(\O;\Y_p)$ for all $r\in [1,\infty)$.

By \cite[Proposition 2.9]{AVreaction-local} with parameters $(p,\a_{\crit},\s,q)$, there exist $T_0,C_0,N_0>0$ and stopping times $\sigma_0,\sigma_1,\sigma_0^{(n)}\in (0,\infty)$ such that for all $\g>0$, $n\geq N_0$ and $t\in [0,T_0]$,
\begin{align}
\label{eq:local_continuity_p_setting_1}
\P\Big(\sup_{s\in [0,t]}\|u(s)-u_{0}^{(n)}(s)\|_{\Y_p}\geq \g ,\ \sigma_0\wedge \sigma_0^{(n)}>t\Big)
\leq \frac{C_0}{\g^p}\E\|u_0-u_0^{(n)}\|_{\Y_p}^p,&\\
\label{eq:local_continuity_p_setting_2}
\P(\sigma_0\wedge \sigma_0^{(n)}\leq t)
\leq C_0\big[\E\|u_0-u_0^{(n)}\|_{\Y_p}^p+\P(\sigma_1\leq t)\big].&
\end{align}

To exploit instantaneous regularization results for anisotropic spaces (see e.g.\ \cite[Theorem 1.2]{ALV21}), we also use \cite[Theorem 2.7]{AVreaction-local} in another situation.
Fix $r>  p\vee \zeta_0$ so large that $2-\s-\frac{2}{r}-\frac{d}{q}>-\frac{d}{\zeta_0}$ (recall that  $2-\s-\frac{d}{q}>-\frac{d}{\zeta_0}$ by Assumption \ref{ass:global_s}). The choice of $r$ and Sobolev embedding yield $B_{q,r}^{2-\s-\frac{2}{r}}\embed L^{\zeta_0}$.
Let $\mu:=r(\frac{h}{h-1}-\frac{1}{2}(\s+\frac{d}{q}))-1$.
By Theorem \ref{t:reaction_diffusion_global_critical_spaces} and \cite[Remark 3.4 and (3.31)]{AVreaction-local} with parameters $(r,\mu,\s,q)$ and \cite[Theorem 1.2]{ALV21}, there exist $T_1,C_1,N_1>0$ and stopping times $\tau_0,\tau_1,\tau_0^{(n)}\in (0,\infty)$ such that for all $\g>0$, $n\geq N_1$, and $t\in [0,T_1]$,
\begin{align}
\label{eq:proof_continuity_K_estimate}
\E\Big[\one_{\{\tau_0^{(n)}>t\}} \sup_{s\in (0,t]}\big(s^{\mu} \|u^{(n)}(s)\|_{L^{\zeta_0}}^r\big)\Big] &
\leq C_1,\\
\label{eq:local_continuity_r_setting_1}
\E\Big[\one_{\{\tau_0\wedge\tau_0^{(n)}>t\}}\sup_{s\in (0,t]}\big(s^{\mu} \|u(s)-u_{0}^{(n)}(s)\|^r_{L^{\zeta_0}}\big)\Big]
&\leq  C_1\E\|u_0-u_0^{(n)}\|_{\Y_r}^r,\\
\label{eq:local_continuity_r_setting_2}
\P(\tau_0\wedge \tau_0^{(n)}\leq t)  &
\leq C_1\big[\E\|u_0-u_0^{(n)}\|_{\Y_r}^r+\P(\tau_1\leq t)\big].
\end{align}
Although the estimate \eqref{eq:proof_continuity_K_estimate} is not contained in the statement of \cite[Theorem 2.7]{AVreaction-local}, it readily follows from the choice of $\tau_0$ in its proof and the fact that $\sup_{n\geq 1}\E\|u_0^{(n)}\|_{\Y_r}^r<\infty$. Note also that we used the compatibility result of \cite[Proposition 3.5]{AVreaction-local} to ensure that the solution provided by Theorem \ref{t:reaction_diffusion_global_critical_spaces} applied with $(r,\mu,\s,q)$ coincide with $u$.

With the above-introduced notation, we can now the assertion of the theorem. To begin, let us set $T_*:=T_0\wedge T_1$ and $\varphi_i:=\sigma_i\wedge \tau_i$ for $i\in \{0,1\}$ and $\varphi^{(n)}:=\sigma_0^{(n)}\wedge \tau_0^{(n)}$ for all $n\geq 1$.
Fix $t\in (0,T_*]$, $\g>0$ and $n\geq N_0\vee N_1$.
For $0\leq t_1\leq t_2\leq T$ and $\psi\in \{\sigma, \tau,\varphi\}$ we write
\[I^{\psi}_{t_1, t_2}:= \P\Big(\sup_{s\in [t_1,t_2]}\|u(s)-u^{(n)}(s)\|_{\Y_p}\geq \g,\ \psi_0\wedge \psi_0^{(n)}> t \Big).
\]
If $t_1\leq t_2\leq t_3$, then $I^{\psi}_{t_1, t_3}\leq I^{\psi}_{t_1, t_2} + I^{\psi}_{t_2, t_3}$ and $I^{\varphi}_{t_1, t_2}\leq \min\{I^{\sigma}_{t_1, t_2},I^{\tau}_{t_1, t_2}\}$. Therefore, from \eqref{eq:local_continuity_p_setting_1}-\eqref{eq:local_continuity_p_setting_2} and \eqref{eq:local_continuity_r_setting_2}, we obtain
\begin{align*}
\P\big(\sup_{t\in [0,T]}\|u(t)-u^{(n)}(t)\|_{\Y_p}\geq \g\big)
&\leq  \P(\varphi_0\wedge \varphi_0^{(n)}\leq t)+I_{0,T}^{\varphi}
\\ &\leq  \P(\sigma_0\wedge \sigma_0^{(n)}\leq t) + \P(\tau_0\wedge \tau_0^{(n)}\leq t) +I_{0,t}^{\sigma}+I_{t,T}^{\tau}
\\ & \leq C_2(1+\gamma^{-p})\E\|u_0-u_0^{(n)}\|_{\Y_p}^p+C_2 \E\|u_0-u_{0}^{(n)}\|_{\Y_r}^r
\\  &  +
 \P(\sigma_1\leq t)+ \P(\tau_1\leq t) + I_{t,T}^{\tau}.
\end{align*}

{\em Step 3: Estimating $I_{t,T}^{\tau}$.}
Let $\overline{u}$ and $\overline{u}^{(n)}$ be the global $(q,\a_{\crit},\s,q)$-solution to \eqref{eq:reaction_diffusion_system_continuity} with initial data $\one_{\{\tau_0\wedge \tau_0^{(n)}>t\}} u(t)$ and $\one_{\{\tau_0\wedge \tau_0^{(n)}>t\}} u^{(n)}(t)$ at time $t$, respectively.
By localization (see \cite[Theorem 4.7(4)]{AV19_QSEE_1})
\begin{equation*}
u=\overline{u} \ \text{ and } \  u^{(n)}=\overline{u}^{(n)}\ \ \text{ both a.e.\ on }[t,\infty)\times \{\tau_0\wedge \tau_0^{(n)}>t\}.
\end{equation*}
Recall also that $u(s),u^{(n)}(s)\in C^{\theta}$ a.s.\ for all $\theta\in (0,1)$ and $s>0$ by Theorem \ref{t:reaction_diffusion_global_critical_spaces}. Let $C$ be the embedding constant of $L^{q_0}\hookrightarrow \Y_p$ (here we used that $p\geq q_0$).
Then we can write, for all $R>0$,
\begin{align*}
&I_{t,T}^{\tau}
\leq \P\Big(\sup_{s\in [t,T]}\|\overline{u}(s)-\overline{u}^{(n)}(s)\|_{L^{q_0}}\geq \frac{\g}{C}\Big)\\
&\leq  \frac{2e^{R} C^{q_0}}{\g^{q_0}}\E\Big[\one_{\{\tau_0\wedge \tau_0^{(n)}>t\}}\|u(t)-u^{(n)}(t)\|_{L^{q_0}}^{q_0}\Big]\\
&\qquad + \wt{\psi}_t(R)\Big(1+\E\Big[\one_{\{\tau_0\wedge \tau_0^{(n)}>t\}} \|u(t)\|_{L^{\zeta_0}}^{\zeta_0}\Big]
+\E\Big[\one_{\{\tau_0\wedge \tau_0^{(n)}>t\}} \|u^{(n)}(t)\|_{L^{\zeta_0}}^{\zeta_0}\Big]\Big) & \text{(by \eqref{eq:tail_estimateFeller})} \\
&\leq  \frac{2e^{R} C^{q_0}}{\g^{q_0}}\Big[\E\big(\one_{\{\tau_0\wedge \tau_0^{(n)}>t\}}\|u(t)-u^{(n)}(t)\|_{L^{q_0}}^{r}\big)\Big]^{\frac{q_0}{r}}+ \Psi_t(R)(1+2C_1^{\frac{\zeta_0}{r}}) & \text{(by \eqref{eq:proof_continuity_K_estimate}-\eqref{eq:local_continuity_r_setting_1})} \\
&\leq
 \frac{2e^{R}C^{q_0}}{\g^{q_0} t^{(q_0\mu)/r}} \Big[\E\big[\one_{\{\tau_0\wedge \tau_0^{(n)}>t\}}\sup_{s\in (0,t]}\big(s^{\mu}  \|u(s)-u^{(n)}(s)\|_{L^{q_0}}^r\big)\big]\Big]^{\frac{q_0}{r}}+ \Psi_t(R)(1+2C_1^{\frac{\zeta_0}{r}})\\
 &\leq \frac{2e^{R} C_1^{q_0/r} C^{q_0}}{\g^{q_0} t^{(q_0\mu)/r}} \big[\E\|u_0-u_0^{(n)}\|_{\Y_r}^r\big]^{\frac{q_0}{r}} + \Psi_t(R)(1+2C_1^{\frac{\zeta_0}{r}}) & \text{(by \eqref{eq:local_continuity_r_setting_1}),}
\end{align*}
where $\Psi_t (R):=t^{-(\zeta_0\mu)/r} \wt{\psi}_t(R)$ for $R,t>0$.

{\em Step 3: Conclusion.}  Collecting the previous estimates and letting $n\to \infty$ we obtain
\begin{align*}
\limsup_{n\to \infty}
\P\Big(\sup_{t\in [0,T]}\|u(t)-u^{(n)}(t)\|_{\Y_p}\geq \g\Big)\leq \P(\sigma_1\leq t)+ \P(\tau_1\leq t)
+ \Psi_t(R)(1+2C_1^{\frac{\zeta_0}{r}}) ,
\end{align*}
Now first using $\lim_{R\uparrow \infty}\wt{\psi}_t(R)=0$, and then $\lim_{t\downarrow 0}\P(\sigma_1\leq t) = \lim_{t\downarrow 0}\P(\tau_1\leq t)=0$, we obtain the desired convergence in probability.
\end{proof}

\begin{proof}[Proof of Proposition \ref{prop:L_m_continuity}]

{\em Step 1: Continuity of $u$.}
We first prove that $u\in L^0(\O;C([0,T];L^{q}))$. By localization (see \cite[Theorem 4.7(4)]{AV19_QSEE_1}) we may suppose that $u_0\in L^q(\Omega;L^q)$.
Let $(u_0^{(n)})_{n\geq 1}$ be such that $u_0^{(n)}\to u_0$ in $L^{\m}(\O\times \Tor^d;\R^{\ell})$ and $u_0^{(n)}\in L^{\m}_{\F_0}(\O;C^1)$ for all $n\geq 1$. Let $u^{(n)}$ be the solution corresponding to initial value $u^{(n)}_0$.

By Assumption \ref{ass:global_s}$(0,\zeta)$ and Lemma \ref{l:tail_estimate}, we find that
\begin{align*}
\lim_{m,n\to\infty} \P\Big(\sup_{t\in [0,T]}\|u^{(n)}(t)-u^{(m)}(t)\|_{L^{q}}\geq \g\Big)
 \lesssim_T
 \psi_0(R) (1+2\E\|u_0\|_{L^{\m}}^{\m}).
\end{align*}
Letting $R\to \infty$, it follows that $(u^{(n)})_{n\geq 1}$ is a Cauchy sequence in $L^0(\O;C([0,T];L^{q}))$ and therefore it converges to some $\wt{u}\in L^0(\O;C([0,T];L^q))$. Since $u^{(n)}\to u$ in $L^0(\O;C([0,T];B^0_{q,p}))$ by Theorem \ref{t:continuity_general}, it follows that $\wt{u}=u$. This proves the required continuity.

{\em Step 2: For all $u_0,v_0\in L^q_{\F_0}(\O;L^q(\Tor^d;\R^{\ell}))$ and $T,R>0$ the following estimate holds:}
\begin{align}\label{eq:tail_estimateFeller2}
\P\Big(\sup_{t\in [0,T]}\|u(t)-v(t)\|_{L^{q}}\geq \g\Big) \lesssim_T
\frac{C_R}{\g^{q}}\E\|u_0-v_0\|_{L^{q}}^{q}+ \psi_0(R) (1+\E\|u_0\|_{L^{\m}}^{\m}+\E\|v_0\|_{L^{\m}}^{\m}),
\end{align}
{\em where $\psi_0$ is as in Assumption \ref{ass:global_s}$(0,\zeta)$.}

The estimate \eqref{eq:tail_estimateFeller2} is an extension of part of the estimate \eqref{eq:tail_estimateFeller}. To prove it we need to get rid of the H\"older continuity assumption via an approximation argument. Let $(u^{(n)}_0)_{n\geq 1}$ and $(v^{(n)}_0)_{n\geq 1}$ be such that
$u^{(n)}_0\to u_0$ and $v^{(n)}_0\to v_0$ in $L^q(\O;L^q(\Tor^d;\R^{\ell}))$ and $u^{(n)}_0,v^{(n)}_0\in L^q_{\F_0}(\Omega;C^1)$. Let $u^{(n)}$ and $v^{(n)}$ be the solutions corresponding to the initial data $u^{(n)}_0$ and $v^{(n)}_0$.

By Lemma \ref{l:tail_estimate}, estimate \eqref{eq:tail_estimateFeller2} holds for $s=0$ with $(u,v)$ replaced by $(u^{(n)},v^{(n)})$. From the proof of Step 1 we obtain that $u^{(n)}\to u$ and $v^{(n)}\to v$ in $L^0(\O;C([0,T];L^q))$. Therefore, \eqref{eq:tail_estimateFeller2} follows by taking limits $n\to \infty$.

{\em Step 3: The convergence assertion.} As in Step 1 of Theorem \ref{t:continuity_general} we may assume $u_0$ and $u_0^{(n)}$ in $L^p(\O\times \Tor^d;\R^{\ell})$. To prove the convergence assertion we can argue in a similar way as in Step 1, but now using \eqref{eq:tail_estimateFeller2}.
\end{proof}

It remains to prove Corollary \ref{cor:estimate_up_to_0}.

\begin{proof}[Proof of Corollary \ref{cor:estimate_up_to_0}]
We only prove \eqref{it:estimate_up_to_0_Lzeta} as the other follows similarly.
More precisely, we only prove the following estimate:
\begin{equation}
\label{eq:proof_corollary_estimate_up_to_zero}
\sup_{t\in [0,T]}\E \|u(t)\|_{L^{\m}}^{\m }+  \E \int_{0}^{T}\int_{\Tor^d} |u|^{\m-2} |\nabla u|^{2}\,\dd x\,\dd t
\leq N_0(1+\E\|u_0\|_{L^{\m}}^{\m}).
\end{equation}
The estimate \eqref{eq:aprioribounds2} follows similarly.
Recall that $\ell=i=1$ in the setting of \eqref{it:estimate_up_to_0_Lzeta}.

Fix $u_0\in L^{\m}_{\F_0}(\O\times \Tor^d)$.
Choose $(u_{0}^{(n)})_{n\geq 1}\subseteq L^{\m}_{\F_0}(\O;C^{1}(\Tor^d))$ such that $u_0^{(n)}\to u_0$ in $L^{\m}(\O\times \Tor^d)$. Let $u^{(n)}$  be the global $(p,\a_{\crit},\s,q)$-solution to \eqref{eq:reaction_diffusion_system_continuity} with $i=\ell=1$.
By the last assertion of Theorem \ref{t:global_well_posedness_general}, one sees that
\eqref{eq:proof_corollary_estimate_up_to_zero} holds with $(u,u_0)$ replaced by $(u^{(n)},u_0^{(n)})$.
From Proposition \ref{prop:L_m_continuity}, we see that up to extracting a subsequence, we have $u^{(n)}\to u$ a.s.\ in $C([0,T];L^q(\Tor^d))$, and in particular, in $L^0(\Omega\times(0,T)\times \T^d)$.

Since $\big|\nabla |u^{(n)}|^{\m/2}\big|^2= \frac{\m^2}{4} |u^{(n)}|^{\m-2}|\nabla u^{(n)}|^2$, by the above we see that $v_n = |u_{n}|^{\m/2}$ is a bounded sequence in $L^2(\Omega\times(0,T)\times \T^d)$. Therefore, it has a subsequence such that $v_{n_k}\to v$ weakly in $L^2(\Omega\times(0,T)\times \T^d)$. Since we already know that $u^{(n)}\to u$ in measure, it follows that $v=|u|^{\m/2}$. By lower semi-continuity of the $L^2$-norm in the topology of weak convergence, we see that
\begin{align*}
\E \int_{0}^{T}\int_{\Tor^d} \big|\nabla |u|^{\m/2}\big|^2\,\dd x\,\dd t
&\leq \liminf_{k\to \infty}\E \int_{0}^{T}\int_{\Tor^d} \big|\nabla |u^{(n_k)}|^{\m/2}\big|^2 \,\dd x\,\dd t
\\ & \leq \liminf_{k\to\infty} N_0(1+\E\|u_0^{(n_k)}\|_{L^{\m}}^{\m})
 =N_0(1+\E\|u_0\|_{L^{\m}}^{\m}).
\end{align*}
This gives the required estimate for the second term on the LHS of \eqref{eq:proof_corollary_estimate_up_to_zero}.  For the proof for the first term fix $t\in [0,T]$. Since $(u^{(n)}(t))_{n\geq 1}$ is bounded in $L^{\zeta}(\Omega\times \T^d)$ is has a subsequence such that $u^{(n)}(t)\to \xi$ weakly $L^{\zeta}(\Omega\times \T^d)$. Since $u^{(n)}(t)\to u(t)$ in $L^0(\Omega;L^q(\T^d))$, it follows that $u(t) = \xi$. Therefore,
as before
\begin{align*}
\E \|u(t)\|_{L^{\m}}^{\m } \leq \E \liminf_{k\to \infty}\|u^{(n_k)}(t)\|_{L^{\m}}^{\m } \leq \liminf_{k\to \infty} N_0(1+\E\|u_0^{(n_k)}\|_{L^{\m}}^{\m}) =
N_0(1+\E\|u_0\|_{L^{\m}}^{\m}).
\end{align*}
\end{proof}

\appendix

\section{Generalized It\^o's formula}
The next version of It\^o's formula will be applied several times. The proof is analogous to the one of  \cite[Proposition A.1]{DHV16} and it is omitted for brevity.

\begin{lemma}[Generalized It\^o's formula]
\label{lem:Ito_generalized}
Let $\xi:\R\to \R$ be a $C^2$-function with bounded derivatives up to the second order. Fix $0<T<\infty$ and let $\tau$ be a stopping time with values in $[0,T]$.
Assume that $v_{0}:\O\to L^2(\Tor^d)$ is $\F_{0}$-measurable and
\begin{itemize}
\item $\Phi:=(\Phi_j)_{j=1}^d\in L^2(0,T;L^2(\Tor^d;\R^d))$ and $\phi\in L^1(0,T;L^2(\Tor^d))$;
\item $(\psi_n)_{n\geq 1}\in L^2(0,T;L^2(\Tor^d;\ell^2))$ a.s.;
\item $\phi,\Phi$ and $\psi_n$ are progressively measurable.
\end{itemize}
Suppose that, a.s., $v\in C([0,T];L^2(\Tor^d))$, $v\in L^2(0,\tau;H^1(\Tor^d))$ and $v$ satisfies
\begin{equation*}
\dd v=\phi(t)\,\dd t +\one_{[0,\tau]}\div (\Phi(t))\,\dd t +\sum_{n\geq 1}  \psi_n(t)\,\dd w^n_t\ \ \text{ on }[0,T]\times \O, \quad  v(0)=v_{0},
\end{equation*}
both in $H^{-1}(\T^d)$. Then, a.s.\ for all $t\in [0,T]$,
\begin{align*}
\int_{\Tor^d} \xi(v(t))\,\dd x
&= \int_{\Tor^d} \xi(v_{0})\,\dd x+ \int_{0}^t \int_{\Tor^d} \xi'(v(s)) \phi(s)\,\dd x\,\dd s\\
& - \int_{0}^t \int_{\Tor^d}\one_{[0,\tau]} \xi''(v(s)) \nabla v(s)\cdot \Phi(s)\,\dd x\,\dd s\\
& + \sum_{n\geq 1}\int_{0}^t \int_{\Tor^d} \xi'(v(s)) \psi_n \,\dd x\,\dd w^n_s
+\frac{1}{2}\sum_{n\geq 1} \int_{0}^t \int_{\Tor^d} \xi''(v(s)) |\psi_n(s)|^2\,\dd x\,\dd s.
\end{align*}
\end{lemma}

Note that $\nabla v\cdot \Phi\in L^1((0,\tau)\times \T^d)$ a.s.\ due to the assumptions of Lemma \ref{lem:Ito_generalized}.

\section{Proof of Lemma \ref{lem:dissipationII}}
\label{app:proofs_lemmas}

\begin{proof}[Proof of  Lemma \ref{lem:dissipationII}]
Since all arguments are pointwise in $(t,\omega)$ we omit it from the notation for convenience and just write $\cdot$ instead.

{\em Step 1: Estimate for $F$.}
Set $\mathcal{F}(\cdot,x,u):=\int_0^{u} |y|^{\m-2} F(\cdot,x,y)\,\dd y$ for $u\in\R$.
Since $(x,y)\mapsto F(\cdot,x,y)$ and $(x,y)\mapsto \partial_{x_j} F(\cdot,x,y)$ are continuous, by Leibniz integral rule we have, for $u\in C^1(\T^d)$,
$$
\div_x\big[ \mathcal{F}(\cdot,x,u(x))\big]
= |u|^{\m-2} F(\cdot,x,u(x))\cdot \nabla u
+ \int_{0}^{u(x)}|y|^{\m-2} \, \div_x F(\cdot,x,y)\,\dd y.
$$
Since $F$ and $u$ are periodic the integral over $\T^d$ of the LHS vanishes, and therefore
\begin{align*}
\Big|\int_{\T^d} |u(x)|^{\m-2}  &F(\cdot,x, u(x)) \cdot \nabla u \, \dd x\Big| =
\Big|\int_{\T^d} \int_{0}^{u(x)}|y|^{\m-2} \, \div_x F(\cdot,x,y)\,\dd y \, \dd x\Big|
\\ & \leq \frac{1}{\m-1}\int_{\T^d} |u(x)|^{\m-1} \sup_{|y|\leq |u(x)|}|\div_x F(\cdot,x,y)| \, \dd x
\\ & \leq \varepsilon \int_{\T^d} |u(x)|^{\m-2} \sup_{|y|\leq |u(x)|}|\div_x F(\cdot,x,y)|^2 \, \dd x
+ C_{\varepsilon} \int_{\T^d} |u(x)|^{\m} \, \dd x,
\end{align*}
where we used Young's inequality with $\varepsilon>0$ arbitrary and $C_{\varepsilon}>0$.

{\em Step 2: Estimate for the product of $b$ and $g$.}
Next, we estimate the product term involving $b_n$ and $g_n$ using a similar idea.
Set
$\mathcal{G}_n(\cdot,x,u):= \int_0^u |y|^{\m-2} g_n(\cdot,x,y)\,\dd y$ for $u\in\R$.
Then Leibniz integral rule gives that
$$
\partial_{x_j}\big[\mathcal{G}_n(\cdot,x,u(x))\big]
= |u(x)|^{\m-2} g_n(\cdot,x,u(x))\partial_j u(x) + \int_0^{u(t,x)} |y|^{\m-2}\partial_{x_j} g_n(\cdot,x,y)\,\dd y.
$$
Integrating by parts and arguing as before using the periodicity of $b$, $\mathcal{G}_n$ and $u$ we find that
\begin{align*}
-&\int_{\T^d} |u(x)|^{\m-2}\big[ (b_n(\cdot,x) \cdot \nabla) u(x) \, g_n(\cdot,x, u(x))\big] \, \dd x\\
& \quad =   \int_{\T^d} \div(b_n(\cdot, x)) \mathcal{G}_n(\cdot, x,u(x)) \,\dd x + \int_{\T^d} b_n(\cdot,x)  \int_0^{u(t,x)} |y|^{\m-2}\nabla_x g_n(\cdot,x,y)\,\dd y\, \dd x.
\end{align*}
Thus
\begin{align*}
 &\Big|\int_{\T^d} |u(x)|^{\m-2} \sum_{n\geq 1} \big[(b_n(\cdot,x) \cdot \nabla) u(x)\,  g_n(\cdot,x, u(x)) \big]\, \dd x\Big|
\\ & \leq \int_{\T^d} \int_{-|u(x)|}^{|u(x)|} |y|^{\m-2} \sum_{n\geq 1} |\div(b_n)(\cdot, x)|  |g_n(\cdot, x,y)| \, \dd y  \, \dd x
\\ & \quad + \int_{\T^d} \int_{-|u(x)|}^{|u(x)|} |y|^{\m-2} \sum_{n\geq 1} |b_n(\cdot,x)| \, |\nabla_x g_n(\cdot, x,y)| \,\dd y \, \dd x
\\ & \leq \|(\div(b_n))_{n\geq 1}\|_{L^\infty(\T^d;\ell^2)} \int_{\T^d}  \int_{-|u(x)|}^{|u(x)|} |y|^{\m-2}  \|(g_n(\cdot, x,y))_{n\geq 1}\|_{\ell^2} \, \dd y  \, \dd x
\\ & \quad + \|(b_n)_{n\geq 1}\|_{L^\infty(\T^d;\ell^2)} \int_{\T^d} \int_{-|u(x)|}^{|u(x)|} |y|^{\m-2} \, \|(\nabla_x g_n(\cdot, x,y))_{n\geq 1}\|_{\ell^2} \,\dd y \, \dd x.
\end{align*}
To estimate the last two terms for $k\in \{0,1\}$ and any $\delta>0$ we can write
\begin{align*}
&\int_{\T^d}  \int_{-|u(x)|}^{|u(x)|} |y|^{\m-2}  \|(\nabla_x^k g_n(\cdot, x,y))_{n\geq 1}\|_{\ell^2} \, \dd y  \, \dd x
\\ & \leq \frac{2}{\m-1}\int_{\T^d}  |u(x)|^{\m-1}  \sup_{|y|\leq |u(x)|} \|(\nabla_x^k g_n(\cdot, x,y))_{n\geq 1}\|_{\ell^2} \, \dd x
\\ & \leq \delta\int_{\T^d}  |u(x)|^{\m-2}  \sup_{|y|\leq |u(x)|} \|(\nabla_x^k g_n(\cdot, x,y))_{n\geq 1}\|_{\ell^2}^2 \, \dd x  + C_{\delta} \int_{\T^d}  |u(x)|^{\m} \, \dd x,
\end{align*}
where $C_{\delta}>0$ depends on $\delta$. Therefore, multiplying by the $L^\infty$-norm of the $b$ term or its derivative, we can conclude that for every $\varepsilon>0$
\begin{align*}
&\Big|\int_{\T^d} |u(x)|^{\m-2} \sum_{n\geq 1} (b_n(\cdot,x) \cdot \nabla) u  g_n(\cdot,x, u(x)) \, \dd x\Big|
\\ & \leq \varepsilon\sum_{k\in \{0,1\}}\int_{\T^d}  |u(x)|^{\m-2}  \sup_{|y|\leq |u(x)|} \|(\nabla_x^k g_n(\cdot, x,y))_{n\geq 1}\|_{\ell^2}^2 \, \dd x  + C_{b,\varepsilon} \int_{\T^d}  |u(x)|^{\m} \, \dd x.
\end{align*}

{\em Step 3: Conclusion.}
Therefore, we find that
\begin{align*}
& \am \nabla u\cdot \nabla u   + F(\cdot, u) \cdot \nabla u - \frac{u f(\cdot, u)}{\m-1} -\frac12 \sum_{n\geq 1} \big[(b_n \cdot \nabla) u + g_n(\cdot, u) \big]^2
\\ & \geq
\am \nabla u\cdot \nabla u  -\frac{1}{2} \sum_{n\geq 1} |(b_n\cdot \nabla)u|^2  - \frac{u f(\cdot, u)}{\m-1} - \frac{1}{2} \sum_{n\geq 1} |g_n\cdot,u)|^2-R(\cdot, u)
\\ & \geq \nu |\nabla u|^2 - \frac{u f(\cdot, u)}{\m-1} - \frac12\|(g_n(\cdot,u)_{n\geq 1}\|^2_{\ell^2} -R(\cdot,u),
\end{align*}
where $R(\cdot, u) = -F(\cdot, u) \cdot \nabla u +\sum_{n\geq 1} [(b_n \cdot \nabla) u \, g_n(\cdot, u)]$. Therefore, to check Assumption \ref{ass:dissipation_general} we multiply by $|u(x)|^{\m-2}$, and integrate over $\T^d$ to obtain
\begin{align*}
 \int_{\T^d} |u|^{\m-2} \Big( \am \nabla u\cdot \nabla u   + F(\cdot, u) \cdot \nabla u - \frac{u f(\cdot, u)}{\m-1} -\frac12 \sum_{n\geq 1} \big[(b_n \cdot \nabla) u + g_n(\cdot, u) \big]^2\Big) \, \dd x&
\\  \geq \int_{\T^d} |u|^{\m-2}\Big( \nu |\nabla u|^2 - \frac{u f(\cdot, u)}{\m-1} - \frac{1}{2} \|(g_n(\cdot,u)_{n\geq 1}\|^2_{\ell^2} -R(\cdot,u)\Big) \,\dd x&.
\end{align*}
Now note that for every $\varepsilon>0$
\begin{align*}
\int_{\T^d} |u|^{\m-2} R(\cdot,u) \,\dd x
& \leq C_{b,\varepsilon}' \int_{\T^d} |u(x)|^{\m} \, \dd x + \varepsilon \int_{\T^d} |u(x)|^{\m-2} \sup_{|y|\leq |u(x)|}|\div_x F(\cdot,x,y)|^2 \, \dd x
\\ & \ \ + \varepsilon\sum_{k\in \{0,1\}}\int_{\T^d}  |u(x)|^{\m-2}  \sup_{|y|\leq |u(x)|} \|(\nabla_x^k g_n(\cdot, x,y))_{n\geq 1}\|_{\ell^2}^2 \, \dd x.
\end{align*}
Using this together with the pointwise condition \eqref{eq:dissipationII}, Assumption \ref{ass:dissipation_general} follows.

The final assertion concerning $S$-valued functions follows in the same way.
\end{proof}

{\small
\subsubsection*{Acknowledgements}
The authors are grateful to Caterina Balzotti for her support in creating the figures.
The authors thank Udo B\"ohm and the anonymous referees for careful reading and helpful comments.
}

\bibliographystyle{plain}
\bibliography{literature}

\end{document}